\documentclass[12pt,a4paper]{amsart}
\usepackage{color,graphicx,fullpage,amssymb}
\usepackage{url}
\usepackage[colorlinks=true]{hyperref}
\usepackage{etoolbox}
\usepackage[colorinlistoftodos]{todonotes}

\newtheorem{theorem}{Theorem}[section]
\newtheorem{lemma}[theorem]{Lemma}
\newtheorem{proposition}[theorem]{Proposition}
\newtheorem{corollary}[theorem]{Corollary}

\newtheorem{maintheorem}{Theorem}

\theoremstyle{definition}
\newtheorem{definition}[theorem]{Definition}
\newtheorem{remark}[theorem]{Remark}
\newtheorem{example}[theorem]{Example}

\newcommand{\Q}{\mathbb{Q}}
\newcommand{\Qp}{\mathbb{Q}_p}

\newcommand{\Zp}{\mathbb{Z}_p}
\newcommand{\N}{\mathbb{N}}
\newcommand{\R}{\mathbb{R}}
\newcommand{\Z}{\mathbb{Z}}

\newcommand{\F}{\mathbb{F}}
\newcommand{\ball}{\mathrm{B}}
\newcommand{\cyl}{\mathrm{Z}}
\newcommand{\dd}{\mathrm{d}}
\newcommand{\ii}{\mathrm{i}}
\newcommand{\e}{\mathrm{e}}

\newcommand{\tr}{^{\mathrm{T}}}
\newcommand{\linearwidth}{{\mathrm{w_L}}}

\newcommand{\Circle}{\mathrm{S}^1_p}
\newcommand{\letnpos}{Let $n$ be a positive integer}
\newcommand{\letpprime}{Let $p$ be a prime number}

\newcommand{\M}{\mathcal{M}}

\renewcommand{\le}{\leqslant}
\renewcommand{\ge}{\geqslant}

\DeclareMathOperator{\ord}{ord}

\DeclareMathOperator{\ASp}{ASp}
\numberwithin{equation}{section}

\title{Rigidity and flexibility in $p$-adic symplectic geometry}
\author[Luis Crespo, \'Alvaro Pelayo]{Luis Crespo\,\,\,\,\,\, \'Alvaro Pelayo}

\address{Luis Crespo,
	Departamento de Matem\'{a}ticas, Estad\'{i}stica y Computaci\'{o}n, Universidad de Cantabria, Av.~de Los Castros 48, 39005 Santander, Spain}
\email{luis.cresporuiz@unican.es}

\address{\'Alvaro Pelayo,
	Facultad de Ciencias Matem\'aticas,
	Universidad Complutense de Madrid, 28040 Madrid, Spain, and Real Academia de Ciencias Exactas, F\'isicas y Naturales, Madrid, Spain}
\email{alvpel01@ucm.es}

\begin{document}
	
\begin{abstract}
	Let $n\ge 2$ be an integer and let $p$ be a prime number. We prove that the analog of Gromov's non-squeezing theorem does not hold for $p$-adic embeddings: for any $p$-adic absolute value $R$, the entire $p$-adic space $(\Qp)^{2n}$ is symplectomorphic to the $p$-adic cylinder $\cyl_p^{2n}(R)$ of radius $R$, showing a degree of flexibility which stands in contrast with the real case.
	However, some rigidity remains: we prove that the $p$-adic affine analog of Gromov's result still holds. We will also show that in the non-linear situation, if the $p$-adic embeddings are equivariant with respect to a torus action, then non-squeezing holds, which generalizes a recent result by Figalli, Palmer and the second author. This allows us to introduce equivariant $p$-adic analytic symplectic capacities, of which the $p$-adic equivariant Gromov width is an example.
\end{abstract}

\maketitle

\section{Introduction}

Let $n\ge 2$ be an integer. Let $(x_1,y_1,\ldots,x_n,y_n)$ be the standard coordinates on $\R^{2n}$. Let us consider the $2n$-dimensional ball $\ball^{2n}(r)$ of radius $r>0$ defined by the inequality $\sum_{i=1}^n x_i^2+y_i^2<r^2$ and the $2n$-dimensional cylinder $\cyl^{2n}(R)=\ball^2(R)\times\R^{2n-2}$ of radius $R>0$. Endow both the ball $\ball^{2n}(r)$ and the cylinder $\cyl^{2n}(R)$ with the standard symplectic form $\sum_{i=1}^n\dd x_i\wedge\dd y_i$ on $\R^{2n}$. \emph{Gromov's non-squeezing theorem} \cite{Gromov} states that there exists a symplectic embedding
$f:\ball^{2n}(r)\hookrightarrow\cyl^{2n}(R)$
if and only if $r\le R$. However, since $\cyl^{2n}(R)$ has infinite volume with respect to the volume form $\dd x_1\wedge\dd y_1\wedge\ldots\wedge\dd x_n\wedge\dd y_n$, even if $r>R$ one can always find volume-preserving embeddings from $\ball^{2n}(r)$ to $\cyl^{2n}(R)$, with respect to this volume form. This is a \emph{rigidity} result, it shows that being ``symplectic'' is more rigid than being ``volume-preserving''. See Figure \ref{fig:squeezing-real} for an illustration. Gromov's theorem and the techniques Gromov used to prove it in his pioneering paper \cite{Gromov} have played an important role in Hamiltonian and symplectic dynamics, see for example the classical book by Hofer-Zehnder \cite{HofZeh} and the article by Bramham-Hofer \cite{BraHof}. In the present paper we explore Gromov's result from the angle of $p$-adic analytic geometry.

Next we state the main results of the paper. In later sections we will
prove slightly more general versions of these results, but which are also more technical.

\subsection{$p$-adic linear non-squeezing and $p$-adic non-linear squeezing}

Let $p$ be a prime number. We
endow the $2n$-dimensional $p$-adic space $(\Qp)^{2n}$ with the standard $p$-adic symplectic form
$\omega_0=\sum_{i=1}^n \dd x_i\wedge \dd y_i$ on $(\Qp)^{2n}$, where $(x_1,y_1,\ldots,x_n,y_n)$ are the standard coordinates on $(\Qp)^{2n}$. Also, for any $p$-adic absolute values $r,R$ let $\ball_p^{2n}(r)$ and $\cyl_p^{2n}(R)$ respectively denote the $2n$-dimensional $p$-adic ball and cylinder endowed with $\omega_0$ (see expression \eqref{eq:cylinder} for the explicit definition of $\cyl_p^{2n}(R)$, which can be given by direct analogy with the real case). We start with the $p$-adic affine version of Gromov's non-squeezing theorem. The $p$-adic analogue of the situation in the real case still holds for $p$-adic analytic embeddings, which shows the \emph{linear rigidity} of $p$-adic symplectic geometry:

\begin{maintheorem}\label{thm:linear1}
	Let $n$ be an integer with $n\ge 2$. Let $p$ be a prime number. Let $r$ be a $p$-adic absolute value. There exists a $p$-adic affine symplectic embedding $f : \ball_p^{2n}(r) \hookrightarrow \cyl_p^{2n} (1)$ if and only if $r \le 1$.
\end{maintheorem}

We refer to Theorem \ref{thm:linear} for a more general, but equivalent, version of Theorem \ref{thm:linear1}. In addition, also concerning the linear/affine situation, we will prove several $p$-adic analogs of classical results in real symplectic topology such as a characterization of the linear squeezing property (Theorem \ref{thm:rigidity}) or a characterization of matrices which preserve the $p$-adic linear symplectic width (Theorem \ref{thm:width}).

In the nonlinear situation then there is no analog of Theorem \ref{thm:linear1}. On the contrary, the following nonlinear statement stands in strong contrast with Theorem \ref{thm:linear1}:

\begin{maintheorem}\label{thm:total-embedding1}
	Let $n$ be an integer with $n\ge 2$. Let $p$ be a prime number. There exists a $p$-adic analytic symplectomorphism \[\phi : (\Qp)^{2n}\overset{\cong}{\longrightarrow} \cyl_p^{2n} (1).\] In particular, $\phi|_{\ball_p^{2n}(r)}:\ball_p^{2n}(r)\hookrightarrow\cyl_p^{2n}(1)$ is a $p$-adic analytic symplectic embedding for every $p$-adic absolute value $r$.
\end{maintheorem}

Theorem \ref{thm:total-embedding1} may be considered a \emph{flexibility} result:
the entire $p$-adic space $(\Qp)^{2n}$ can be symplectically squeezed into a thin cylinder.
We refer to Theorem \ref{thm:total-embedding} for an extended version of Theorem \ref{thm:total-embedding1}, which in particular includes
detailed information about the properties of the $p$-adic analytic symplectomorphism $\phi$.

\subsection{$p$-adic squeezing and non-squeezing while preserving symmetries}

Next we state a result which shows how subtle the phenomena of $p$-adic symplectic squeezing is, in the sense that if we want to squeeze while preserving symmetries then it may or not be possible, depending on ``how much'' symmetry we require.  These symmetries are expressed by means of invariance under natural group actions. Let $\Circle$ be the $p$-adic circle. Consider the standard rotational action of the $n$-dimensional torus $(\Circle)^n$ on $(\Qp)^{2n}$ component by component:
\[((a_1,b_1),\ldots,(a_n,b_n))\cdot(x_1,y_1,\ldots,x_n,y_n)=((a_1,b_1)\cdot(x_1,y_1),\ldots,(a_n,b_n)\cdot(x_n,y_n)),\]
where the action of $\Circle$ on $(\Qp)^2$ is given by
\begin{equation}\label{eq:action}
	(a,b)\cdot(x,y)=(ax-by,bx+ay).
\end{equation}

The action given in \eqref{eq:action} is well defined on $(\Qp)^{2n}$ but is not well defined when restricted to $\ball_p^{2n}(r)$ because it does not leave $\ball_p^2(r)$ invariant. But this technical problem can be dealt with by instead considering the largest subgroup $\mathrm{G}_p$ of $\Circle$ which leaves $\ball_p^2(r)$ invariant, which results in a well defined action of $(\mathrm{G}_p)^n$ on $\ball_p^{2n}(r)$. In fact,  $\mathrm{G}_p$ can be explicitly described as
follows: $\mathrm{G}_p=\Circle$, if $p\not\equiv 1\mod 4$; if on the other hand $p \equiv 1 \mod 4$, then $\mathrm{G}_p$ consists exactly of the elements $(a,b)$ of $\Circle$ such that $\ord_p(a + \ii b) = 0$,
where $\ii$ is any number in $\Qp$ such that $\ii^2 = -1$. Finally, let us consider the open subset of $(\Qp)^{2n}$ given by \[\mathrm{T}_p^{2n}=\Big\{(x,y)\in(\Qp)^2:(x,y)\ne(0,0)\Big\}^n.\] The following result displays both rigidity and flexibility aspects of $p$-adic symplectic geometry under symmetries.

\begin{maintheorem}\label{thm:equivariant}
	Let $n$ be an integer with $n\ge 2$. Let $p$ be a prime number. There is a $(\mathrm{G}_p \times (\Circle)^{n-1})$-equivariant $p$-adic analytic
	symplectomorphism \[\phi : \mathrm{T}_p^{2n}\overset{\cong}{\longrightarrow} \mathrm{T}_p^{2n}\cap\cyl_p^{2n} (1).\]
	However, for any $p$-adic absolute value $r>1$ there is no $(\mathrm{G}_p)^n$-equivariant $p$-adic analytic symplectic embedding
	$f : \ball_p^{2n} (r) \hookrightarrow \cyl_p^{2n} (1)$.
\end{maintheorem}

Theorem \ref{thm:equivariant} is implied by the slightly more general statements Theorems \ref{thm:analytic-actions-complete} and \ref{thm:total-actions-nofixed}.  We will also prove other results along the lines of Theorem \ref{thm:equivariant} (such as Theorem \ref{thm:total-actions-semi}).

We conclude with the following result concerning symplectic capacities.

\begin{maintheorem}\label{thm:capacity1}
	Let $n$ be an integer with $n\ge 2$. Let $p$ be a prime number. There is no $p$-adic analytic symplectic capacity on $(\Qp)^{2n}$ but there exists a $p$-adic analytic symplectic $(\mathrm{G}_p)^n$-capacity on $(\Qp)^{2n}$.
\end{maintheorem}

Theorem \ref{thm:capacity1} is implied by Theorems \ref{thm:no-capacity} and \ref{thm:capacity}.

\subsection{Works on $p$-adic geometry and symplectic geometry}

The literature on $p$-adic geometry is very extensive, as it is also
the case in symplectic geometry and topology. We recommend the articles \cite{Eliashberg,Pelayo-hamiltonian,Schlenk,Weinstein-symplectic} for surveys on different aspects of symplectic geometry and topology, and \cite{HofZeh,MarRat,McDSal,OrtRat} for books in symplectic geometry and its connection to mechanics. The field of $p$-adic geometry is large and broad, see Lurie's lecture \cite{Lurie}, the book by Scholze-Weinstein \cite{SchWei} and the references therein for a recent treatment. See also Zelenov \cite{Zelenov} and Hu-Hu \cite{HuHu} for a construction of symplectic $p$-adic vector spaces and the $p$-adic Heisenberg group and Maslov index.

The present paper is the fourth of a sequence initiated ten years ago by Voevodsky, Warren the second author \cite[Section 7]{PVW} with the goal of developing symplectic geometry and integrable systems with $p$-adic coefficients, and eventually implementing it using homotopy type theory and the proof assistant Coq. In our recent papers \cite{CrePel-JC,CrePel-williamson} we carried out part of this proposal. However, the present paper is the first of the four to focus on results of flexibility/rigidity in $p$-adic symplectic geometry and topology.

\subsection*{Structure of the paper}

The paper is organized as follows. Section \ref{sec:ball-cylinder} introduces $p$-adic balls and $p$-adic cylinders; Section \ref{sec:affine} presents our first main result, about $p$-adic symplectic non-squeezing in the linear case; Section \ref{sec:nonlinear} presents the second main result, about $p$-adic symplectic squeezing in the nonlinear case; Sections \ref{sec:char-squeezing} and \ref{sec:width} continue the study of the linear case, defining the $p$-adic analogs of affine rigidity and linear symplectic width; Section \ref{sec:polar} gives a new construction of polar coordinates on the plane $(\Qp)^2$; Section \ref{sec:equivariant} uses the coordinates of Section \ref{sec:polar} in order to prove an equivariant version of symplectic squeezing in the non-linear $p$-adic case; Section \ref{sec:capacities} introduces the concept of $p$-adic analytic symplectic capacity, which is a nonlinear $p$-adic generalization of the concept of linear symplectic width; Section \ref{sec:examples} gives some examples and Section \ref{sec:remarks} some final remarks; we close the paper with an appendix (Section \ref{sec:appendix}), where we recall the basic properties of the $p$-adic numbers.

\section{$p$-adic balls and $p$-adic cylinders}\label{sec:ball-cylinder}

In this section we recall the concept of $p$-adic ball and introduce the notion of $p$-adic cylinder, both of which are needed to formulate the main results of the present paper. We refer to the Appendix (Section \ref{sec:appendix}) for a brief review of the basic concepts regarding $p$-adic numbers.

\letpprime. Recall that the \emph{$p$-adic absolute value} $|\cdot|_p$ on the field of $p$-adic numbers $\Qp$ is defined by
\[|x|_p=p^{-\ord_p(x)}.\]
Let $m$ be a positive integer. The \emph{$p$-adic norm} $\|\cdot\|_p$ on $(\Qp)^{m}$ is defined by
\[\|(v_1,\ldots,v_m)\|_p=\max\Big\{|v_1|_p,\ldots,|v_m|_p\Big\}.\]
Let $r$ be a $p$-adic absolute value, that is, an element of $\Q$ which is the absolute value of some nonzero element of $\Qp$. The \emph{$m$-dimensional $p$-adic ball of radius $r$} is the set
\begin{equation}\label{eq:ball}
	\ball_p^{m}(r)=\Big\{(v_1,v_2,\ldots,v_m)\in(\Qp)^{n}:\|(v_1,v_2,\ldots,v_m)\|_p<pr\Big\}.
\end{equation}
Let $R$ be a $p$-adic absolute value. Suppose that $m\ge 2$. The \emph{$m$-dimensional $p$-adic cylinder of radius $R$} is the set
\begin{equation}\label{eq:cylinder}
	\cyl_p^{m}(R)=\Big\{(v_1,v_2,\ldots,v_m)\in(\Qp)^{m}:\|(v_1,v_2)\|_p< pR\Big\}=\ball_p^2(R)\times(\Qp)^{m-2}.
\end{equation}
If $m=2$ then
\[\cyl_p^2(R)=\ball_p^2(R).\]
We often simply say that $\ball_p^{m}(r)$ is a \emph{$p$-adic ball} and $\cyl_p^{m}(R)$ is a \emph{$p$-adic cylinder}, without specifying the parameters $r$, $R$ or $m$. The topology on $(\Qp)^m$ is the one induced by the $p$-adic norm. In this topology, both $\ball_p^m(r)$ and $\cyl_p^m(R)$ are open and closed subsets of $(\Qp)^m$, because for any $v\in(\Qp)^m$, \[\|v\|_p<pr\] is equivalent to \[\|v\|_p\le r.\] This is in strong contrast with the intuition we have about these sets in the real case, where neither of them is a closed subset of $\R^m$. We refer to Figures \ref{fig:ball-cylinder-2} and \ref{fig:ball-cylinder-3} for a representation of $p$-adic balls, $p$-adic cylinders and similar products in dimension $2$ and $3$.

\begin{figure}
	\begin{tabular}{cc}
		\includegraphics[width=0.45\linewidth]{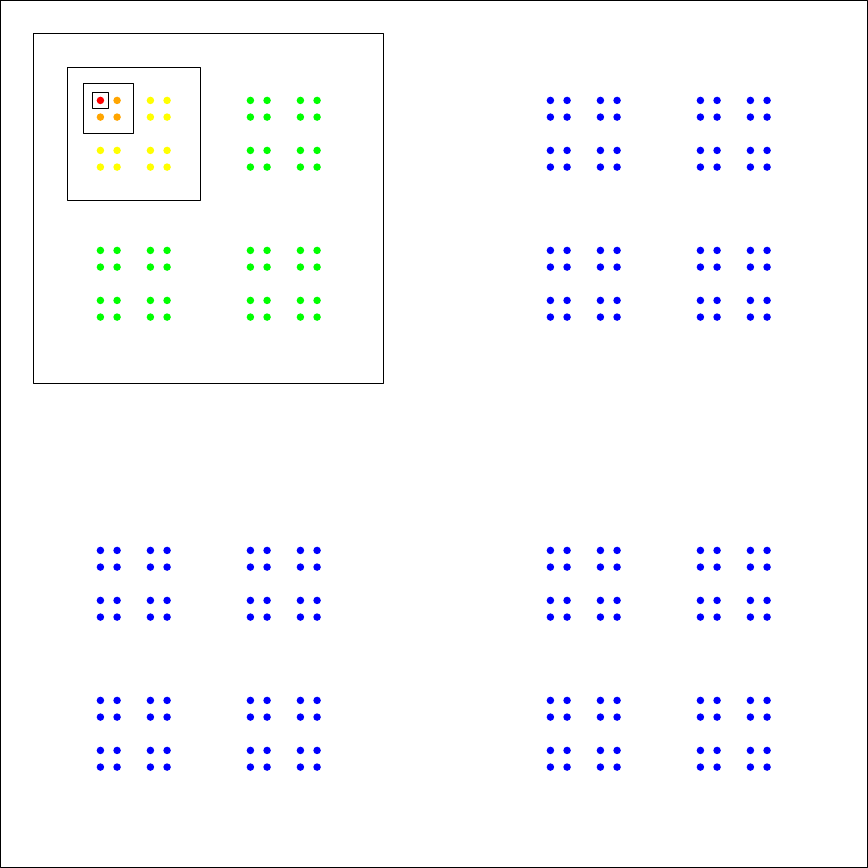} & \includegraphics[width=0.45\linewidth]{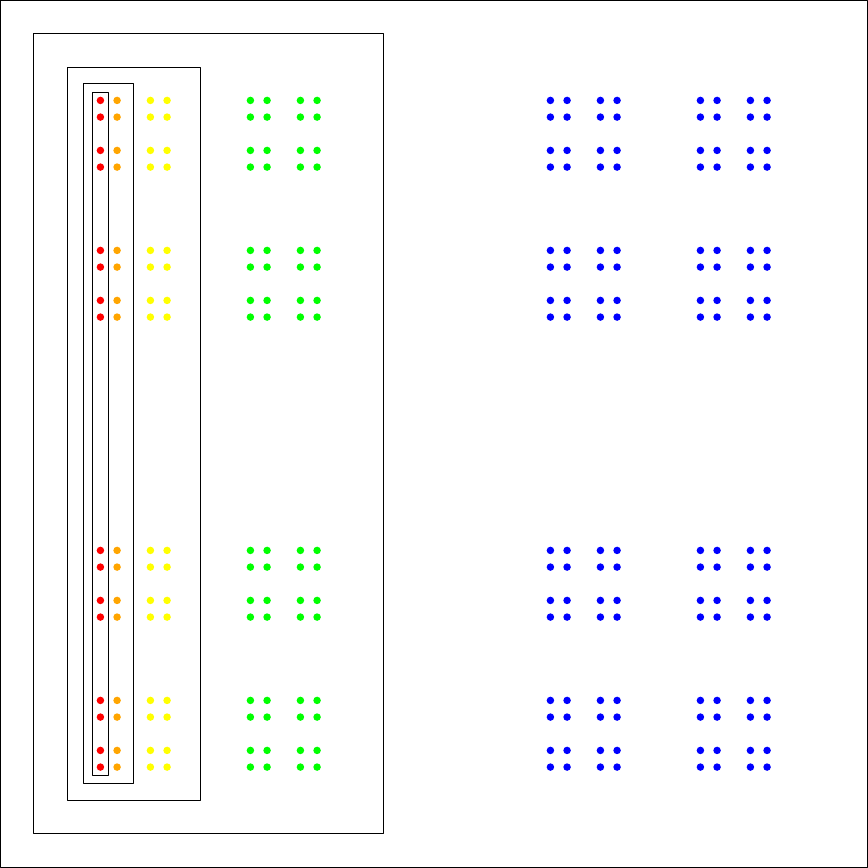} \\
		$m=2,s=2,r_0=16$ & $m=2,s=1,r_0=16$ \\[2cm]
		\fbox{\includegraphics[width=0.45\linewidth]{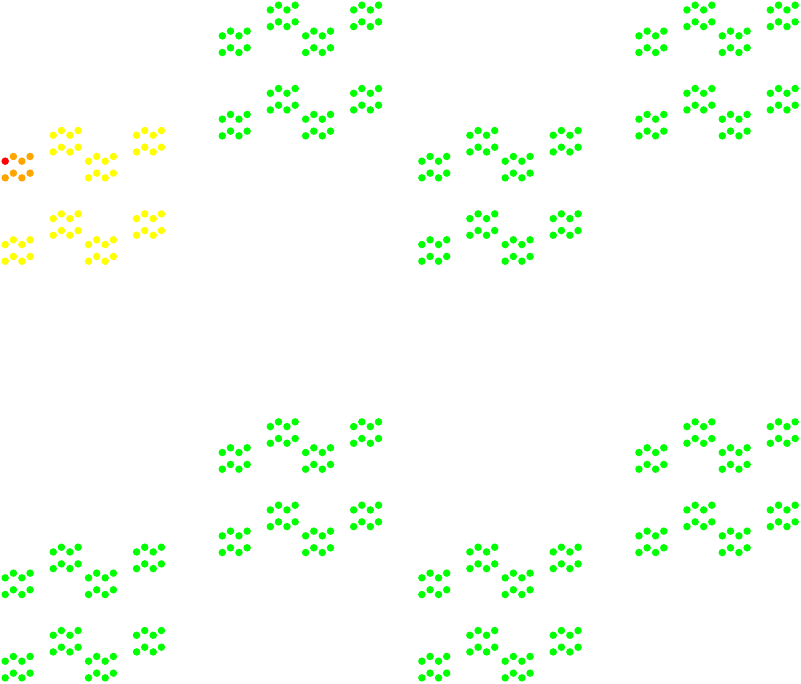}} & \fbox{\includegraphics[width=0.45\linewidth]{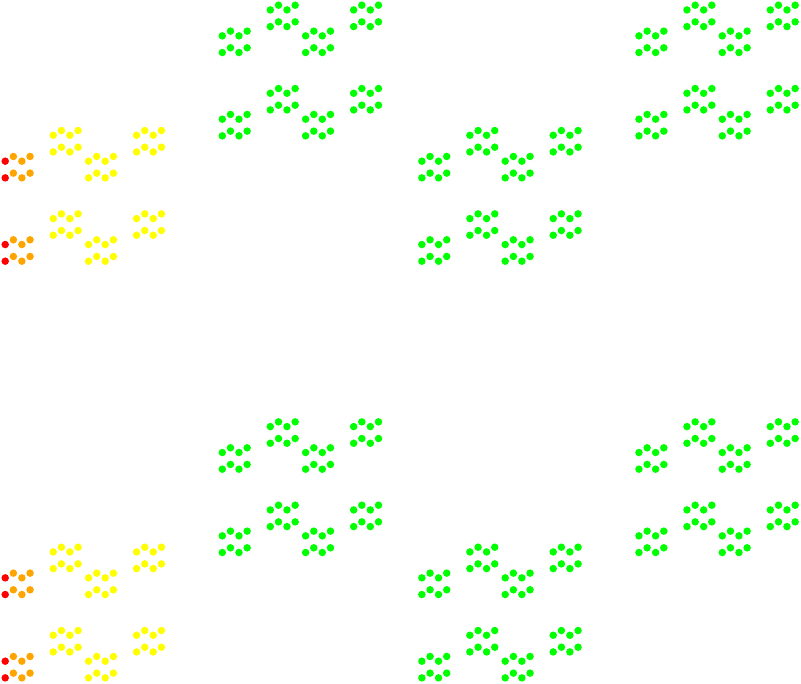}} \\
		$m=3,s=3,r_0=8$ & $m=3,s=2,r_0=8$
	\end{tabular}
	\caption{Each one of the four figures is a symbolic representation of \[\ball_2^m(r_0)=\Big\{(v_1,v_2,\ldots,v_m)\in(\Q_2)^{m}:\|(v_1,v_2,\ldots,v_m)\|_2<2r_0\Big\},\] as it appears in expression \eqref{eq:ball}, for some $m$ ($2$ in the upper figures and $3$ in the lower figures) and $r_0$ ($16$ and $8$ respectively). Each dot is a $2$-adic ball of radius $1$, and the $2$-adic balls that are closer in the representation are also $2$-adically closer. The dots of the same color are those $2$-adic balls that are contained in \[\ball_2^s(r)\times(\Q_2)^{m-s},\] where $s$ is $1$, $2$ or $3$ depending on the figure and $r$ is $1$ for the red dots, $2$ for orange, $4$ for yellow, $8$ for green and $16$ for blue.}
	\label{fig:ball-cylinder-2}
\end{figure}

\begin{figure}
	\begin{tabular}{cc}
		\includegraphics[width=0.45\linewidth]{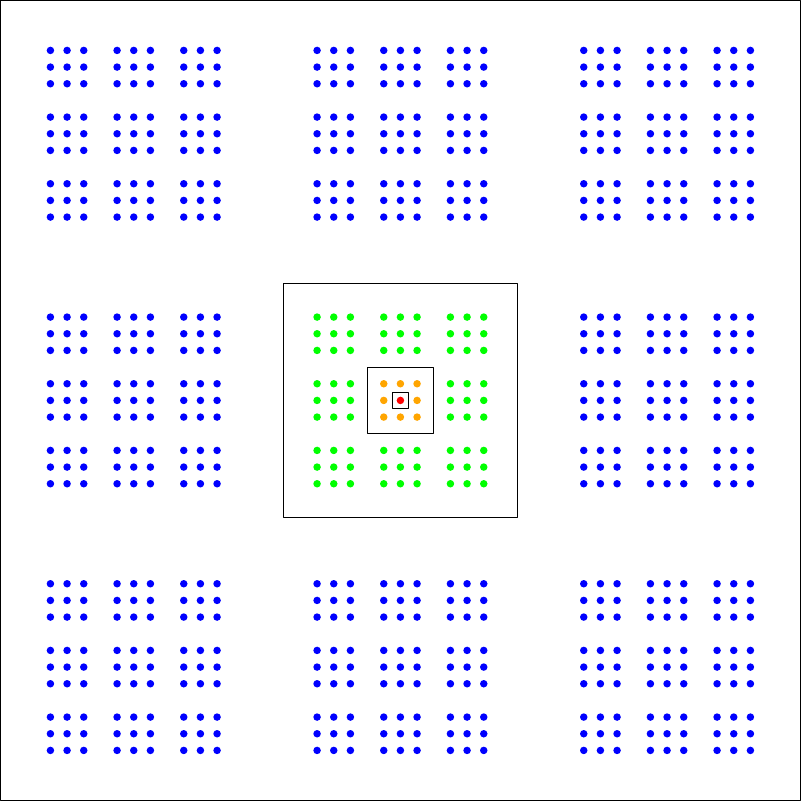} & \includegraphics[width=0.45\linewidth]{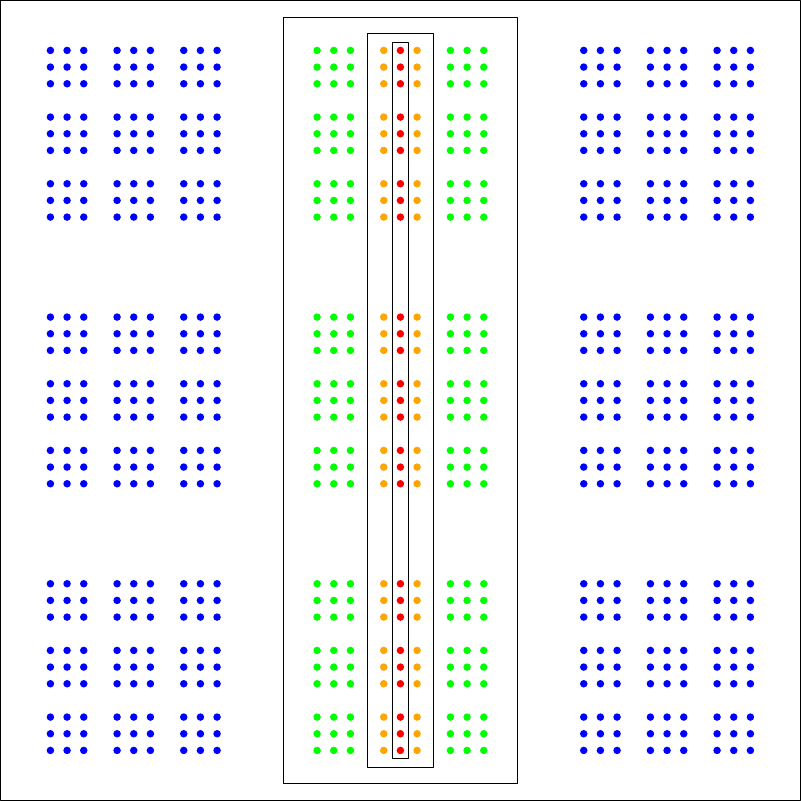} \\
		$m=2,s=2,r_0=27$ & $m=2,s=1,r_0=27$ \\[2cm]
		\fbox{\includegraphics[width=0.45\linewidth]{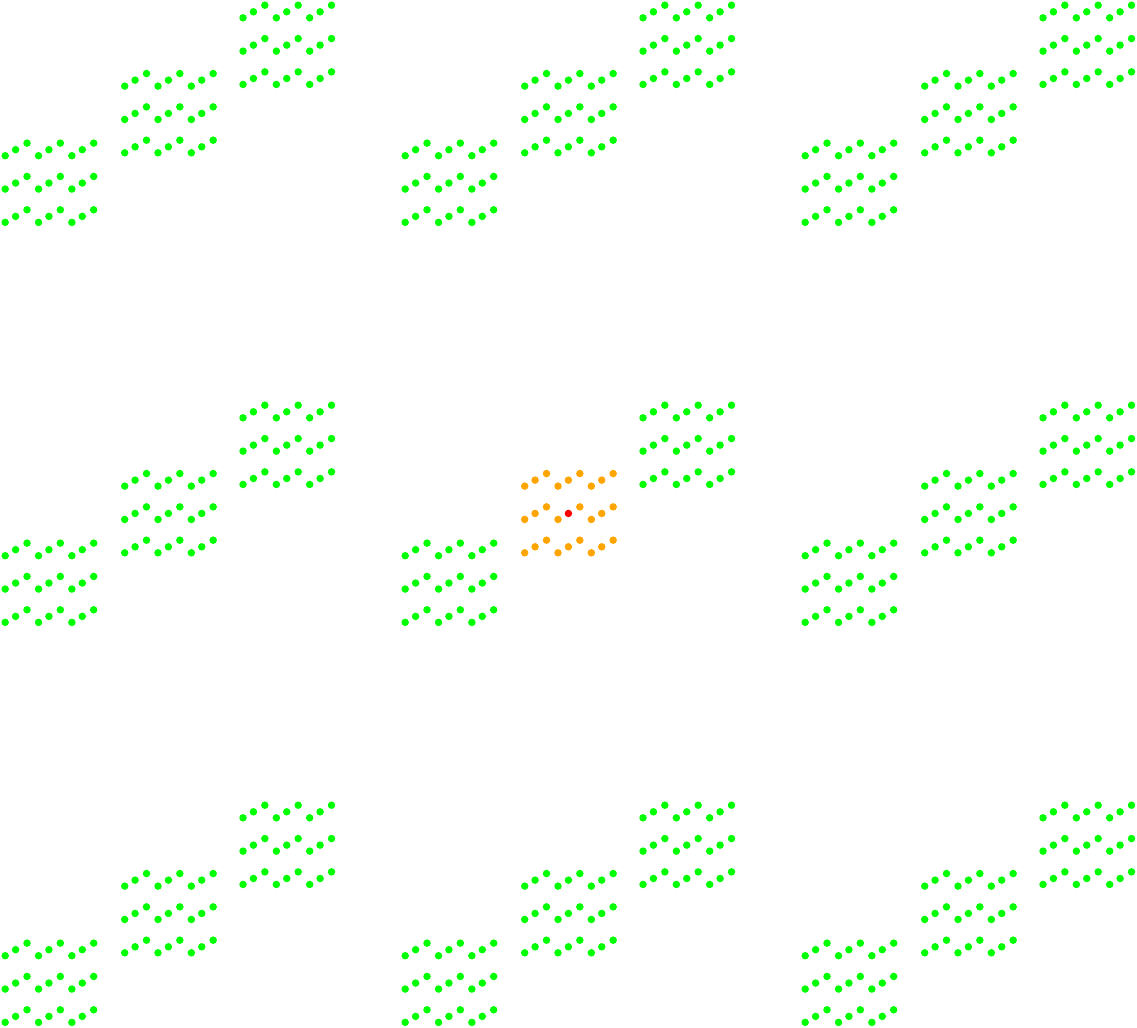}} & \fbox{\includegraphics[width=0.45\linewidth]{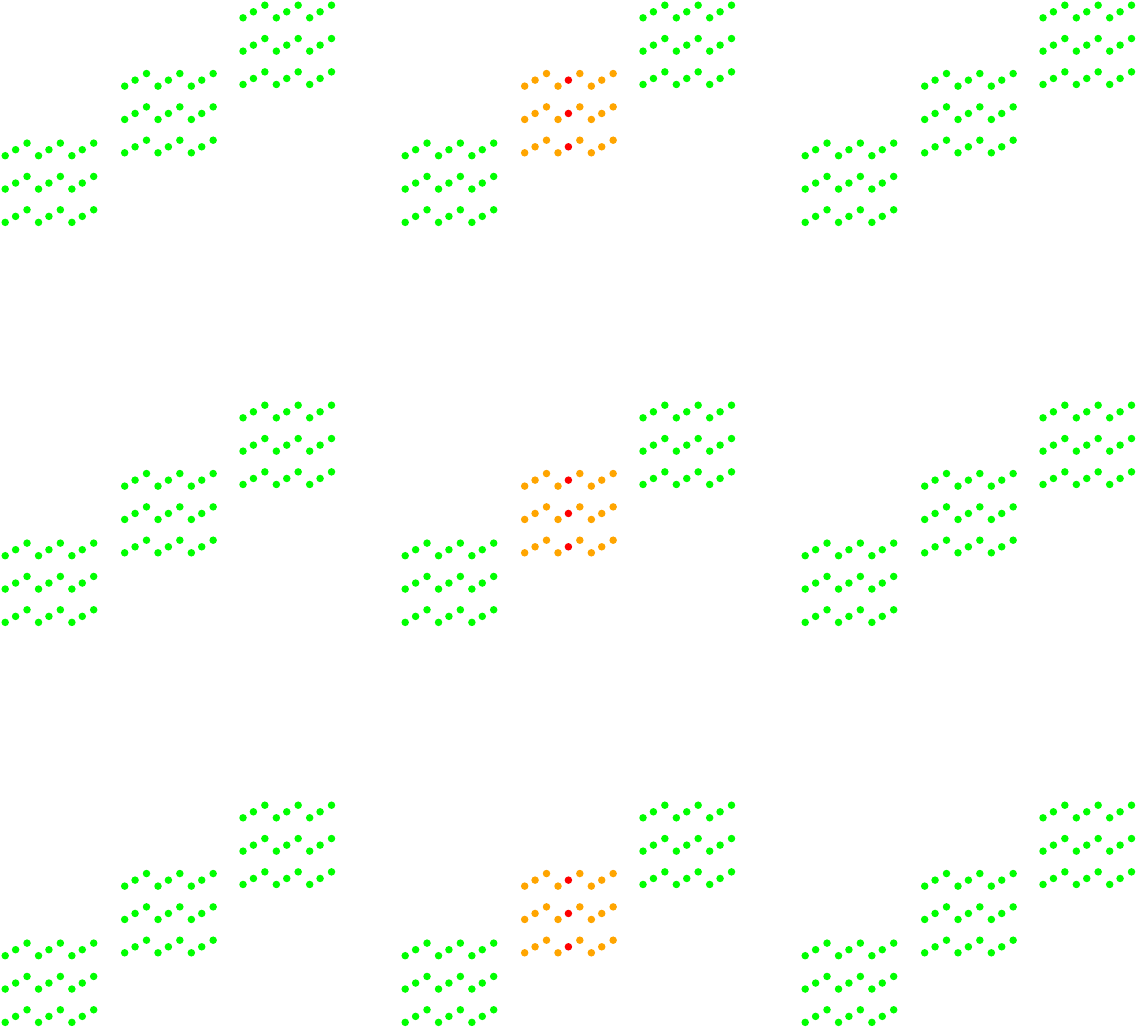}} \\
		$m=3,s=3,r_0=9$ & $m=3,s=2,r_0=9$
	\end{tabular}
	\caption{Each one of the four figures is a symbolic representation of \[\ball_3^m(r_0)=\Big\{(v_1,v_2,\ldots,v_m)\in(\Q_3)^{m}:\|(v_1,v_2,\ldots,v_m)\|_3<3r_0\Big\},\] as it appears in expression \eqref{eq:ball}, for some $m$ ($2$ in the upper figures and $3$ in the lower figures) and $r_0$ ($27$ and $9$ respectively). Each dot is a $3$-adic ball of radius $1$, and the $3$-adic balls that are closer in the representation are also $3$-adically closer. The dots of the same color are those $3$-adic balls that are contained in \[\ball_3^s(r)\times(\Q_3)^{m-s},\] where $s$ is $1$, $2$ or $3$ depending on the figure and $r$ is $1$ for the red dots, $3$ for orange, $9$ for green and $27$ for blue.}
	\label{fig:ball-cylinder-3}
\end{figure}

In the $p$-adic case, unlike the real case, a ball is the same if we change the ``center'' to any point of the ball: for any $v_0\in\ball_p^m(r)$,
\[\Big\{v+v_0:v\in\ball_p^m(r)\Big\}=\ball_p^m(r).\]

Any two $p$-adic balls are diffeomorphic: we can go from $\ball_p^m(r)$ to $\ball_p^m(r')$ by scaling by a factor $c\in\Qp$ such that $|c|_p=r'/r$ (for example $c=r/r'$). Similarly, any two $p$-adic cylinders are also diffeomorphic. Also, by definition, $\ball_p^m(r)\subset\cyl_p^m(r)$, hence we can embed any $p$-adic ball into any $p$-adic cylinder by means of the inclusion map.

In the $p$-adic case, like in the real case, it is possible to embed a $p$-adic ball into a $p$-adic cylinder of a smaller radius by means of a $p$-adic linear volume-preserving transformation:
\[\begin{array}{ccc}
	(\Qp)^m & \longrightarrow & (\Qp)^m \\
	(v_1,v_2,v_3,v_4,v_5,\ldots,v_m) & \mapsto & (cv_1,cv_2,\frac{1}{c}v_3,\frac{1}{c}v_4,v_5,\ldots,v_m)
\end{array}\]
for $c$ small enough. Here by ``volume preserving'' we mean that $f$ satisfies \[f^*(\dd v_1\wedge\ldots\wedge\dd v_m)=\dd x_1\wedge\ldots\wedge\dd x_m,\] where $(v_1,\ldots,v_m)$ are the standard coordinates on $(\Qp)^m$.

In the remaining part of the paper, $m$ is an even positive integer of the form $2n$, that is, we will only be concerned with even-dimensional cylinders and even-dimensional balls. We will use coordinates $(v_1,v_2,\ldots,v_{2n-1},v_{2n})=(x_1,y_1,\ldots,x_n,y_n)$ on $(\Qp)^{2n}$.

\section{$p$-adic affine Gromov's non-squeezing} \label{sec:affine}

In this section we prove a $p$-adic analog of Gromov's linear non-squeezing theorem (Theorem \ref{thm:linear}).

\begin{definition}\label{def:linear-symplectic}
	\letnpos. \letpprime. A \emph{$p$-adic linear symplectic form} $\omega:(\Qp)^{2n}\times(\Qp)^{2n}\to\Qp$ on $(\Qp)^{2n}$ is a non-degenerate antisymmetric bilinear form. Let $(x_1,y_1,\ldots,x_n,y_n)$ be the standard coordinates on $(\Qp)^{2n}$. The $p$-adic linear symplectic form $\omega_0=\sum_{i=1}^n\dd x_i\wedge\dd y_i$, which is given by the block-diagonal matrix with $n$ blocks
	\begin{equation}\label{eq:omega}
		\begin{pmatrix}
			0 & 1 & & & \\
			-1 & 0 & & & \\
			& & \ddots & & \\
			& & & 0 & 1 \\
			& & & -1 & 0
		\end{pmatrix},
	\end{equation}
	is called the \emph{standard $p$-adic symplectic form} on $(\Qp)^{2n}$.
\end{definition}

\begin{definition}[$p$-adic linear/affine symplectomorphism of $(\Qp)^{2n}$]
	\letnpos. \letpprime. Given a $p$-adic linear symplectic form $\omega$ on $(\Qp)^{2n}$, a \emph{$p$-adic linear symplectomorphism} $\varphi:(\Qp)^{2n}\overset{\cong}{\to}(\Qp)^{2n}$ is a linear isomorphism of vector spaces such that $\varphi^*\omega=\omega$, that is, $\varphi$ is given by a matrix $S$ such that $S\tr \Omega S=\Omega$, where $\Omega$ is the matrix of $\omega$. A map $\phi:(\Qp)^{2n}\to(\Qp)^{2n}$ is a \emph{$p$-adic affine symplectomorphism} if there exists a $p$-adic linear symplectomorphism $\varphi:(\Qp)^{2n}\to(\Qp)^{2n}$ and $v_0\in(\Qp)^{2n}$ such that $\phi=\varphi+v_0$, that is, $\phi$ is given by $v\mapsto Sv+v_0$, where $S$ is the matrix of $\varphi$. We denote by $\ASp((\Qp)^{2n})$ the group of $p$-adic affine symplectomorphisms of $(\Qp)^{2n}$ under composition.
\end{definition}

\begin{definition}[$p$-adic affine symplectic embedding between open subsets of $(\Qp)^{2n}$]
	\letnpos. \letpprime. Given a $p$-adic linear symplectic form $\omega$ and two open subsets $U$ and $V$ of $(\Qp)^{2n}$, a \emph{$p$-adic affine symplectic embedding} $f:U\hookrightarrow V$ is a $p$-adic affine embedding such that $f^*\omega=\omega$.
\end{definition}

Next we state the analog of Gromov's non-squeezing theorem for $p$-adic affine symplectomorphisms, where the analogy with the real case does hold (see for example \cite[Theorem 2.4.1]{McDSal} for the statement in the real case). For the following statement recall the definition of the $p$-adic ball given in \eqref{eq:ball} and of the $p$-adic cylinder given in \eqref{eq:cylinder}.

\begin{theorem}[$p$-adic analog of the affine Gromov's non-squeezing theorem]\label{thm:linear}
	Let $n$ be a positive integer. \letpprime. Let $r$ and $R$ be $p$-adic absolute values. Endow the $2n$-dimensional $p$-adic ball $\ball_p^{2n}(r)$ of radius $r$ and the $2n$-dimensional $p$-adic cylinder $\cyl_p^{2n}(R)$ of radius $R$ with the standard $p$-adic linear symplectic form $\sum_{i=1}^n\dd x_i\wedge\dd y_i$, where $(x_1,y_1,\ldots,x_n,y_n)$ are the standard coordinates on $(\Qp)^{2n}$. Then, there exists a $p$-adic affine symplectic embedding $f:\ball_p^{2n}(r)\hookrightarrow\cyl_p^{2n}(R)$ if and only if $r\le R$.
\end{theorem}

\begin{proof}
	For every positive integer $n$ and every $i\in\{1,\ldots,2n\}$, let $\e_i$ be the $i$-th vector of the canonical basis of $(\Qp)^{2n}$, given by $(0,\ldots,1,\ldots,0)$ with the $1$ in position $i$.
	
	If $r\le R$, then the inclusion of the ball in the cylinder gives the embedding we want. Suppose now that a $p$-adic affine symplectic embedding $f:\ball_p^{2n}(r)\hookrightarrow\cyl_p^{2n}(R)$ exists. If $n=1$ the cylinder of a given radius is equal to the ball $\ball_p^{2n}(R)$, and the embedding $\ball_p^{2n}(r)\hookrightarrow\ball_p^{2n}(R)$ must preserve the volume form $\dd x_1\wedge\dd y_1\wedge\ldots\wedge\dd x_n\wedge\dd y_n$ and hence the total $p$-adic volume, so $r^2\le R^2$ (the concept of $p$-adic volume can be defined in analogy with the real case, see for instance Popa's notes \cite{Popa}). Hence $r\le R$, since $r$ and $R$ are $p$-adic absolute values. Assume now that $n\ge 2$.
	
	Any $p$-adic linear symplectomorphism of $(\Qp)^{2n}$ is given as multiplication by a symplectic matrix $S$ and an affine one as multiplication by $S$ plus a translation by a fixed vector $v_0\in(\Qp)^{2n}$.
	
	Let $u_1=S\tr \e_1$ and $u_2=S\tr \e_2$. We have that \[\omega_0(u_1,u_2)=\omega_0(\e_1,\e_2)=1\] and
	\[1=|\omega_0(u_1,u_2)|_p\le\|u_1\|_p\|u_2\|_p,\]
	hence one of the norms must be at least $1$.
	
	Without loss of generality we may assume that $\|u_1\|_p\ge 1$. By definition of $p$-adic norm, this means that a coordinate of $u_1$ has $p$-adic absolute value at least $1$. Suppose that $|u_1\tr \e_i|_p\ge 1$ for some $i\in\{1,\ldots,2n\}$. Then the vector $\e_i/r$ is in $\ball_p^{2n}(r)$, hence \[\frac{S\e_i}{r}+v_0\in\cyl_p^{2n}(R),\] which implies that
	\[\left|\e_1\tr \frac{S\e_i}{r}+\e_1\tr v_0\right|_p\le R.\]
	Since $0\in\ball_p^{2n}(r)$, we have $v_0\in \cyl_p^{2n}(R)$, so \[|\e_1\tr v_0|_p\le R\] and
	\[\frac{|\e_1\tr S\e_i|_p}{|r|_p}=\left|\e_1\tr \frac{S\e_i}{r}\right|_p\le\max\left\{\left|\e_1\tr \frac{S\e_i}{r}+\e_1\tr v_0\right|_p,|\e_1\tr v_0|_p\right\}\le R.\]
	Since $r$ is a power of $p$, $|r|_p=1/r$ and
	\[R\ge r|\e_1\tr S\e_i|_p=r|u_1\tr \e_i|_p\ge r,\]
	as we wanted to prove.
\end{proof}

We refer to Figure \ref{fig:squeezing-linear} for an illustration of the proof of Theorem \ref{thm:linear}.

\begin{figure}
	\includegraphics{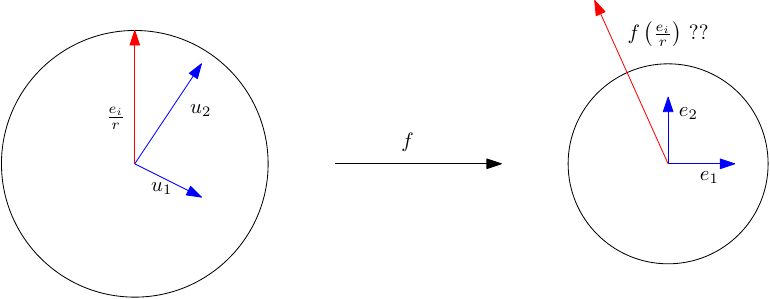}
	\caption{Illustration of the proof of Theorem \ref{thm:linear}, or more precisely of its real analog. If $u_2=S\tr \e_2$ is longer than $\e_2$ and the cylinder is narrower than the ball, $f$ would send out of the cylinder a vector as long as possible in any direction close to $u_2$.}
	\label{fig:squeezing-linear}
\end{figure}

\section{$p$-adic non-linear symplectic squeezing}\label{sec:nonlinear}

\letnpos. In the nonlinear case, the situation is genuinely different than in the affine situation treated in Section \ref{sec:affine}: not only one can embed any $p$-adic ball into any $p$-adic cylinder by means of a $p$-adic analytic symplectic embedding, but one can show that, for any $p$-adic absolute value $R$, the total $2n$-dimensional $p$-adic space \[(\Qp)^{2n}=(\Qp)^2\times\ldots\times(\Qp)^2\] is symplectomorphic to the $2n$-dimensional $p$-adic cylinder $\cyl_p^{2n}(R)$. We refer to \cite[Appendix B]{CrePel-JC} for a basic review of $p$-adic analytic symplectic manifolds.

\begin{definition}[$p$-adic analytic symplectic manifolds and embeddings]
	\letpprime. A \emph{$p$-adic analytic symplectic manifold} $(M,\omega)$ is a $p$-adic analytic manifold $M$ endowed with a closed non-degenerate $p$-adic $2$-form $\omega$. The form $\omega$ is called a \emph{$p$-adic analytic symplectic form on $M$.} (Sometimes we simply call it a \emph{$p$-adic symplectic form} for simplicity.) Let $(M,\omega)$ and $(N,\sigma)$ be $p$-adic analytic symplectic manifolds. A \emph{$p$-adic analytic symplectic embedding} $f:M\hookrightarrow N$ is a $p$-adic analytic embedding such that $f^*\sigma=\omega$. A \emph{$p$-adic analytic symplectomorphism} $\phi:M\overset{\cong}{\to} N$ is a $p$-adic analytic diffeomorphism such that $\phi^*\sigma=\omega$.
\end{definition}

In this paper, \emph{all $p$-adic analytic symplectic manifolds are open subsets of $(\Qp)^{2n}$} endowed with the standard $p$-adic symplectic form $\sum_{i=1}^n\dd x_i\wedge\dd y_i$, as in Definition \ref{def:linear-symplectic}. The following statement shows the striking difference between the behavior of symplectic embeddings in the real and $p$-adic cases (see \cite[section 12]{McDSal} for a discussion of the real case). Recall that the definition of a $p$-adic ball $\ball_p^{2n}(r)$ was given in \eqref{eq:ball}, and the definition of $p$-adic cylinder $\cyl_p^{2n}(R)$ was given by \eqref{eq:cylinder}.

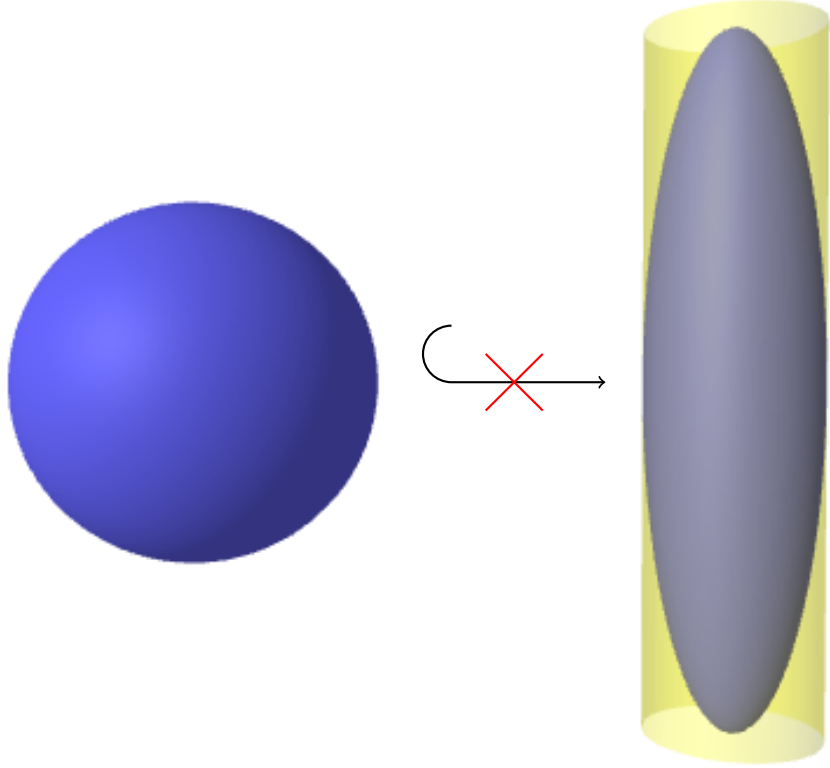
\begin{figure}
	\begin{tikzpicture}[scale=1.5]
		\node at (0,0) {\includegraphics[trim=7cm 1cm 7cm 1cm,width=0.75\linewidth]{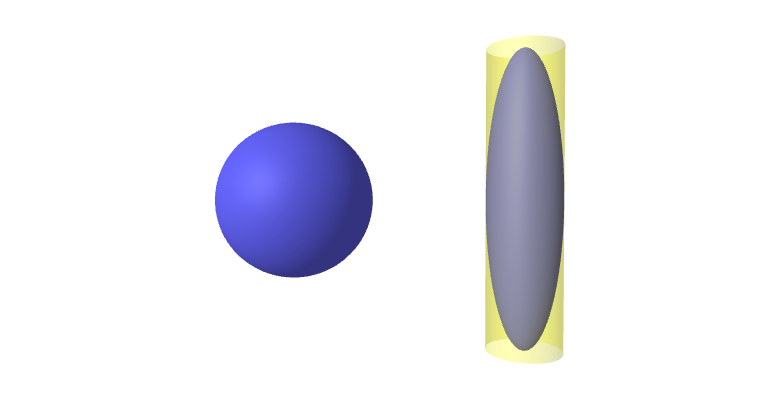}};
		\draw[->,thick] (0.25,0.5) arc (90:270:0.25)--(1.6,0);
		\draw[red,thick] (0.55,0.25)--(1.05,-0.25);
		\draw[red,thick] (0.55,-0.25)--(1.05,0.25);
	\end{tikzpicture}
	\caption{Gromov's non-squeezing theorem tells us that it is not possible to embed a ball inside a cylinder narrower than the ball while preserving the symplectic form. Theorem \ref{thm:total-embedding} shows that the situation is different for $p$-adic manifolds: the entire $p$-adic $2n$-dimensional space $(\Qp)^{2n}$ is symplectomorphic to any thin $p$-adic cylinder of the same dimension.}
	\label{fig:squeezing-real}
\end{figure}

\begin{theorem}[$p$-adic symplectic squeezing theorem]\label{thm:total-embedding}
	Let $n\ge 2$ be an integer. \letpprime. Endow both the $2n$-dimensional $p$-adic space $(\Qp)^{2n}$ and the $2n$-dimensional $p$-adic cylinder $\cyl_p^{2n}(1)$ of radius $1$ with the standard $p$-adic symplectic form $\sum_{i=1}^n\dd x_i\wedge\dd y_i$, where $(x_1,y_1,\ldots,x_n,y_n)$ are the standard coordinates on $(\Qp)^{2n}$. Then there exists a $p$-adic analytic symplectomorphism \[\phi:(\Qp)^{2n}\overset{\cong}{\longrightarrow}\cyl_p^{2n}(1).\]
	
	Moreover the $p$-adic analytic symplectomorphism $\phi:(\Qp)^{2n}\overset{\cong}{\to}\cyl_p^{2n}(1)$ may be chosen to have the following properties:
	\renewcommand{\theenumi}{\roman{enumi}}
	\begin{enumerate}
		\item $\phi(0,0,\ldots,0)=(0,0,\ldots,0)$;
		\item For every $p$-adic absolute value $r$,
		\[\phi(\ball_p^{2n}(r))=\begin{cases}
			\ball_p^{2n}(r) & \text{if }r\le 1, \\
			\ball_p^2(1)\times\ball_p^2(r^2)\times\ball_p^{2n-4}(r) & \text{if }r>1;
		\end{cases}\]
		\item For every integer $i\in\{1,\ldots,n\}$,
		\[\phi(\{(x_1,y_1,\ldots,x_n,y_n):x_i=0\})=\]
		\[\begin{cases}
			\{(x_1,y_1,\ldots,x_n,y_n):x_1=0,x_2\text{ has all digits at odd fractional places equal to }0\} & \text{if }i=1, \\
			\{(x_1,y_1,\ldots,x_n,y_n):x_2\text{ has all digits at even fractional places equal to }0\} & \text{if }i=2, \\
			\{(x_1,y_1,\ldots,x_n,y_n):x_i=0\} & \text{if }i>2;
		\end{cases}\]
		\item the same statement as in (iii) above holds when changing $x_i$ by $y_i$;
		\item For every $p$-adic absolute value $r$,
		\[\phi^{-1}(\ball_p^{2n}(r))=\begin{cases}
			\ball_p^{2n}(r) & \text{if }r\le 1, \\
			\ball_p^2(p^k)\times\ball_p^2(p^k)\times\ball_p^{2n-4}(r) & \text{if $r=p^{2k}$ for $k\ge 1$ integer}, \\
			\ball_p^2(p^{k+1})\times\ball_p^2(p^k)\times\ball_p^{2n-4}(r) & \text{if $r=p^{2k+1}$ for $k\ge 0$ integer;}
		\end{cases}\]
		\item For every integer $i\in\{1,\ldots,n\}$,
		\[\phi^{-1}(\{(x_1,y_1,\ldots,x_n,y_n):x_i=0\})=\begin{cases}
			\{(x_1,y_1,\ldots,x_n,y_n):\text{the integer part of $x_1$ is }0\} & \text{if }i=1, \\
			\{(x_1,y_1,\ldots,x_n,y_n):x_1\in\Zp,x_2=0\} & \text{if }i=2, \\
			\{(x_1,y_1,\ldots,x_n,y_n):x_i=0\} & \text{if }i>2;
		\end{cases}\]
		\item the same statement as in (vi) above holds when changing $x_i$ by $y_i$.
	\end{enumerate}
\end{theorem}

\begin{proof}
	We refer to Figure \ref{fig:embedding} for an intuitive idea of the $p$-adic analytic symplectomorphism given in this proof.

We start with a point $(x_1,y_1,\ldots,x_n,y_n)\in(\Qp)^{2n}$ and write
\begin{equation}\label{eq:descomp-before}
	\left\{
	\begin{aligned}
		x_1 & =a_0+\sum_{i=1}^\infty a_ip^{-i}; \\
		y_1 & =b_0+\sum_{i=1}^\infty b_ip^{-i}; \\
		x_2 & =c_0+\sum_{i=1}^\infty c_ip^{-i}; \\
		y_2 & =d_0+\sum_{i=1}^\infty d_ip^{-i},
	\end{aligned}
	\right.
\end{equation}
where $a_0,b_0,c_0,d_0\in\Zp$ and $a_i,b_i,c_i,d_i\in\{0,\ldots,p-1\}$ for $i\ge 1$. Now we define \[\phi(x_1,y_1,\ldots,x_n,y_n)=(x_1',y_1',x_2',y_2',x_3,y_3,\ldots,x_n,y_n),\] where
\begin{equation}\label{eq:descomp-after}
	\left\{
	\begin{aligned}
		x_1' & =a_0; \\
		y_1' & =b_0; \\
		x_2' & =c_0+\sum_{i=1}^\infty (a_ip^{2i-1}+c_ip^{2i}); \\
		y_2' & =d_0+\sum_{i=1}^\infty (a_ip^{2i-1}+c_ip^{2i}).
	\end{aligned}
	\right.
\end{equation}
The fact that $\phi$ is a bijection between $(\Qp)^{2n}$ and $\cyl_p^{2n}(1)$ follows from the fact that every element of $\Qp$ has a unique $p$-adic expansion (Proposition \ref{prop:expansion}), which means that the decomposition in \eqref{eq:descomp-before} exists and is unique, and a decomposition in the form \eqref{eq:descomp-after} exists if and only if the point $(x_1,\ldots,x_n)$ is in the $p$-adic cylinder $\cyl_p^{2n}(1)$ of radius $1$, and in that case it is unique. Moreover, $\phi$ becomes a translation when restricted to any $p$-adic ball of radius $1$, and any translation preserves the $p$-adic symplectic form $\omega_0$, hence $\phi$ is a $p$-adic analytic symplectomorphism.

Now we prove properties (i)--(vii):
\begin{enumerate}
	\renewcommand{\theenumi}{\roman{enumi}}
	\item If all the coordinates are $0$, all the $p$-adic digits are $0$ and the result is also $0$.
	\item Being in the ball of radius $p^k$ means that the digits are $0$ at the right of the position $-k$. If $k\le 0$, only $a_0$, $b_0$, $c_0$ and $d_0$ are nonzero and the point is unchanged. Otherwise, there are at most $k$ nonzero digits at the right of the decimal point. $x_1'$ and $y_1'$ will always be integers, and $x_2'$ and $y_2'$ have $2k$ digits at the right of the decimal point, hence they are in a ball of radius $r^2$. The rest of coordinates are unchanged.
	\item If a coordinate different from the first four is forced to be $0$, it will still be $0$ in the final point. If we make $x_1$ equal to $0$, the final $x_1'$ will also be $0$, and $a_i=0$ for all $i$, hence the digits at odd places at the right of the decimal point are $0$ in $x_2'$. If we force $x_2$ to be $0$ instead, $b_i=0$ for any $i$ and the digits at even places are $0$ in $x_2'$.
	\item Same as (iii) but with $y_i$ instead of $x_i$.
	\item If we want the result to be in a ball of radius less than $1$, only $a_0$, $b_0$, $c_0$ and $d_0$ can be nonzero and the point is in the same ball. If \[r=p^{2k}\] for $k\ge 1$, we need at most $2k$ digits at the right of the decimal point of all $x_i'$ and $y_i'$, which means the same for all $x_i$ and $y_i$ for $i\ge 3$, and $2k$ digits also for $x_2'$ and $y_2'$, which means that $x_1$, $y_1$, $x_2$ and $y_2$ must have $k$ digits. If \[r=p^{2k+1}\] for $k\ge 0$, $x_i$ and $y_i$ have $2k+1$ digits, and also $x_2'$ and $y_2'$, hence $x_1$ and $y_1$ have $k+1$ digits and $x_2$ and $y_2$ have $k$.
	\item If we want $x_i'=0$ for $i>2$, then we need $x_i=0$. If we want $x_1'=0$, we need $a_0=0$, that is, the integer part of $x_1$ is $0$. If we want $x_2'=0$, we need $b_0=0$ and also all fractional digits of $x_1$ and $x_2$ are $0$, hence $x_1\in\Zp$ and $x_2=0$.
	\item Same as (vi) but with $y_i$ instead of $x_i$.\qedhere
\end{enumerate}
\end{proof}

See Figure \ref{fig:embedding} for a representation of the $p$-adic analytic symplectomorphism $\phi$ in the case $p=3$. We will give an equivariant version of Theorem \ref{thm:total-embedding} in Section \ref{sec:equivariant}.

\begin{corollary}\label{cor:embedding}
	Let $n\ge 2$ be an integer. \letpprime. Let $N$ be an open subset of $(\Qp)^{2n}$ endowed with the standard $p$-adic symplectic form $\sum_{i=1}^n\dd x_i\wedge\dd y_i$, where $(x_1,y_1,\ldots,x_n,y_n)$ are the standard coordinates on $(\Qp)^{2n}$. Then, there exists a $p$-adic analytic symplectic embedding $f:N\hookrightarrow\cyl_p^{2n}(1).$
\end{corollary}

\begin{corollary}
	Let $n\ge 2$ be an integer. \letpprime. Let $R_1,R_2$ be $p$-adic absolute values. Endow the $2n$-dimensional $p$-adic cylinders $\cyl_p^{2n}(R_1)$ and $\cyl_p^{2n}(R_2)$ with the standard $p$-adic symplectic form $\sum_{i=1}^n\dd x_i\wedge\dd y_i$, where $(x_1,y_1,\ldots,x_n,y_n)$ are the standard coordinates on $(\Qp)^{2n}$. Then there exists a $p$-adic analytic symplectomorphism \[\phi:\cyl_p^{2n}(R_1)\overset{\cong}{\longrightarrow}\cyl_p^{2n}(R_2).\]
\end{corollary}

\begin{proof}
	Let $\phi_1$ be the scaling with a factor of $R_1$. We have that $\phi_1(\cyl_p^{2n}(R_1))=\cyl_p^{2n}(1)$, because $|R_1|_p=1/R_1$. Let $\phi_2$ be the $p$-adic analytic symplectomorphism from $\cyl_p^{2n}(1)$ to $(\Qp)^{2n}$. Now
	\[\phi_1^{-1}\circ\phi_2\circ\phi_1:\cyl_p^{2n}(R_1)\overset{\cong}{\to}(\Qp)^{2n}\]
	is a $p$-adic analytic symplectomorphism, because
	\begin{align*}
		(\phi_1^{-1}\circ\phi_2\circ\phi_1)^*\omega_0 & =\phi_1^*\phi_2^*(\phi_1^{-1})^*\omega_0 \\
		& =\phi_1^*\phi_2^*\frac{\omega_0}{R_1^2} \\
		& =\phi_1^*\frac{\omega_0}{R_1^2} \\
		& =\omega_0.
	\end{align*}
	By using $R_2$ instead of $R_1$, we conclude that $\cyl_p^{2n}(R_2)$ is also symplectomorphic to $(\Qp)^{2n}$. Composing the two symplectomorphisms, the result is the $\phi$ that we want.
\end{proof}

\begin{figure}
	\includegraphics{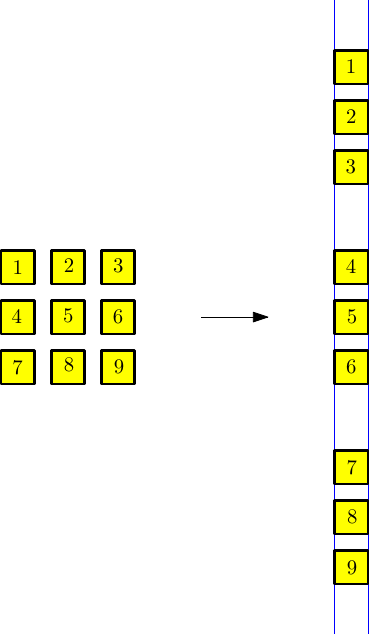}
	\caption{A representation of the $p$-adic analytic symplectomorphism of Theorem \ref{thm:total-embedding} for $p=3$. The horizontal and vertical directions represent $x_1$ and $x_2$. The embedding is continuous in the $p$-adic case because the balls of radius $1$ (represented here by squares) are at a fixed distance from each other. In the real case the balls would share common boundaries, and the embedding is discontinuous.}
	\label{fig:embedding}
\end{figure}

\section{Characterization of the $p$-adic linear squeezing property}\label{sec:char-squeezing}

In this section we study the concept of $p$-adic affine symplectic rigidity, which in the real case is stated in McDuff-Salamon \cite[Theorem 2.4.2]{McDSal}. In the $p$-adic case, it is possible to prove an analogous theorem, but the hypothesis is different.

\letnpos. Recall that a \emph{linear symplectic ball} in $\R^{2n}$ is the image of a ball $\ball^{2n}(r)$ of some radius $r>0$ by a linear symplectic embedding, and a \emph{linear symplectic cylinder} in $\R^{2n}$ is the image of a cylinder $\cyl^{2n}(R)$ by a linear symplectic embedding. A matrix in $\M_{2n}(\R)$ is \emph{squeezing} if it sends some linear symplectic ball into a linear symplectic cylinder of smaller radius, and \emph{non-squeezing} otherwise. A matrix $A$ is \emph{symplectic} if it preserves the standard symplectic form $\Omega_0$ in \eqref{eq:omega}, that is, $A^T\Omega_0A=\Omega_0$, and \emph{antisymplectic} if it inverts the form, that is, $A^T\Omega_0A=-\Omega_0$.

\begin{theorem}[{\cite[Theorem 2.4.2]{McDSal}}]\label{thm:rigidity-real}
	\letnpos. Let $A\in\M_{2n}(\R)$ be an invertible matrix such that $A$ and $A^{-1}$ are non-squeezing. Then $A$ is either symplectic or antisymplectic.
\end{theorem}

Next we introduce the corresponding $p$-adic notions and state the main theorem of this section (Theorem \ref{thm:rigidity}), which is the $p$-adic analogue of Theorem \ref{thm:rigidity-real}.

\begin{definition}[$p$-adic linear/affine ball and $p$-adic linear/affine cylinder]\label{def:symplectic-ball}
	\letnpos. \letpprime. Let $r,R$ be $p$-adic absolute values. A \textit{$2n$-dimensional $p$-adic linear symplectic ball of radius $r$} is the image $(f(\ball_p^{2n}(r)),\omega_0)$ of the $2n$-dimensional $p$-adic ball $\ball_p^{2n}(r)$ of radius $r$, endowed with $\omega_0=\sum_{i=1}^n\dd x_i\wedge\dd y_i$, by a $p$-adic linear symplectic embedding $f:\ball_p^{2n}(r)\hookrightarrow(\Qp)^{2n}$. A \textit{$2n$-dimensional $p$-adic linear symplectic cylinder of radius $R$} is the image $(f(\cyl_p^{2n}(R)),\omega_0)$ of the $2n$-dimensional $p$-adic cylinder $\cyl_p^{2n}(R)$ of radius $R$, endowed with $\omega_0$, by a $p$-adic linear symplectic embedding $f:\cyl_p^{2n}(R)\hookrightarrow(\Qp)^{2n}$. One can analogously define a \textit{$2n$-dimensional $p$-adic affine symplectic ball of radius $r$} and a \textit{$2n$-dimensional $p$-adic affine symplectic cylinder of radius $R$}, where the embeddings are affine instead of linear.
\end{definition}

\begin{figure}
	\includegraphics[trim=15cm 2cm 15cm 2cm,width=0.5\linewidth]{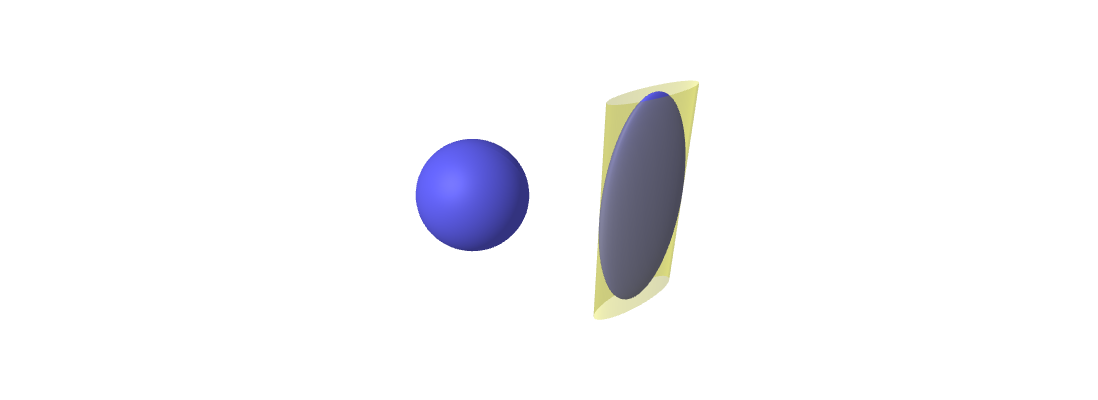}
	\caption{In Definition \ref{def:squeezing}, we allow any $p$-adic symplectic ball and $p$-adic symplectic cylinder to be used, not only the standard ball and cylinder, because otherwise it would exclude many cases of squeezing matrices such as the one shown in the figure, which squeezes the ball only in one direction. This is also the reason why we write $S_1AS_2v$ in the proof of Theorem \ref{thm:rigidity}.}
	\label{fig:squeezing-matrix-real}
\end{figure}

\begin{definition}[Squeezing matrix]\label{def:squeezing}
	\letnpos. \letpprime.
	We say that a matrix $A\in\M_{2n}(\Qp)$ is \emph{squeezing} if there exist $p$-adic absolute values $r,R$ such that $R<r$ and a $2n$-dimensional $p$-adic linear symplectic ball of radius $r$ whose image by $A$ is contained in a $2n$-dimensional $p$-adic linear symplectic cylinder of radius $R$. Otherwise we say that $A$ is \emph{non-squeezing}. See Figures \ref{fig:squeezing-matrix-real} and \ref{fig:squeezing-matrix-padic} for illustrations of this definition.
\end{definition}

\begin{theorem}[Characterization of squeezing matrices]\label{thm:rigidity}
	\letnpos. \letpprime. Let $A\in\M_{2n}(\Qp)$. Let $A\tr$ denote the transpose matrix of $A$. Let $\omega_0=\sum_{i=1}^n\dd x_i\wedge\dd y_i$ be the standard $p$-adic symplectic form on $(\Qp)^{2n}$, where $(x_1,y_1,\ldots,x_n,y_n)$ are the standard coordinates on $(\Qp)^{2n}$. Then the following statements are equivalent:
	\begin{enumerate}
		\renewcommand{\theenumi}{\roman{enumi}}
		\item $A$ is squeezing.
		\item There exist $u,v\in(\Qp)^{2n}$ such that
		\[|\omega_0(A\tr u,A\tr v)|_p\le\frac{|\omega_0(u,v)|_p}{p^2}.\]
	\end{enumerate}
\end{theorem}

\begin{figure}
	\includegraphics{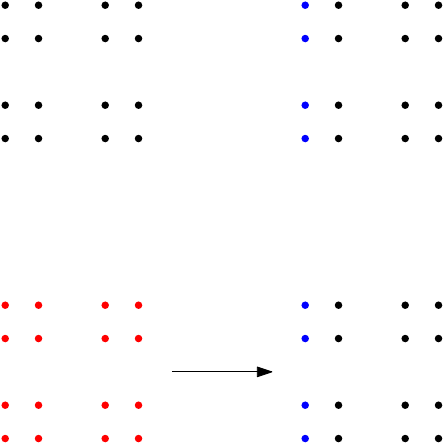}
	\caption{A $p$-adic example of the situation in Figure \ref{fig:squeezing-matrix-real}. The dots are balls of radius $1$ in a projection of $(\Q_2)^4$ to the coordinates $(x_1,y_1)$. A matrix that maps the red ball to the blue region can be squeezing in the sense of Definition \ref{def:squeezing}, because the blue region has smaller area, but it does not symplectically embed the $p$-adic ball into a $p$-adic cylinder of smaller radius.}
	\label{fig:squeezing-matrix-padic}
\end{figure}

\begin{proof}
	First we assume that $A$ is squeezing in the sense of Definition \ref{def:squeezing}. Then there exists a $2n$-dimensional $p$-adic linear symplectic ball whose image by $A$ is contained in a $2n$-dimensional $p$-adic linear symplectic cylinder of smaller radius. Let $r$ and $R$ be the radii of the ball and cylinder. Then we have that
	\[\Big\{S_1AS_2v:v\in \ball_p^{2n}(r)\Big\}\subset \cyl_p^{2n}(R)\]
	for some $p$-adic symplectic matrices $S_1$ and $S_2$. Also, $r>R$ implies that $r\ge pR$, because the radius of a $p$-adic ball (or cylinder) is always a power of $p$. Let \[C=S_1AS_2.\]
	
	For every $i\in\{1,\ldots,2n\}$, $\e_i/r$ is in $\ball_p^{2n}(r)$, hence \[\frac{C\e_i}{r}\in \cyl_p^{2n}(R)\] and
	\begin{equation}\label{eq:char-squeezing-1}
		|\e_1\tr C\e_i|_p =\frac{1}{r}|\e_1\tr \frac{C\e_i}{r}|_p\le \frac{R}{r}\le\frac{1}{p}.
	\end{equation}
	Since \eqref{eq:char-squeezing-1} happens for any $i\in\{1,\ldots,2n\}$, we have that
	\begin{equation}\label{eq:char-squeezing-2}
		\|C\tr \e_1\|_p\le \frac{1}{p}.
	\end{equation}
	The same as in \eqref{eq:char-squeezing-2} happens if we replace $\e_1$ by $\e_2$. This implies that
	\[|\omega_0(C\tr \e_1,C\tr \e_2)|_p\le\|C\tr \e_1\|_p\|C\tr \e_2\|_p\le\frac{1}{p^2}.\]
	Now we call $u=S_1\tr \e_1$ and $v=S_1\tr \e_2$. We have that $\omega_0(u,v)=1$ and
	\[|\omega_0(A\tr u,A\tr v)|_p=|\omega_0(S_2\tr A\tr S_1\tr \e_1,S_2\tr A\tr S_1\tr \e_2)|_p\le\frac{1}{p^2},\]
	as we wanted to prove.
	
	Now assume that there are $u$ and $v$ satisfying the condition (ii). Without loss of generality, we may assume that $\omega_0(u,v)=1$ (otherwise multiply $u$ by a constant). We define
	\[\left\{\begin{aligned}
		u' & =\frac{A\tr u}{p}; \\
		v' & =\frac{pA\tr v}{\omega_0(A\tr u,A\tr v)}.
	\end{aligned}\right.\]
	
	We have that $\omega_0(u,v)=1$ and
	\[\omega_0(u',v')=\frac{\omega_0(A\tr u,pA\tr v)}{p\omega_0(A\tr u,A\tr v)}=1,\]
	which means that there are $p$-adic symplectic matrices $S_1$ and $S_2$ such that
	\[S_1\tr \e_1=u,S_1\tr \e_2=v,S_2\tr u'=\e_1,S_2\tr v'=\e_2.\]
	Let \[C=S_1AS_2.\] We have that
	\[C\tr \e_1=S_2\tr A\tr S_1\tr \e_1=S_2\tr A\tr u=pS_2\tr u'=p\e_1\]
	and
	\begin{align*}
		C\tr \e_2 & =S_2\tr A\tr S_1\tr \e_2 \\
		& =S_2\tr A\tr v \\
		& =\frac{\omega_0(A\tr u,A\tr v)}{p}S_2\tr v' \\
		& =\frac{\omega_0(A\tr u,A\tr v)}{p}\e_2.
	\end{align*}
	Now we show that $C$ sends $\ball_p^{2n}(r)$ to $\cyl_p^{2n}(r/p)$, which will imply that $A$ is squeezing. Let $w\in \ball_p^{2n}(r)$. We want to show that $Cw\in \cyl_p^{2n}(r/p)$, that is,
	\begin{equation}\label{eq:char-squeezing-3}
		|\e_1\tr Cw|_p\le\frac{r}{p}\text{ and }|\e_2\tr Cw|_p\le\frac{r}{p}.
	\end{equation}
	The first inequality in \eqref{eq:char-squeezing-3} holds because
	\[|\e_1\tr Cw|_p=|p\e_1\tr w|_p\le\frac{r}{p}\]
	and the second one holds because
	\begin{align*}
		|\e_2\tr Cw|_p & =\left|\frac{\omega_0(A\tr u,A\tr v)}{p}\e_2\tr w\right|_p \\
		& =\frac{|\omega_0(A\tr u,A\tr v)|_p}{|p|_p}|\e_2\tr w|_p \\
		& \le\frac{p^{-2}}{p^{-1}}r \\
		& =\frac{r}{p}.\qedhere
	\end{align*}
\end{proof}

\begin{corollary}\label{cor:rigidity}
	\letnpos. \letpprime. Let $\Omega_0$ be the matrix \eqref{eq:omega} of the standard symplectic form $\sum_{i=1}^n\dd x_i\wedge\dd y_i$ on $(\Qp)^{2n}$, where $(x_1,y_1,\ldots,x_n,y_n)$ are the standard coordinates on $(\Qp)^{2n}$. Let $A\in\M_{2n}(\Qp)$ be an invertible matrix. Let $A\tr$ denote the transpose matrix of $A$. Then the following statements are equivalent:
	\begin{enumerate}
		\renewcommand{\theenumi}{\roman{enumi}}
		\item The matrices $A$ and $A^{-1}$ are both non-squeezing.
		\item There exists $c\in\Qp$ such that $\ord_p(c)\in\{-1,0,1\}$ and $A\Omega_0A\tr =c\Omega_0$.
	\end{enumerate}
\end{corollary}

\begin{proof}
	If $c$ exists with the required condition, then
	\[\omega_0(A\tr u,A\tr v)=c\omega_0(u,v).\]
	By Theorem \ref{thm:rigidity} for $A$ and $A^{-1}$, since the order of $c$ is in $\{-1,0,1\}$, both matrices are non-squeezing.
	
	Now suppose that $A$ and $A^{-1}$ are non-squeezing. By Theorem \ref{thm:rigidity},
	\begin{equation}\label{eq:rigidity}
		\frac{1}{p^2}<\frac{|\omega_0(A\tr u,A\tr v)|_p}{|\omega_0(u,v)|_p}<p^2
	\end{equation}
	for all $u,v\in(\Qp)^{2n}$. Let $\Omega_1=A\Omega_0A\tr $. Equation \eqref{eq:rigidity} can be written
	\[\frac{1}{p^2}<\frac{|u\tr \Omega_1 v|_p}{|u\tr \Omega_0 v|_p}<p^2,\]
	which is equivalent to
	\[-1\le\ord_p(u\tr \Omega_1 v)-\ord_p(u\tr \Omega_0 v)\le 1.\]
	In particular this implies that $u\tr \Omega_0 v$ and $u\tr \Omega_1 v$ are $0$ exactly for the same values of $u$ and $v$, which is only possible if $\Omega_1=c\Omega_0$ for some constant $c$. In turn, this $c$ must have order between $-1$ and $1$.
\end{proof}

\begin{remark}
	The relation $A\Omega_0A\tr =c\Omega_0$ is equivalent to $A\tr \Omega_0A=c\Omega_0$, because
	\begin{align*}
		A\Omega_0A\tr=c\Omega_0 & \Leftrightarrow (A\tr )^{-1}(-\Omega_0)A^{-1}=c^{-1}(-\Omega_0) \\
		& \Leftrightarrow c\Omega_0=A\tr \Omega_0 A.
	\end{align*}
\end{remark}

\section{$p$-adic linear symplectic width and the maps which preserve it}\label{sec:width}

In this section we study the concept of $p$-adic linear symplectic width. \letnpos. We characterize which matrices preserve the $p$-adic linear symplectic width of any subset of the $p$-adic space $(\Qp)^{2n}$ and compute the $p$-adic linear symplectic width of ellipsoids.

\subsection{General definitions and results}

The following definition is analogous to McDuff-Salamon \cite[page 56]{McDSal}. In what follows we denote by $\ASp((\Qp)^{2n})$ the group of affine symplectomorphisms of $(\Qp)^{2n}$ under composition.

\begin{definition}[$p$-adic linear symplectic width]\label{def:width}
	\letnpos. \letpprime. Let $X$ be a subset of $(\Qp)^{2n}$. We define the \emph{$2n$-dimensional $p$-adic linear symplectic width} $\linearwidth(X)$ of $X$ as the square of the radius of the $2n$-dimensional $p$-adic affine symplectic ball (as in Definition \ref{def:symplectic-ball}) of smallest radius which is contained in $X$, that is,
	\begin{equation}\label{eq:width}
		\linearwidth(X)=\sup\Big\{r^2:\psi(\ball_p^{2n}(r))\subset X\text{ for some }\psi\in\ASp((\Qp)^{2n})\Big\}.
	\end{equation}
\end{definition}

Note that, in Definition \ref{def:width}, $X$ is any subset of $(\Qp)^{2n}$, and no additional structure on $X$ is required, since $\linearwidth(X)$ is defined using global $p$-adic analytic symplectomorphisms of $(\Qp)^{2n}$.

A difference between Definition \ref{def:width} and the analogous definition in the real case is that here the supremum is a maximum whenever it is finite.

\begin{proposition}
	If the supremum in expression \eqref{eq:width} is finite, then it is a maximum.
\end{proposition}

\begin{proof}
	Let $S$ be the set
	\[\Big\{r^2:\psi(\ball_p^{2n}(r))\subset X\text{ for some }\psi\in\ASp((\Qp)^{2n})\Big\}.\]
	Since every $p$-adic absolute value is a power of $p$, we have $S\subset\Q$. By the supremum axiom in $\R$, $S$ has a supremum, which may be infinity or a real number.
	
	Now suppose this supremum is finite and not a maximum. Let $s\in\R$ be the supremum. Then $s\notin S$. By definition, $s>0$. Let \[k=\lceil\log_p s\rceil-1,\] where $\lceil x\rceil$ denotes the smallest integer greater or equal to $x$. We have $k<\log_p s$, that is, $p^k<s$, and $s$ is the supremum of $S$, so there is an element $s'\in S$ with $p^k<s'<s$. Since all the elements in $S$ are powers of $p$, we have $s'=p^l$ for some $l\in\Z$. Now $s'=p^l>p^k$ implies $l>k$, hence $l\ge k+1$ and
	\[s=p^{\log_p s}\le p^{\lceil\log_p s\rceil}=p^{k+1}\le p^l=s'<s,\]
	which is a contradiction.
\end{proof}

\begin{proposition}\label{prop:width}
	Let $n$ be a positive integer. \letpprime. The $2n$-dimensional $p$-adic linear symplectic width given by \eqref{eq:width} has the following properties:
	\begin{enumerate}
		\renewcommand{\theenumi}{\roman{enumi}}
		\item \emph{Monotonicity}: if $X, Y$ are subsets of $(\Qp)^{2n}$ and there exists $\psi\in\ASp((\Qp)^{2n})$ such that $\psi(X)\subset Y$, then $\linearwidth(X)\le \linearwidth(Y)$.
		\item \emph{Conformality}: if $X$ is a subset of $(\Qp)^{2n}$ and $c\in\Qp$, then $\linearwidth(cX)=|c|_p^2\linearwidth(X)$.
		\item \emph{Non-triviality}: \[\linearwidth(\ball_p^{2n}(1))=\linearwidth(\cyl_p^{2n}(1))=1.\]
	\end{enumerate}
\end{proposition}

\begin{proof}
	The first two properties and the first part of the third are immediate consequences of the definition. The second part of the third property follows from Theorem \ref{thm:linear}.
\end{proof}

The following result is the $p$-adic analog of McDuff-Salamon \cite[Theorem 2.4.4]{McDSal}.

\begin{theorem}\label{thm:width}
	Let $n$ be a positive integer. \letpprime. Let $\Omega_0$ be the matrix \eqref{eq:omega} of the standard symplectic form $\sum_{i=1}^n\dd x_i\wedge\dd y_i$ on $(\Qp)^{2n}$, where $(x_1,y_1,\ldots,x_n,y_n)$ are the coordinates on $(\Qp)^{2n}$. Let $A\in\M_{2n}(\Qp)$. Then the following statements are equivalent:
	\begin{enumerate}
		\item For every subset $X$ of $(\Qp)^{2n}$ we have that \[\linearwidth(\{Av:v\in X\})=\linearwidth(X).\]
		\item There exists $c\in\Zp$ with $\ord_p(c)=0$ and $A\tr \Omega_0A=c\Omega_0$.
	\end{enumerate}
\end{theorem}

\begin{proof}
	Assume first that (2) holds and let $X\subset(\Qp)^{2n}$ and \[Y=\{Av:v\in X\}.\] $A$ must be invertible, because otherwise there would be $v\ne 0$ such that $Av=0$ and (2) would not hold because $\omega_0$ is non-degenerate. We want to see that $\linearwidth(X)=\linearwidth(Y)$. It is enough to prove that $\linearwidth(X)\le \linearwidth(Y)$, because for the other direction one only needs to change $c$ by $1/c$, which also has order $0$.
	
	Let $r\in\Q$ be such that $r^2=\linearwidth(X)$ and let $\psi\in\ASp(2n,\Qp)$ be such that $\psi(\ball_p^{2n}(r))\subset X$. Let
	\[C=\begin{pmatrix}
		1 & & & & & & \\
		& c & & & & & \\
		& & 1 & & & & \\
		& & & c & & & \\
		& & & & \ddots & & \\
		& & & & & 1 & \\
		& & & & & & c
	\end{pmatrix}\]
	This matrix keeps the ball invariant and satisfies
	\[\omega_0(Cu,Cv)=c\omega_0(u,v)\]
	for all $u,v\in(\Qp)^{2n}$. Hence the $p$-adic affine map $\psi'$ given by \[\psi'(v)=A\psi(C^{-1}v)\] is a $p$-adic affine symplectomorphism, and $\psi'(\ball_p^{2n}(r))\subset Y$. This proves that \[\linearwidth(X)\le \linearwidth(Y).\]
	
	Now assume that (1) holds. Then $A$ is invertible (otherwise it would make the ball into a set of zero width) and $A^k$ also preserves the $p$-adic linear symplectic width for any $k\in\Z$. We claim that $A^k$ is non-squeezing.
	
	In order to prove this, let $B$ be a $p$-adic symplectic ball of radius $r$ as in Definition \ref{def:symplectic-ball} and \[B'=\{A^kv:v\in B\}.\] Since $A^k$ preserves the $p$-adic linear symplectic width, \[\linearwidth(B')=\linearwidth(B)=r^2.\] If $B'$ is contained in a $p$-adic symplectic cylinder $Z$ as in Definition \ref{def:symplectic-ball}, \[\linearwidth(Z)\ge r^2,\] which implies by Proposition \ref{prop:width}(3) that the radius of the cylinder is at least $r$. The claim is proved.
	
	By Corollary \ref{cor:rigidity}, we have that, for each $k\in\Z$, there exists $c_k\in\Qp$ such that $(A^k)\tr \Omega_0A^k=c_k\Omega_0$ and $\ord_p(c_k)\in\{-1,0,1\}$. Now we have
	\[c_2\Omega_0=(A^2)\tr \Omega_0A^2=c_1A\tr \Omega_0A=c_1^2\Omega_0,\]
	that is, $c_2=c_1^2$. Together with $\ord_p(c_1)\in\{-1,0,1\}$ and $\ord_p(c_2)\in\{-1,0,1\}$, this implies \[\ord_p(c_1)=0,\] and we are done.
\end{proof}

\subsection{$p$-adic linear symplectic width of $p$-adic ellipsoids}

Next we compute the $p$-adic linear symplectic width of $p$-adic ellipsoids. Here by a $p$-adic ellipsoid we mean the image by a $p$-adic affine map of the $p$-adic ball. See Figure \ref{fig:ellipsoids} for two examples. The following definition generalizes the one given in \eqref{eq:ellipsoid}.

\begin{definition}\label{def:ellipsoid}
	Let $m$ be a positive integer. \letpprime.
	An \emph{$m$-dimensional $p$-adic ellipsoid} is any subset of $(\Qp)^m$ of the form
	\begin{equation}\label{eq:ellipsoid}
		E=\Big\{v\in(\Qp)^m:\|A(v-v_0)\|_p\le 1\Big\}=\Big\{v\in(\Qp)^m:\|A(v-v_0)\|_p<p\Big\},
	\end{equation}
	where $A\in\M_m(\Qp)$ is an invertible matrix and $v_0\in(\Qp)^m$. We call $A$ \emph{a matrix which defines $E$}, or a \emph{defining matrix of $E$}.
\end{definition}

\begin{remark}
	The matrix $A$ which defines an ellipsoid as in \eqref{eq:ellipsoid} is not unique. For example, adding to a row an integer linear combination of the rest of rows does not change the ellipsoid.
\end{remark}

\begin{remark}
	For any prime number $p$, a $p$-adic affine symplectic ball (Definition \ref{def:symplectic-ball}) is a particular example of a $p$-adic ellipsoid (Definition \ref{def:symplectic-ball}) when $A$ is a multiple of a symplectic matrix. In turn, a $p$-adic linear symplectic ball is an example of a $p$-adic affine symplectic ball when $v_0=0$.
\end{remark}

\begin{remark}
	As it happens with $p$-adic balls, changing $v_0$ by any other point in expression \eqref{eq:ellipsoid} gives the same $p$-adic ellipsoid.
\end{remark}

\begin{figure}
	\includegraphics[scale=2.6]{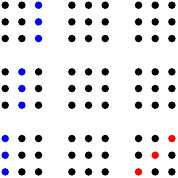}
	\caption{Two $2$-dimensional $3$-adic ellipsoids, represented as subsets of $\ball_3^2(9)$, where each dot represents a $3$-adic ball of radius $1$. The red one has width $1$, because a $3$-adic ball of radius $1$ fits in the $3$-adic ellipsoid, but one of radius $3$ (which is the next possible radius) has bigger volume. The blue $3$-adic ellipsoid has width $9$, because we can transform it to a $3$-adic ball of radius $3$ by means of a $p$-adic affine symplectomorphism, for example the nine dots in the upper left, leaving the $x$ coordinate unchanged and adding $1/3$ of it to $y$.}
	\label{fig:ellipsoids}
\end{figure}

Recall that $\Zp$ is the set of $p$-adic numbers of absolute value at most $1$:
\[\Zp=\Big\{x\in\Qp:|x|_p\le 1\Big\}=\Big\{x\in\Qp:\ord_p(x)\ge 0\Big\}.\]

We start by proving a $p$-adic version of Farkas' Lemma \cite{Farkas}, which is also of independent interest. This result will allow us to find a necessary and sufficient condition for a $p$-adic ellipsoid to be contained in another one.

\begin{lemma}[$p$-adic analog of Farkas' Lemma]\label{lemma:farkas}
	Let $m$ amd $k$ be positive integers. \letpprime. Let $A\in\M_{m\times k}(\Qp)$ and $b\in(\Qp)^m$. Then one and only one of the following two vectors exists:
	\begin{itemize}
		\item $x\in(\Zp)^k$ such that $Ax=b$.
		\item $y\in(\Qp)^m$ such that $y\tr A\in(\Zp)^k$ and $y\tr b\notin\Zp$.
	\end{itemize}
\end{lemma}

\begin{proof}
	First we show that $x$ and $y$ cannot exist at the same time: if they exist,
	\[y\tr Ax\in\Zp\text{ and }y\tr Ax=y\tr b\notin\Zp,\]
	which is a contradiction.
	
	Now we prove that $x$ or $y$ exists by induction on $m$. For $m=1$, there are two possibilities:
	\begin{itemize}
		\item If $|b|_p\le\|A\|_p$, there is an integer combination of the entries in $A$ which gives $b$, and $x$ exists.
		\item If $|b|_p>\|A\|_p$, we can take $y=\|A\|_p$, and
		\[\|(\|A\|_pA)\|_p=\frac{\|A\|_p}{\|A\|_p}=1;\]
		\[|(\|A\|_pb)|_p=\frac{|b|_p}{\|A\|_p}>1.\]
	\end{itemize}
	
	Now suppose that it holds for $m\ge 1$ and we prove it for $m+1$. We distinguish four cases.
	
	\begin{itemize}
		\item \emph{Case 1: The first row of $A$ is full of zeros and the first entry of $b$ is zero.} We apply induction to $A$ and $b$ with the first row removed. If $x$ exists, this $x$ is still a solution of $Ax=b$ after adding a row of zeros; if $y$ exists, we add a zero coordinate at the beginning of $y$.
		\item \emph{Case 2: The first row of $A$ is full of zeros and the first entry of $b$ is nonzero.} We take $y=p^{-\ell}\e_1$ for $\ell$ big enough.
		\item \emph{Case 3: $k=1$ and the first entry of $A$ is nonzero.} Let $\alpha$ and $\beta$ be the first entries of $A$ and $b$, and \[x=\frac{\beta}{\alpha}.\] If $x\in\Zp$ and $Ax=b$, we are done. If $x\notin\Zp$, we take \[y=\frac{\e_1}{\alpha},\] and now $y\tr A=1$ and $y\tr b=x\notin\Zp$. If $Ax\ne b$, let $i$ be a position different from zero in $Ax-b$. We take $\ell$ big enough so that \[p^{-\ell}\alpha \e_i\tr (Ax-b)\notin\Zp\] and
		\[y=p^{-\ell}(\e_1\e_i\tr A-\e_i\e_1\tr A),\]
		and obtain
		\[y\tr A=p^{-\ell}(\e_1\tr A\e_i\tr A-\e_i\tr A\e_1\tr A)=0,\]
		\begin{align*}
			y\tr b & =p^{-\ell}(\e_1\tr b\e_i\tr A-\e_i\tr b\e_1\tr A) \\
			& =p^{-\ell}(\beta \e_i\tr A-\alpha \e_i\tr b) \\
			& =p^{-\ell}\alpha \e_i\tr (Ax-b)\notin\Zp.
		\end{align*}
		\item \emph{Case 4: $k\ge 2$ and there is at least a nonzero entry in the first row of $A$.} Let $(1,j)$ be the position in $A$ of the entry with smallest order in the first row and let $A_1$ be the result of swapping the columns $1$ and $j$ in $A$.
		
		Now we add an integer multiple of the first column of $A_1$ to the rest of columns to make $0$ the rest of the first row; let $A_2$ be the resulting matrix. By construction, $A_2=AC$ for some integer matrix $C$ whose inverse is also integer. There exist $\alpha,\beta\in\Qp$, $a\in(\Qp)^m$, $b'\in(\Qp)^m$ and $A'\in\M_{m\times(k-1)}(\Qp)$ such that
		\[A_2=\begin{pmatrix}
			\alpha & 0 \\
			a & A'
		\end{pmatrix},
		b=\begin{pmatrix}
			\beta \\ b'
		\end{pmatrix}.\]
		Let \[x_1=\frac{\beta}{\alpha}.\] If $x_1\notin\Zp$, we take \[y=\frac{\e_1}{\alpha},\] which gives \[y\tr A=y\tr A_2C^{-1}=\e_1C^{-1}\in(\Zp)^n\] and $y\tr b=x_1\notin\Zp$. If $x_1\in\Zp$, we apply the inductive hypothesis to the matrix $A'$, which has $m$ rows, and the vector $b'-x_1a$. There are two possibilities:
		\begin{itemize}
			\item There is $x'\in(\Zp)^{k-1}$ such that $A'x'=b'-x_1a$. Then $x_2=(x_1,x')\in(\Zp)^k$ satisfies $A_2x_2=b$ and $ACx_2=b$, hence we can take $x=Cx_2$.
			\item There is $y'\in(\Qp)^m$ such that ${y'}\tr A'\in(\Zp)^{k-1}$ and ${y'}\tr (b'-x_1a)\notin\Zp$. We take $y=(-{y'}\tr a/\alpha,y')$ and have
			\begin{align*}
				y\tr A & =y\tr A_2C^{-1} \\
				& =(-{y'}\tr a+{y'}\tr a,{y'}\tr A')C^{-1}\in(\Zp)^k
			\end{align*}
			and
			\begin{align*}
				y\tr b & =-x_1{y'}\tr a+{y'}\tr b' \\
				& ={y'}\tr (b'-x_1a)\notin\Zp,
			\end{align*}
			as we wanted. \qedhere
		\end{itemize}
	\end{itemize}
\end{proof}

With the help of Lemma \ref{lemma:farkas} we can now prove the following.

\begin{proposition}\label{prop:contained}
	Let $m$ be a positive integer. \letpprime. Let $E_1$ and $E_2$ be two $m$-dimensional $p$-adic ellipsoids in $(\Qp)^m$ as in Definition \ref{def:ellipsoid} and let $E_i$ be defined by \[\|A_i(v-v_i)\|\le 1,\text{ for }i\in\{1,2\}.\] Then the following statements are equivalent:
	\begin{enumerate}
		\item $E_1\subset E_2$.
		\item $v_1\in E_2$ and there exists a matrix $C\in\M_m(\Zp)$ such that $CA_1=A_2$.
	\end{enumerate}
\end{proposition}

\begin{proof}
	Assume that (2) holds and let $v\in E_1$. This means that $A_1(v-v_1)\in(\Zp)^m$. From $v_1\in E_2$, we get that $A_2(v_1-v_2)\in(\Zp)^m$ and
	\begin{align*}
		A_2(v-v_2) & =A_2(v-v_1)+A_2(v_1-v_2) \\
		& =CA_1(v-v_1)+A_2(v_1-v_2)\in(\Zp)^m,
	\end{align*}
	that is, $v\in E_2$, and $E_1\subset E_2$.
	
	Assume now that $E_1\subset E_2$. Then $v_1\in E_2$. To prove the other part of (2), let $b$ be a row of $A_2$. We apply Lemma \ref{lemma:farkas} to $A_1\tr $ and $b$. There are two possibilities: it gives $x\in(\Zp)^m$ such that $x\tr A_1=b$, or it gives $y\in(\Qp)^m$ such that $A_1y\in(\Zp)^m$ and $by\notin\Zp$.
	
	Suppose for a contradiction that $y$ exists. Let $v=v_1+y$. Then $A_1(v-v_1)=A_1y\in(\Zp)^m$ and $v\in E_1$, which implies $v\in E_2$. But now \[by=b(v-v_1)=b(v-v_2)-b(v_1-v_2)\] and both are in $\Zp$ because $v$ and $v_1$ are both in $E_2$, so $by\in\Zp$, which is a contradiction. Hence it must be $x$ which exists.
	
	This means that, for every $i\in\{1,\ldots,m\}$, there is $x_i\in(\Zp)^m$ such that $x_i\tr A_1$ gives the $i$-th row if $A_2$. Then the matrix $C$ with the $x_i$'s as rows satisfies \[CA_1=A_2,\] and we are done.
\end{proof}

Now we apply Proposition \ref{prop:contained} to the case where one of the $p$-adic ellipsoids is a $p$-adic symplectic ball.

\begin{proposition}\label{prop:ball-contained}
	\letnpos. \letpprime. Let $\Omega_0$ be the matrix \eqref{eq:omega} of the standard $p$-adic symplectic form $\sum_{i=1}^n\dd x_i\wedge\dd y_i$, where $(x_1,y_1,\ldots,x_n,y_n)$ are the standard coordinates on $(\Qp)^{2n}$. A $2n$-dimensional $p$-adic ellipsoid $E\subset(\Qp)^{2n}$, with defining matrix $A$ as in Definition \ref{def:ellipsoid}, contains a $2n$-dimensional $p$-adic affine symplectic ball of radius $1$, as in Definition \ref{def:symplectic-ball}, if and only if every entry of $A\Omega_0A\tr $ is in $\Zp$.
\end{proposition}

\begin{proof}
	Assume that there is a $p$-adic affine symplectic ball $B$ of radius $1$ such that $B\subset E$. Let $S$ be the matrix of $B$, which is $p$-adic symplectic. By Proposition \ref{prop:contained}, there exists a matrix $C\in\M_{2n}(\Zp)$ such that $A=CS$ and
	\[A\Omega_0A\tr =CS\Omega_0S\tr C\tr =C\Omega_0C\tr \in\M_{2n}(\Zp).\]
	
	Now assume that $A\Omega_0A\tr \in\M_{2n}(\Zp).$ Let $D$ be this matrix, which is antisymmetric. We apply a Gaussian reduction to $D$ based on the following three transformations:
	\begin{enumerate}
		\item Dividing a row and a column with the same index by a constant in $\Zp$.
		\item Adding a row multiplied by a constant in $\Zp$ to another row, and do the same operation to the columns.
		\item Exchanging two rows and the corresponding columns.
	\end{enumerate}
	We can apply a sequence of transformations of these three types from $D$ to $\Omega_0$. Indeed, we first exchange rows and columns so that the element in position $(1,2)$ is nonzero. Next we divide the first row and column by that element, so that it becomes $1$ and that in position $(2,1)$ becomes $-1$. After this, we add multiples of the first and second rows to the rest of rows, and the same with the columns, to make $0$ the rest of these rows and columns. The result is
	\[\begin{pmatrix}
		0 & 1 & 0 \\
		-1 & 0 & 0 \\
		0 & 0 & D'
	\end{pmatrix}\]
	for an antisymmetric matrix $D'$. Now we apply iteratively the process to $D'$ until we end up with $\Omega_0$.
	
	If we reverse the sequence, we go from $\Omega_0$ to $D$ by a sequence of transformations of the same three types, except that in type (1) we multiply instead of dividing. Each type corresponds to changing a $p$-adic matrix $M\in\M_{2n}(\Zp)$ to the $p$-adic matrix $NMN\tr $ for some $N\in\M_{2n}(\Zp)$. Hence, the whole sequence corresponds to going from $\Omega_0$ to $C\Omega_0C\tr $ for some $C\in\M_{2n}(\Zp)$. This means that
	\[A\Omega_0A\tr =D=C\Omega_0C\tr \Longrightarrow C^{-1}A\Omega_0(C^{-1}A)\tr =\Omega_0.\]
	Let \[S=C^{-1}A.\] We have that $S$ is $p$-adic symplectic. We take an arbitrary $v_0\in E$ and define
	\[B=\{v\in(\Qp)^n:\|S(v-v_0)\|_p\le 1\}.\]
	Since $v_0\in E$ and $CS=A$ for $C\in\M_{2n}(\Zp)$, Proposition \ref{prop:contained} implies that $B\subset E$.
\end{proof}

Proposition \ref{prop:ball-contained} has as a consequence the characterization of the $p$-adic linear symplectic width of ellipsoids. The real equivalent is \cite[Theorem 2.4.8]{McDSal}.

\begin{corollary}\label{cor:width-ellipsoid}
	\letnpos. \letpprime. Let $\Omega_0$ be the matrix \eqref{eq:omega} of the standard $p$-adic symplectic form $\sum_{i=1}^n\dd x_i\wedge\dd y_i$, where $(x_1,y_1,\ldots,x_n,y_n)$ are the standard coordinates on $(\Qp)^{2n}$. Let $E\subset(\Qp)^{2n}$ be a $2n$-dimensional $p$-adic ellipsoid with defining matrix $A$, as in Definition \ref{def:ellipsoid}. The $2n$-dimensional $p$-adic linear symplectic width of $E$ is the highest power of $p^2$ which simultaneously divides all entries of $A\Omega_0A\tr $.
\end{corollary}

\begin{proof}
	We may assume that $v_0=0$, since translating the $p$-adic ellipsoid will not change the $p$-adic linear symplectic width.
	
	Let $k$ be the highest integer such that $p^{2k}$ divides all entries of $A\Omega_0 A\tr$. Let $E'$ be the $p$-adic ellipsoid given by \[\|p^{-k}A\|_p\le 1.\] By our choice of $k$, the matrix \[p^{-k}A \Omega_0 p^{-k}A\tr=p^{-2k}A\Omega_0 A\tr\] is an integer matrix, and by Proposition \ref{prop:ball-contained}, $E'$ contains a $p$-adic symplectic ball of radius $1$ and
	\begin{equation}\label{eq:width-ellipsoid}
		\linearwidth(E')\ge 1.
	\end{equation}
	
	If \eqref{eq:width-ellipsoid} is an equality, then we can apply property (2) in Proposition \ref{prop:width} and conclude that \[\linearwidth(E)=p^{2k},\] as we wanted. Suppose otherwise that the width of $E'$ is not $1$. Then it must be at least $p^2$. Let $E''$ be the ellipsoid given by \[\|p^{-k-1}A\|_p\le 1.\] Again by property (2) in Proposition \ref{prop:width}, \[\linearwidth(E'')\ge 1,\] that is, a $p$-adic ball of radius $1$ is contained in $E''$. By Proposition \ref{prop:ball-contained}, \[p^{-k-1}A\Omega_0p^{-k-1}A\tr\] is an integer matrix and all entries of $A\Omega_0A\tr$ are divisible by $p^{2k+2}$, contradicting our choice of $k$.
\end{proof}

\section{A construction of $p$-adic polar coordinates}\label{sec:polar}

In order to prove the equivariant version of the $p$-adic symplectic squeezing theorem, Theorem \ref{thm:total-embedding}, which we do in the following section, we need a $p$-adic version of the polar coordinates on the plane $(\Qp)^2$. We start proving some preparatory results concerning modular arithmetic.

\begin{proposition}[{see for example \cite[Corollary A.4]{CrePel-williamson}}]
	\letpprime.
	\begin{enumerate}
		\item If $p\ne 2$, a $p$-adic number is a square in $\Qp$ if and only if it has even order and its leading digit is a square modulo $p$.
		\item If $p=2$, a $p$-adic number is a square in $\Qp$ if and only if it has even order and its three rightmost digits are $001$.
	\end{enumerate}
\end{proposition}

\begin{definition}\label{def:D}
	\begin{enumerate}
		\item \letpprime\ such that $p\ne 2$. We define the set \[\mathrm{Z}_p=\F_p^*\] and
		\[\mathrm{D}_p(c)=\{(a,b)\in(\F_p)^2:a^2+b^2=c\}\]
		for $c\in \mathrm{Z}_p$.
		\item We define the set $\mathrm{Z}_2=\{1,2,5\}$ and
		\[\mathrm{D}_2(c)=\{(a,b)\in(\Z/4\Z)^2:a^2+b^2\equiv c\mod 8\},\]
		for $c\in \mathrm{Z}_2$, that is,
		\[\mathrm{D}_2(1)=\{(0,1),(0,3),(1,0),(3,0)\},\]
		\[\mathrm{D}_2(2)=\{(1,1),(1,3),(3,1),(3,3)\},\]
		\[\mathrm{D}_2(5)=\{(2,1),(2,3),(1,2),(3,2)\}.\]
	\end{enumerate}
	We also define
	\[\mathrm{D}_p=\bigcup_{c\in \mathrm{Z}_p}\mathrm{D}_p(c).\]
	If $p\equiv 1\mod 4$, let $\ii_0\in\F_p$ be such that $\ii_0^2=-1$ and define
	\begin{align*}
		\mathrm{D}_p^+(0) & =\{(a,\ii_0a):a\in\F_p^*\}, \\
		\mathrm{D}_p^-(0) & =\{(a,-\ii_0a):a\in\F_p^*\}.
	\end{align*}
	Otherwise we consider $\mathrm{D}_p^+(0)=\mathrm{D}_p^-(0)=\varnothing$.
	
	We define on $\mathrm{D}_p\cup \mathrm{D}_p^+(0)\cup \mathrm{D}_p^-(0)$ the binary operation given by
	\[(a,b)\cdot(a',b')=(aa'-bb',ab'+a'b).\]
\end{definition}

\begin{proposition}\label{prop:D}
	\letpprime. Let $\mathrm{Z}_p$ and $\mathrm{D}_p(c)$ be the sets in Definition \ref{def:D}. Then the following statements hold.
	\begin{enumerate}
		\item For any $c\in \mathrm{Z}_p$, multiplying an element of $\mathrm{D}_p(c)$ by all elements of $\mathrm{D}_p(1)$ results in all elements of $\mathrm{D}_p(c)$.
		\item If $p\equiv 1\mod 4$, the same happens with $\mathrm{D}_p^+(0)$ and $\mathrm{D}_p^-(0)$ instead of $\mathrm{D}_p(c)$.
	\end{enumerate}
\end{proposition}

\begin{proof}
	We can consider elements of $\mathrm{D}_p$ as matrices of the form
	\[\begin{pmatrix}
	a & -b \\
	b & a
	\end{pmatrix}.\]
	As such, these $p$-adic matrices form a group whose product is exactly that of $\mathrm{D}_p$, and $\mathrm{D}_p(1)$ is the subgroup of matrices with determinant $1$. Since the determinant is multiplicative, these matrices multiplied by one of $\mathrm{D}_p(1)$ gives the whole $\mathrm{D}_p(c)$. If $p\equiv 1\mod 4$, we can associate to each element of $\mathrm{D}_p$ the quantity $a+\ii_0 b$. This quantity is multiplicative, and it is $0$ if and only if $(a,b)\in \mathrm{D}_p^+(0)$, hence the product of something in $\mathrm{D}_p^+(0)$ times anything must give something in $\mathrm{D}_p^+(0)$. For $\mathrm{D}_p^-(0)$, we can make the same argument with $a-\ii_0 b$.
\end{proof}

\begin{definition}
	\letpprime\ and $c\in \mathrm{Z}_p$. We denote by \[(a_p(c),b_p(c))\] the lowest element of $\mathrm{D}_p(c)$ in lexicographic order. If $p\equiv 1\mod 4$, we also denote \[(a_p^+(0),b_p^+(0))\] and \[(a_p^-(0),b_p^-(0))\] the lowest elements of $\mathrm{D}_p^+(0)$ and $\mathrm{D}_p^-(0)$ in lexicographic order, respectively.
\end{definition}

\begin{definition}\label{def:T}
	\letnpos. \letpprime. We define the open subset of $(\Qp)^2$:
	\[\mathrm{T}_p^{2n}=\Big\{(x,y)\in(\Qp)^2:(x,y)\ne(0,0)\Big\}^n.\]
	The set $\mathrm{T}_p^{2n}$ endowed with the standard $p$-adic symplectic form $\sum_{i=1}^n\dd x_i\wedge\dd y_i$, where $(x_1,y_1,\ldots,x_n,y_n)$ are the standard coordinates on $(\Qp)^{2n}$, is a $p$-adic analytic symplectic manifold of dimension $2n$.
\end{definition}

\begin{example}
	For $n=1$, the set $\mathrm{T}_p^2$ is the plane $(\Qp)^2$ minus one point. For $n=2$, $\mathrm{T}_p^4$ is the space $(\Qp)^4$ minus two $2$-dimensional planes which meet at a point.
\end{example}

\begin{proposition}\label{prop:D2}
	\letpprime. Let $d=2$ if $p=2$, and otherwise $d=1$. Let $\mathrm{T}_p^2$ be the set in Definition \ref{def:T}. Let $(x,y)\in \mathrm{T}_p^2$. Let \[k=\min\{\ord_p(x),\ord_p(y)\}\] and let $x_0$ and $y_0$ be the digits at order $k$ of $x$ and $y$, if $p\ne 2$, and the digits of order $k$ and $k+1$ taken together as a number in $\{0,1,2,3\}$ if $p=2$. Let $z_0$ be the digit of order $2k$ of $x^2+y^2$, if $p\ne 2$, or the three digits of order $2k$, $2k+1$ and $2k+2$ taken together if $p=2$.
	\begin{enumerate}
		\item If $z_0\ne 0$, then $z_0\in \mathrm{Z}_p$ and $(x_0,y_0)\in \mathrm{D}_p(z_0)$.
		\item If $z_0=0$, then $(x_0,y_0)\in \mathrm{D}_p^+(0)\cup \mathrm{D}_p^-(0)$.
	\end{enumerate}
\end{proposition}

\begin{proof}
	Assume first that $z_0\ne 0$. The fact that $z_0\in \mathrm{Z}_p$ is true by definition if $p\ne 2$, and if $p=2$, the one between $x^2$ and $y^2$ with smaller order ends in $001$ and the other can have $000$, $001$ or $100$ at that position, hence the possible values for $z_0$ are $1$, $2$ and $5$. The fact that $(x_0,y_0)\in \mathrm{D}_p(z_0)$ follows from
	\[z_0\equiv x^2+y^2\equiv p^{2k}(x_0^2+y_0^2)\mod p^{2k+2d-1}.\]
	
	Now assume that $z_0=0$. Then we have that
	\[p^{2k}(x_0^2+y_0^2)\equiv 0\mod p^{2k+2d-1}\Longrightarrow x_0^2+y_0^2\equiv 0\mod p^{2d-1}\]
	which is only possible if $p\equiv 1\mod 4$ and $y_0\equiv\pm\ii x_0$, hence $(x_0,y_0)\in \mathrm{D}_p^+(0)\cup \mathrm{D}_p^-(0)$.
\end{proof}

We also need some results involving a $p$-adic circle action on $(\Qp)^2$. The elementary functions are defined in the usual way on $\Qp$ as a power series:
\[\exp(x)=\sum_{i=0}^{\infty}\frac{x^i}{i!},\]
\[\cos(x)=\sum_{i=0}^{\infty}\frac{(-1)^ix^{2i}}{(2i)!},\]
\[\sin(x)=\sum_{i=0}^{\infty}\frac{(-1)^ix^{2i+1}}{(2i+1)!}.\]

\begin{proposition}[{\cite[Proposition A.11]{CrePel-JC}}]
	\letpprime.
	\begin{enumerate}
		\item The convergence domain of the exponential, cosine and sine series is $p^d\Zp$, where $d=2$ if $p=2$ and otherwise $d=1$.
		\item The image of the exponential series is $1+p^d\Zp$.
	\end{enumerate}
\end{proposition}

See Figure \ref{fig:functions} for a graphical representation of these functions.

\begin{figure}
	\includegraphics[width=.3\linewidth]{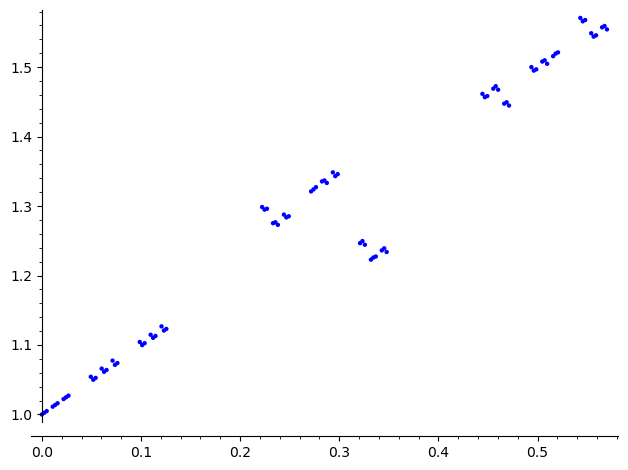}
	\includegraphics[width=.3\linewidth]{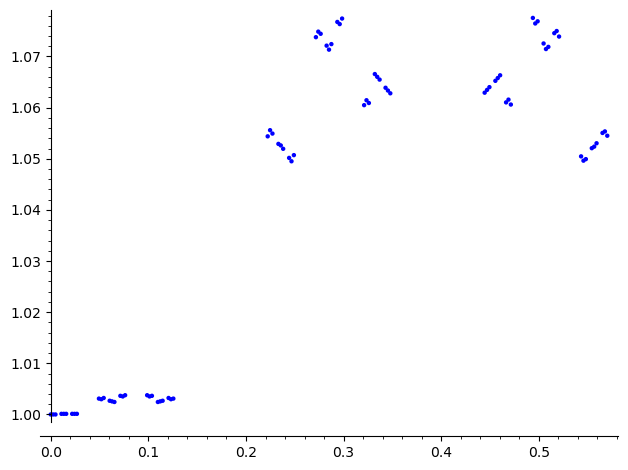}
	\includegraphics[width=.3\linewidth]{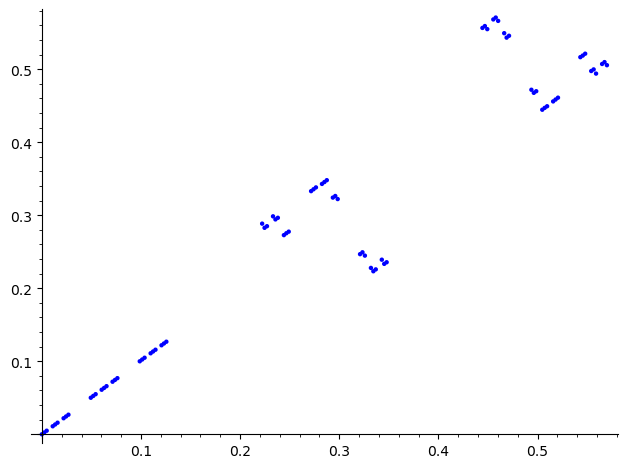}
	\caption{A representation of the $p$-adic exponential, cosine and sine functions for $p=3$.}
	\label{fig:functions}
\end{figure}

\begin{definition}[{\cite[section 2]{CrePel-JC}}]
	\letpprime. The \emph{$p$-adic circle} $\Circle$ is given by
	\[\Circle=\Big\{(x,y)\in(\Qp)^2:x^2+y^2=1\Big\}.\]
\end{definition}

The set $\Circle$ is a group under the binary operation given by
\[(a,b)\cdot(a',b')=(aa'-bb',ab'+a'b).\]
We recall the definition of the $\Circle$-action on $(\Qp)^2$: an element $(a,b)\in\Circle$ acts on $(\Qp)^2$ by multiplication by
\[\begin{pmatrix}
	a & -b \\
	b & a
\end{pmatrix}.\]
The group $\Circle$ contains $p^d\Zp$ as a subgroup, where $d=2$ if $p=2$ and otherwise $d=1$. An element $t\in p^d\Zp$ is identified with $(\cos t,\sin t)$.

\begin{proposition}[{\cite[Proposition 4.2]{CrePel-JC}}]\label{prop:rotation}
	\letpprime. Let \[\mathrm{C}_z=\Big\{(x,y)\in(\Qp)^2:x^2+y^2=z\Big\}\] and $\mathrm{C}_z^*=\mathrm{C}_z\setminus\{0\}$.
	\begin{enumerate}
		\item Any two points in $\mathrm{C}_z^*$ are related by the action of $\Circle$, except if $z=0$, in which case only proportional points are related.
		\item Two points in $\mathrm{C}_z^*$ are related by the action of $p^d\Zp$ if and only if the order $k$ digits of their two coordinates coincide, where $k$ is the lowest order of the coordinates, except if $p=2$, in which case the digits of order $k$ and $k+1$ must coincide.
	\end{enumerate}
\end{proposition}

\begin{proposition}\label{prop:t}
	\letpprime. Let $d=2$ if $p=2$, and otherwise $d=1$. Let $\mathrm{T}_p^2$ be the set in Definition \ref{def:T}. Let $(x,y)\in \mathrm{T}_p^2$ and define $k$, $x_0$ and $y_0$ as in Proposition \ref{prop:D2}. Let $z=x^2+y^2$. Then there exist $(x',y')\in \mathrm{T}_p^2$ and $t\in p^d\Zp$ such that
	\begin{itemize}
		\item $x'^2+y'^2=z$;
		\item $\min\{\ord_p(x'),\ord_p(y')\}=k$;
		\item the digits of order $k$ of $x'$ and $y'$ are $x_0$ and $y_0$;
		\item if $x_0=0$, $x'=0$; if $p=2$ and $x_0=2$, $x'=2^{k+1}$; otherwise $y'=p^ky_0$;
		\item the following holds:
		\begin{equation}\label{eq:t}
		\begin{pmatrix}
		x \\ y
		\end{pmatrix}
		=\begin{pmatrix}
		\cos t & -\sin t \\
		\sin t & \cos t
		\end{pmatrix}
		\begin{pmatrix}
		x' \\ y'
		\end{pmatrix}.
		\end{equation}
	\end{itemize}
	Furthermore, $x'$, $y'$ and $t$ are unique for a given $(x,y)$, and $x'$ and $y'$ only depend in $z$, $x_0$ and $y_0$.
\end{proposition}

\begin{proof}
	First we show that $x'$ and $y'$ exist. Suppose first that $x_0\ne 0$ (if $p\ne 2$) or $x_0$ is odd (if $p=2$). Then $y'=p^ky_0$ and
	\begin{align*}
		z-y'^2 & =x^2+y^2-y'^2 \\
		& \equiv x^2\mod p^{2k+2d-1} \\
		& \equiv p^{2k}x_0^2\mod p^{2k+2d-1}.
	\end{align*}
	Since this is not $0$, $z-y'^2$ is a square in $\Qp$, it has order exactly $2k$, and it has a root whose leading digit (or two leading digits if $p=2$) is $x_0$. This is the $x'$ which we are looking for. If $x_0=0$ (if $p\ne 2$) or $x_0$ is even (if $p=2$), the same happens but with $x'$ and $y'$ swapped. In this process, we have only used $z$, $x_0$ and $y_0$, but not the exact values of $x$ and $y$.
	
	The elements $x'$ and $y'$ are unique because one of them is fixed by the fourth condition (suppose it is $y'$), and the other must be a square root of $z-y'^2$ with a fixed leading digit.
	
	The existence of $t$ is consequence of Proposition \ref{prop:rotation}, and its uniqueness follows from the fact that the sine is injective as a function from $p^d\Zp$ to $p^d\Zp$.
\end{proof}

These definitions and results allow us to define polar coordinates in $(\Qp)^2$:

\begin{definition}\label{def:polar}
	\letpprime. Let $(x,y)\in \mathrm{T}_p^2$. Let $k$, $x_0$, $y_0$ and $z_0$ be defined as in Proposition \ref{prop:D}. We define the \emph{$p$-adic polar coordinates} of $(x,y)$ as $(z,k',k'',a,b,t)$, where
		\begin{itemize}
			\item $z=x^2+y^2$;
			\item $k'=\ord_p(x+\ii y)$ if $p\equiv 1\mod 4$ and $k'=k$ otherwise;
			\item $k''=\ord_p(x-\ii y)$ if $p\equiv 1\mod 4$ and $k''=\ord_p(z)-k$ otherwise;
			\item $(a,b)$ is such that $(a_0,b_0)\times(a,b)=(x_0,y_0)$, where $(a_0,b_0)=(a_p^+(0),b_p^+(0))$ if $(x_0,y_0)\in \mathrm{D}_p^+(0)$, $(a_0,b_0)=(a_p^-(0),b_p^-(0))$ if $(x_0,y_0)\in \mathrm{D}_p^-(0)$, and otherwise $(a_0,b_0)=(a_p(z_0),b_p(z_0))$.
			\item $t$ is as in Proposition \ref{prop:t}.
		\end{itemize}
\end{definition}

See Figure \ref{fig:polar} for an illustration of Definition \ref{def:polar}. This definition is involved and it is not clear that it is well defined a priori. Next we prove that these $p$-adic polar coordinates are indeed well defined.

\begin{figure}
	\includegraphics{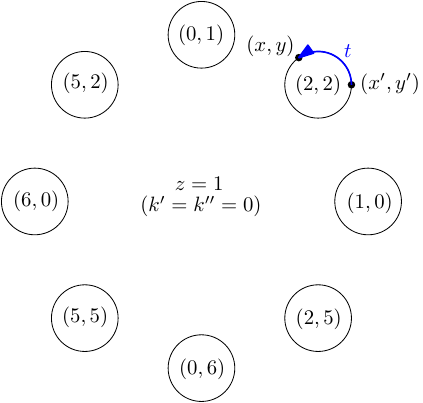}
	\caption{A representation of $p$-adic polar coordinates of Definition \ref{def:polar} for $p=7$. For $z=1$, $k'$ and $k''$ are fixed at $0$ and there are eight possibilities for $(x_0,y_0)$, which in this case coincides with $(a,b)$. These possibilities are represented here as eight circles. For each circle, the position $(x',y')$ is fixed, and $(x,y)$ is the position which results from rotating it an angle $t$.}
	\label{fig:polar}
\end{figure}

\begin{lemma}\label{lemma:orders}
	\letpprime\ such that $p\equiv 1\mod 4$. Let $x,y\in\Qp$. Then
	\[\min\{\ord_p(x+\ii y),\ord_p(x-\ii y)\}=\min\{\ord_p(x),\ord_p(y)\}.\]
\end{lemma}

\begin{proof}
	Each of the orders in the left-hand side is greater or equal than the right-hand side by the triangle inequality, hence the minimum of the two orders is itself greater. Analogously, we have that
	\begin{align*}
		\ord_p(x) & =\ord_p(2x) \\
		& =\ord_p(x+\ii y+x-\ii y) \\
		& \ge\min\{\ord_p(x+\ii y),\ord_p(x-\ii y)\}
	\end{align*}
	and
	\begin{align*}
		\ord_p(y) & =\ord_p(2\ii y) \\
		& =\ord_p(x+\ii y-x+\ii y) \\
		& \ge\min\{\ord_p(x+\ii y),\ord_p(x-\ii y)\},
	\end{align*}
	which proves the other direction.
\end{proof}

\begin{proposition}\label{prop:polar}
	\letpprime. The following statements hold:
	\begin{enumerate}
		\renewcommand{\theenumi}{\roman{enumi}}
		\renewcommand{\theenumii}{\alph{enumii}}
		\item The $p$-adic polar coordinates in Definition \ref{def:polar} are well-defined.
		\item The pair $(a,b)$ and $t$ can take any value in $\mathrm{D}_p(1)$ and $p\Zp$ respectively.
		\item The range of $k'$ and $k''$ is as follows:
		\begin{enumerate}
			\item If $p\equiv 1\mod 4$, $k'$ and $k''$ can take all values in $\Z\cup\{\infty\}$, not both $\infty$ at the same time.
			\item If $p\equiv 3\mod 4$, $k''=k'$.
			\item If $p=2$, either $k''=k'$ or $k''=k'+1$.
		\end{enumerate}
		\item The range of $z$ is all the $p$-adic numbers of order $k'+k''$, except if $p=2$, when it must also end in $01$.
		\item Furthermore, if two points $(x_1,y_1)$ and $(x_2,y_2)$ have the same $p$-adic polar coordinates, then $(x_1,y_1)=(x_2,y_2)$.
	\end{enumerate}
\end{proposition}

\begin{proof}
	That $(a,b)$ is well-defined follows from Propositions \ref{prop:D} and \ref{prop:D2}, because in the three cases $(a_0,b_0)$ is chosen to be in the same orbit modulo multiplication by $\mathrm{D}_p(1)$ as $(x_0,y_0)$. The other coordinates are clearly well-defined.
	
	It is immediate that $(a,b)$ and $t$ take values in $\mathrm{D}_p(1)$ and $p\Zp$ respectively. Respect to $z$, we have that
	\begin{align*}
		k'+k'' & =\ord_p(x+\ii y)+\ord_p(x-\ii y) \\
		& =\ord_p(x^2+y^2) \\
		& =\ord_p(z),
	\end{align*}
	and if $p=2$, it must end in $01$. Respect to $k'$ and $k''$, we have that:
	\begin{itemize}
		\item If $p\equiv 1\mod 4$, there is nothing to prove here.
		\item If $p\equiv 3\mod 4$, we must have $\ord_p(z)=2k$, hence $k'=k''=k$.
		\item If $p=2$, we must have $\ord_p(z)\in\{2k,2k+1\}$, hence $k'=k$ and $k''\in\{k',k'+1\}$.
	\end{itemize}
	
	Now we prove that, given $(z,k',a,b,t)$ in the required range, we can find $(x,y)$ with these values as polar coordinates:
	\begin{enumerate}
		\item First we calculate $k$. If $p\not\equiv 1\mod 4$, this is immediate because $k=k'$. Otherwise, by Lemma \ref{lemma:orders},
		\begin{align*}
			k & =\min\{\ord_p(x),\ord_p(y)\} \\
			& =\min\{\ord_p(x+\ii y),\ord_p(x-\ii y)\} \\
			& =\min\{k',k''\}.
		\end{align*}
		\item Now $z_0$ is the digit of $z$ of order $2k$, or three digits of order $2k$ to $2k+2$.
		\item With this we can deduce $(a_0,b_0)$. If $z_0\ne 0$, this is just applying the definition. If $z_0=0$, which implies $p\equiv 1\mod 4$, we also need to know whether $(x_0,y_0)$ are in $\mathrm{D}_p^+(0)$ or $\mathrm{D}_p^-(0)$. In order to decide this, first note that
		\begin{align*}
			k+\max\{k',k''\} & =\min\{k',k''\}+\max\{k',k''\} \\
			& =\ord_p(z) \\
			& \ge 2k+1
		\end{align*}
		which implies that either $k'\ge k+1$ or $k''\ge k+1$. If it is the former, \[\ord_p(x+\ii y)\ge k+1\] implies that $x_0+\ii y_0\equiv 0\mod p$ and $(x_0,y_0)\in \mathrm{D}_p^+(0)$. If it is the latter, \[\ord_p(x-\ii y)\ge k+1\] and $(x_0,y_0)\in \mathrm{D}_p^-(0)$.
		\item Now we compute $(x_0,y_0)=(a_0,b_0)\cdot(a,b)$.
		\item At this point, $z$ and $(x_0,y_0)$ give us enough information to obtain $(x',y')$ in Proposition \ref{prop:t}.
		\item Finally, we compute $(x,y)$ with equation \eqref{eq:t}.
	\end{enumerate}
	
	We now check that this point $(x,y)$ has $(z,k',k'',a,b,t)$ as polar coordinates. Let $\tilde{k}$, $\tilde{x}_0$, $\tilde{y}_0$, $\tilde{z}_0$, $\tilde{a}_0$ and $\tilde{b}_0$ be the values corresponding to $(x,y)$, $(\tilde{z},\tilde{k}',\tilde{k}'',\tilde{a},\tilde{b},\tilde{t})$ be the actual polar coordinates of this point and $(\tilde{x}',\tilde{y}')$ the result of Proposition \ref{prop:t} applied to $(x,y)$.
	
	The $x'$ and $y'$ obtained in step (5) satisfy the first four conditions of Proposition \ref{prop:t}. Concretely, $x'^2+y'^2=z$, and equation \eqref{eq:t} implies that \[\tilde{z}=x^2+y^2=x'^2+y'^2=z.\] By Proposition \ref{prop:rotation},
	\begin{align*}
		\tilde{k} & =\min\{\ord_p(x),\ord_p(y)\} \\
		& =\min\{\ord_p(x'),\ord_p(y')\} \\
		& =k.
	\end{align*}
	Also by Proposition \ref{prop:rotation}, the digits at order $\tilde{k}=k$ of $x$ and $y$, which are $\tilde{x}_0$ and $\tilde{y}_0$, coincide with those of $x'$ and $y'$, which are $x_0$ and $y_0$. The result of Proposition \ref{prop:t}, $(\tilde{x}',\tilde{y}')$, is calculated with $\tilde{x}_0=x_0$, $\tilde{y}_0=y_0$ and $\tilde{z}=z$, hence \[(\tilde{x}',\tilde{y}')=(x',y'),\] which in turn implies, by step (6), that $\tilde{t}=t$.
	
	Since $\tilde{z}=z$ and $\tilde{k}=k$, by step (2) we also have $\tilde{z}_0=z_0$. In step (3), $(a_0,b_0)$ is calculated applying the definition to $(x_0,y_0)$ and $z_0$, which implies that \[(\tilde{a}_0,\tilde{b}_0)=(a_0,b_0),\] and by step (4) $(\tilde{a},\tilde{b})=(a,b)$.
	
	It is only left to show that $\tilde{k}'=k'$ and $\tilde{k}''=k''$. This happens automatically if $p\not\equiv 1\mod 4$. Otherwise, we need to check that $\ord_p(x+\ii y)=k'$ and $\ord_p(x-\ii y)=k''$. We distinguish three cases.
	\begin{itemize}
		\item $y_0=\ii_0x_0$, that is, $(x_0,y_0)\in \mathrm{D}_p^+(0)$. Then step (3) had $k'\ge k+1$ and $k''=k$. In this case, \[\ord_p(x-\ii y)=k=k''\] and
		\begin{align*}
			\ord_p(x+\ii y) & =\ord_p(z)-\ord_p(x-\ii y) \\
			& =\ord_p(z)-k'' \\
			& =k'.
		\end{align*}
		\item $y_0=-\ii_0x_0$, that is, $(x_0,y_0)\in \mathrm{D}_p^+(0)$. This case is symmetric to the previous one.
		\item The rest of cases. Now $k'=k''=k$ and \[\ord_p(x+\ii y)=\ord_p(x-\ii y)=k=k'=k''.\]
	\end{itemize}
	
	Finally, if we have two points $(x_1,y_1)$ and $(x_2,y_2)$ with the same polar coordinates, we can apply the six steps to these coordinates, obtaining that the points have the same $k$, then the same $z_0$, then the same $(a_0,b_0)$, then the same $(x_0,y_0)$, then the same $(x',y')$, and finally that they are the same point.
\end{proof}

\begin{remark}\label{rem:dimension}
	There are six polar coordinates in Definition \ref{def:polar}, but only $z$ and $t$ are continuous; the other four take discrete values. This corresponds to what we would expect for coordinates on $(\Qp)^2$.
\end{remark}

\begin{proposition}\label{prop:area}
	\letpprime. The $p$-adic area form on the plane $(\Qp)^2$ with standard coordinates $(x,y)$ is expressed in the $p$-adic polar coordinates of Definition \ref{def:polar} as
	\[\dd x\wedge\dd y=\frac{1}{2}\dd z\wedge\dd t.\]
\end{proposition}

\begin{proof}
	Suppose that $x_0\ne 0$. Then $y'=p^ky_0$ is discrete, and
	\begin{align*}
	\dd x\wedge\dd y & =\dd(x'\cos t-y'\sin t)\wedge\dd(y'\cos t+x'\sin t) \\
	& =(\dd x'\cos t-x'\sin t\dd t-y'\cos t\dd t)\wedge(-y'\sin t\dd t+\dd x'\sin t+x'\cos t\dd t) \\
	& =(\cos t(x'\cos t-y'\sin t)+\sin t(x'\sin t+y'\cos t))\dd x'\wedge\dd t \\
	& =x'\dd x'\wedge\dd t \\
	& =\frac{1}{2}\dd x'^2\wedge\dd t \\
	& =\frac{1}{2}\dd (z-y'^2)\wedge\dd t \\
	& =\frac{1}{2}\dd z\wedge\dd t.
	\end{align*}
	The other case is essentially the same, but we include it for completeness: if $x_0=0$, $x'$ is discrete (it is always $0$), and
	\begin{align*}
	\dd x\wedge\dd y & =\dd(-y'\sin t)\wedge\dd(y'\cos t) \\
	& =(-\dd y'\sin t-y'\cos t\dd t)\wedge(\dd y'\cos t-y'\sin t\dd t) \\
	& =(y'\sin^2t+y'\cos^2t)\dd y'\wedge \dd t \\
	& =y'\dd y'\wedge\dd t \\
	& =\frac{1}{2}\dd y'^2\wedge\dd t \\
	& =\frac{1}{2}\dd (z-x'^2)\wedge\dd t \\
	& =\frac{1}{2}\dd z\wedge\dd t.\qedhere
	\end{align*}
\end{proof}

\begin{remark}
	Note that the form given in Proposition \ref{prop:area} is the same form as in the reals, where it is usually expressed in terms of the coordinate \[r=\sqrt{x^2+y^2}:\] then we have $z=r^2$ and $\dd z=2r\dd r$, and the form results in $r\dd r\wedge\dd t$.
\end{remark}

\begin{proposition}\label{prop:polar-actions}
	\letpprime.
	\begin{enumerate}
		\item The rotational action of $p^d\Zp$ on $(\Qp)^2\setminus\{(0,0)\}$ transforms only the polar coordinate $t$ in Definition \ref{def:polar}, and preserves the remaining ones.
		\item The rotational action of $\Circle$ on $(\Qp)^2\setminus\{(0,0)\}$ transforms the polar coordinates $a$, $b$ and $t$, and also $k'$ and $k''$ if $p\equiv 1\mod 4$, and preserves $z$. More precisely, the action of a fixed element of $\Circle$, in polar coordinates, consists of adding a constant to $k'$, $k''$ and $t$, and multiplying $(a,b)$ by a constant in $\mathrm{D}_p(1)$; these constants are independent of $z$.
	\end{enumerate}
\end{proposition}

\begin{proof}
	This is a consequence of Proposition \ref{prop:rotation}. The points which are related by the action of $p^d\Zp$ are those with the same $z$, $k'$, $k''$, $a$ and $b$, and by definition the rotation modifies $t$. The action of $\Circle$ relates all points with the same $z$, except if $z=0$, in which case there are two orbits which differ by \[x_0^{-1}y_0\equiv a^{-1}b\mod p.\] These points have different $a$, $b$ and $t$, and also different $k'$ and $k''$ if $p\equiv 1\mod 4$; otherwise $k$ is fixed by $z$.
	
	Given $(A,B)\in\Circle$, if
	\[\begin{pmatrix}
	x \\ y
	\end{pmatrix}
	=\begin{pmatrix}
	A & -B \\
	B & A
	\end{pmatrix}
	\begin{pmatrix}
	x' \\ y'
	\end{pmatrix},\]
	then
	\[x+\ii y=(A+\ii B)(x_0+\ii y_0),\]
	so the change in $k'$ induced by $(A,B)$ is given by the order of $A+\ii B$. Analogously, the change in $k''$ is given by the order of $A-\ii B$, and $(a,b)$ is being multiplied by the class of $(A,B)$ in $\mathrm{D}_p(1)$. By definition of $t$, a variation from $t_0$ to $t$ corresponds to a rotation by the angle $t-t_0$ in $p\Zp$, which is a multiplication by $(\cos(t-t_0),\sin(t-t_0))\in\Circle$.
\end{proof}

\section{Nonlinear $p$-adic equivariant squeezing and nonlinear $p$-adic equivariant non-squeezing}\label{sec:equivariant}

In this section we prove a $p$-adic version of the \textit{equivariant} Gromov's non-squeezing theorem proved, in the real case, by Figalli-Palmer-Pelayo \cite[Proposition 1.1]{FPP} (generalizing related ideas from Figalli-Pelayo \cite{FigPel} and Pelayo \cite{Pelayo-packing,Pelayo-embeddings}). In the equivariant case the embeddings and symplectomorphisms are required to preserve the standard action of a torus on the ball.

A fundamental ingredient needed to prove the results in this section is the definition of the $p$-adic polar coordinates in Section \ref{sec:polar}.

\letnpos. \letpprime. The action of the group $\Circle$ on $(\Qp)^2$ induces an action of $(\Circle)^n$ on $(\Qp)^{2n}$ coordinatewise:
\[(g_1,\ldots,g_n)\cdot(x_1,y_1,\ldots,x_n,y_n)=(g_1\cdot(x_1,y_1),\ldots,g_n\cdot(x_n,y_n)),\]
or of $(\Circle)^s$ on $(\Qp)^{2n}$ for $1\le s\le n-1$:
\[(g_1,\ldots,g_s)\cdot(x_1,y_1,\ldots,x_n,y_n)=(g_1\cdot(x_1,y_1),\ldots,g_s\cdot(x_s,y_s),x_{s+1},y_{s+1},\ldots,x_n,y_n).\]
See Figure \ref{fig:action} for an illustration of this action in the real case.

\begin{figure}
	\includegraphics[width=0.95\linewidth]{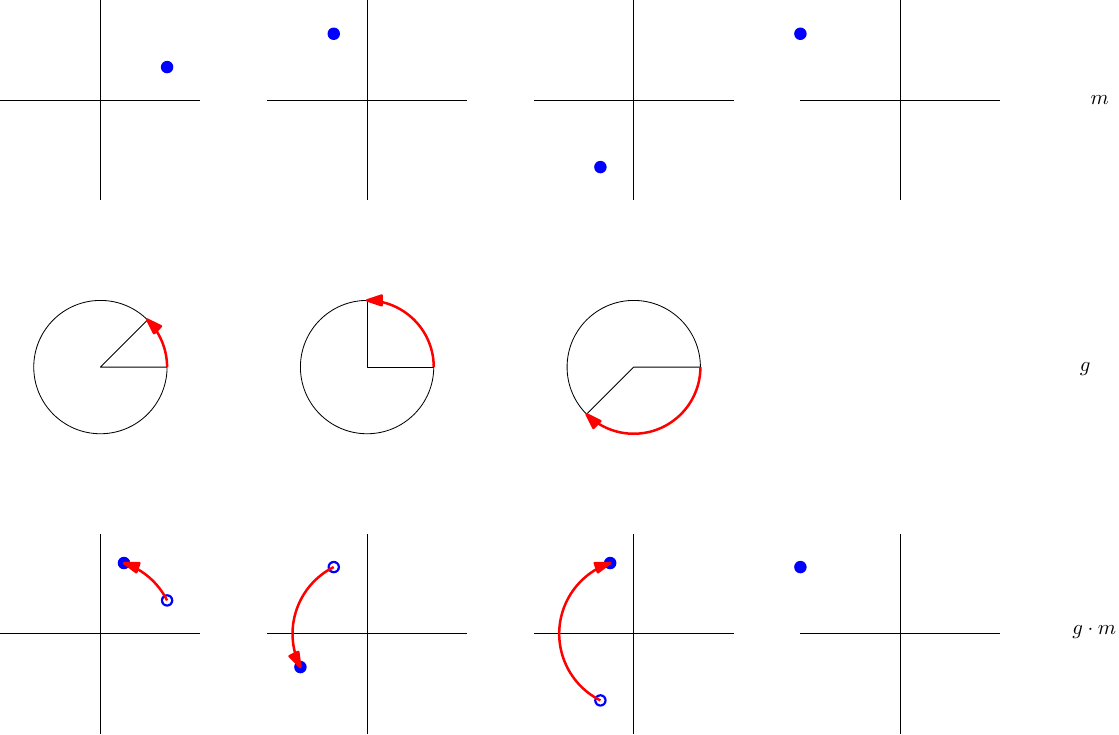}
	\caption{The action of $(\mathrm{S}^1)^3$ on $\R^8$. The four blue points represent a point $m\in\R^8$ and the circles below them represent a point $g\in(\mathrm{S}^1)^3$. Below them there is the result $g\cdot m$ of the rotational action.}
	\label{fig:action}
\end{figure}

The $p$-adic analytic symplectomorphism in the proof of Theorem \ref{thm:total-embedding} is not invariant with respect to the rotational action of $(p^d\Zp)^n$ on $(\Qp)^{2n}$, not even the action of $p^d\Zp$ on the first two coordinates (much less the action of the entire $(\Circle)^n$ or $\Circle$). If we require that the action of the torus must be preserved, then the non-squeezing theorem holds.

For $(x,y)\in(\Qp)^2,$ with polar coordinates $(z,k',k'',a,b,t),$
\begin{align*}
	\|(x,y)\|_p & =p^{-\min\{\ord_p(x),\ord_p(y)\}} \\
	& =p^{-\min\{k',k''\}}.
\end{align*}
The points in $\ball_p^{2n}(r)$ are exactly those with $\min\{k',k''\}\ge\ord(r)$. If $p\equiv 1\mod 4$, then $\ball_p^{2n}(r)$ is not invariant by the action of $\Circle$, because by Proposition \ref{prop:polar-actions} this action changes the polar coordinates $k'$ and $k''$. Hence, we must restrict the action to the subgroup of $\Circle$ which does not vary them.

\begin{definition} \label{def:G}
	\letnpos. \letpprime. Let $\mathrm{G}_p$ be the largest subgroup of $\Circle$ that leaves invariant $\ball_p^2(r)$ as a subset of $(\Qp)^2$ for any radius $r$ by the action in \eqref{eq:action}, or equivalently $\ball_p^{2n}(r)$ as a subset of $(\Qp)^{2n}$ for any radius $r$ and any positive integer $n$.
\end{definition}

\begin{remark}
	The group $\mathrm{G}_p$ in Definition \ref{def:G} can be explicitly described by: $\mathrm{G}_p=\Circle$, if $p\not\equiv 1\mod 4$, and
	\[\mathrm{G}_p=\Big\{(a,b)\in\Circle:\ord_p(a+\ii b)=0\Big\},\]
	where $\ii$ is any number in $\Qp$ such that $\ii^2=-1$, if $p\equiv 1\mod 4$.
	
	By definition, the largest subgroup of $(\Circle)^n$ that acts on the cylinder $\cyl_p(R)$ is \[\mathrm{G}_p\times(\Circle)^{n-1}.\]
\end{remark}

In the coming results we use the following notion of equivariance.

\begin{definition}\label{def:equivariant}
	Let $G$ be a group. Let $X,Y$ be sets, and let $f:X\to Y$ be a map. Let $G$ act on $X$ and on $Y$. We say that $f$ is \emph{generalized $G$-equivariant} if there exists a group isomorphism $h:G\to G$ such that $f(g\cdot x)=h(g)\cdot f(x)$ for every $g\in G$ and every $x\in X$. If $h:G\to G$ is the identity map we say that $f$ is \emph{$G$-equivariant}.
\end{definition}

Sometimes a generalized $G$-equivariant map as in Definition \ref{def:equivariant} is simply called ``$G$-equivariant''. But for the purposes of this paper we only use ``$G$-equivariant'' when $h:G\to G$ is the identity map.

\begin{theorem}[$p$-adic equivariant analog of Gromov's non-squeezing theorem]\label{thm:analytic-actions-complete}
	\letnpos. \letpprime. Let $r,R$ be $p$-adic absolute values. Endow both the $2n$-dimensional $p$-adic ball $\ball_p^{2n}(r)$ of radius $r$ and the $2n$-dimensional $p$-adic cylinder $\cyl_p^{2n}(R)$ of radius $R$ with the standard symplectic form $\sum_{i=1}^n\dd x_i\wedge\dd y_i$, where $(x_1,y_1,\ldots,x_n,y_n)$ are the standard symplectic coordinates on $(\Qp)^{2n}$. Let $\mathrm{G}_p$ be the group given in Definition \ref{def:G}. Then the following statements are equivalent:
	\begin{enumerate}
		\item There exists a $(\mathrm{G}_p)^n$-equivariant $p$-adic analytic symplectic embedding \[f:\ball_p^{2n}(r)\hookrightarrow \cyl_p^{2n}(R).\]
		\item There exists a generalized $(\mathrm{G}_p)^n$-equivariant $p$-adic analytic symplectic embedding \[f:\ball_p^{2n}(r)\hookrightarrow \cyl_p^{2n}(R).\]
		\item $r\le R$.
	\end{enumerate}
\end{theorem}

\begin{proof}
	$(1)\Rightarrow(2)$ is true by definition. $(3)\Rightarrow(1)$ is immediate by taking the inclusion map $i:\ball_p^{2n}(r)\hookrightarrow\cyl_p^{2n}(R)$. Next we prove the implication $(2)\Rightarrow(3)$.
	
	Given a point $g=(g_1,\ldots,g_n)\in(\mathrm{G}_p)^n,$ the fixed points of its action are those with $x_j=y_j=0$, where $j$ runs over all coordinates for which $g_j\ne 1$ (this $1$ stands for the identity element in $\Circle$). The map $f$ must send fixed points of $g$ to fixed points of $h(g)$, which means that the set of fixed points of $h(g)$ has at least the dimension of that of $g$, for any $g$. This in turn means that, for any $g\in(\mathrm{G}_p)^n$, $h(g)$ has at least the same number of coordinates equal to $1$ as $g$. Since $h$ is a group isomorphism, the two numbers are the same for any $g$.
	
	We choose a $g_0\in \mathrm{G}_p$ with $g_0\ne 1$. Let \[g_i=h(1,\ldots,g_0,\ldots,1),\] with the $g_0$ in position $i$. This $g_i$ must have exactly one coordinate different from $1$. Let $k_i$ be its index. Since $h$ is a group isomorphism, $(k_1,\ldots,k_n)$ is a permutation of $(1,\ldots,n)$. Let $j$ be such that $k_j=1$.
	
	We consider the $2$-dimensional disk of radius $r$ given by
	\[D=\Big\{(x_1,y_1,\ldots,x_n,y_n)\in \ball_p^{2n}(r): x_i=y_i=0 \text{ if } i\ne j\Big\}.\]
	All points of $D$ are fixed points of $(g_0,\ldots,1,\ldots,g_0)$, with $1$ in the position $j$, hence their images by $f$ are fixed points of
	\[h(g_0,\ldots,1,\ldots,g_0)=g_1g_2\ldots g_{j-1}g_{j+1}\ldots g_n.\]
	Each factor in the right has one and only one coordinate different from $1$, and their product has all coordinates different from $1$ except the first, which implies that
	\[f(D)\subset\Big\{(x_1,y_1,\ldots,x_n,y_n)\in \cyl_p^{2n}(R): x_i=y_i=0 \text{ if } i\ne 1\Big\}.\]
	This means that $f(D)$ is contained in a disk $D'$ of radius $R$, so \[\mu(f(D))\le R^2,\] where $\mu$ is the $p$-adic area measure. But $f$ must preserve the $p$-adic symplectic form $\omega_0$, which on the disk $D$ reduces to $\dd x_j\wedge\dd y_j$, that is, the area measure on $D$, and on the disk $D'$ it reduces to $\dd x_1\wedge\dd y_1$, the area measure on $D'$. Hence $f$ preserves the area of $D$, that is, \[\mu(f(D))=\mu(D)=r^2,\] and $r\le R$.
\end{proof}

\begin{remark}
	The condition of being symplectic in parts (1) and (2) of Theorem \ref{thm:analytic-actions-complete} is needed to have non-squeezing, because multiplying the coordinates by arbitrary constants is an equivariant transformation if $x_i$ and $y_i$ are multiplied by the same constant, and it may squeeze the ball into a cylinder, but in general it does not preserve the symplectic form.
\end{remark}

However, there is another, more complicated, $p$-adic analytic symplectic embedding that preserves this action if we remove the fixed points of the action of all elements of the torus, that is, if we restrict to $\mathrm{T}_p^{2n}$ instead of $(\Qp)^{2n}$.
It is also possible to achieve squeezing if the action is not toric, but semitoric, that is, if we consider the action of $\mathrm{G}_p\times(\Circle)^{s-1}$, or $(\mathrm{G}_p)^s$, on $(\Qp)^{2n}$, on the first $2s$ coordinates.

\begin{theorem}[$p$-adic equivariant symplectic squeezing theorem, Version I]\label{thm:total-actions-nofixed}
	Let $n$ be an integer with $n\ge 2$. \letpprime. Let $\mathrm{T}_p^{2n}$ be the $p$-adic analytic manifold given in Definition \ref{def:T}. Endow $\mathrm{T}_p^{2n}$ and the $2n$-dimensional $p$-adic cylinder $\cyl_p^{2n}(1)$ of radius $1$ with the standard $p$-adic symplectic form $\sum_{i=1}^n\dd x_i\wedge\dd y_i$, where $(x_1,y_1,\ldots,x_n,y_n)$ are the standard symplectic coordinates on $(\Qp)^{2n}$. Let $\mathrm{G}_p$ be the group given in Definition \ref{def:G}. Then there exists a $(\mathrm{G}_p\times(\Circle)^{n-1})$-equivariant $p$-adic analytic symplectomorphism \[\phi:\mathrm{T}_p^{2n}\overset{\cong}{\longrightarrow} \mathrm{T}_p^{2n}\cap\cyl_p^{2n}(1).\]
\end{theorem}

\begin{proof}
	We use the same idea as in the proof of Theorem \ref{thm:total-embedding}, but applying it to the $p$-adic equivalent of polar coordinates given in Definition \ref{def:polar}. Without loss of generality we assume that $n=2$, because the group acts independently on each pair of coordinates, and the conclusion for every $n$ follows by multiplying by the identity.
	
	We will construct the $p$-adic analytic symplectomorphism in several steps: first we construct a symplectic embedding from the total space to the cylinder, and then we show how to make that into a symplectomorphism.
	
	\medskip
	
	\textit{First step: construction of the embedding.} Let $m=(x_1,y_1,x_2,y_2)\in \mathrm{T}_p^4$ with polar coordinates \[(z_1,k'_1,k''_1,a_1,b_1,t_1,z_2,k'_2,k''_2,a_2,b_2,t_2).\] This point is in the cylinder if and only if $k'_1\ge 0$ and $k''_1\ge 0$ (recall that these can be $\infty$ if $p\equiv 1\mod 4$). By Proposition \ref{prop:polar-actions}, its orbit by $\mathrm{G}_p\times\Circle$ is determined by the coordinates $(z_1,k_1',k_1'',z_2,k_2',k_2'')$.
	
	We separate the parts of $z_1$ and $z_2$ before and after two places left of the decimal point:
	\[\left\{\begin{aligned}
		z_1 & =C+\sum_{i=-1}^\infty c_ip^{-i}; \\
		z_2 & =D+\sum_{i=-1}^\infty d_ip^{-i},
	\end{aligned}\right.\]
	where $C,D\in p^2\Zp$ and $c_i,d_i\in\{0,\ldots,p-1\}$ for $i\ge -1$.
	
	If the point is already in the cylinder, we define
	\[\left\{\begin{aligned}
		z_1' & =z_1; \\
		z_2' & =D+\sum_{i=-1}^\infty d_ip^{-i-2}.
	\end{aligned}\right.\]
	Otherwise, we define
	\[\left\{\begin{aligned}
		z_1' & =C+1; \\
		u & =D+p+\sum_{i=-1}^\infty (c_ip^{-2i-2}+d_ip^{-2i-3}),
	\end{aligned}\right.\]
	and construct $z_2'$ from $u$ by changing the digit at the first even place at the right of $\ord_p(u)-1$ from $0$ to $1$.
	
	If $p\not\equiv 1\mod 4$, after replacing $z_1$ and $z_2$ by $z_1'$ and $z_2'$ and adjusting $k_1'$, $k_1''$, $k_2'$ and $k_2''$ (for these primes, the $k'$ and $k''$ coordinates are determined by $z$), we get the valid polar coordinates of a point by Proposition \ref{prop:polar}. Indeed, both $z_1'$ and $z_2'$ have even order if $z_1$ and $z_2$ have even order, they end in $01$ if $z_1$ and $z_2$ end in $01$, the coordinates $a$, $b$ and $t$ are unchanged, and their range does not depend in the other coordinates. The result is a map $f$ sending $\mathrm{T}_p^2$ to the cylinder. We can see that it is injective, because from the digit of order $1$ of $z_2'$ we can deduce which of the two cases we are in; after that, reconstructing the original $z_1$ and $z_2$ (and with them the original $k'_i$ and $k''_i$) is just a matter of rearranging the digits to their original place. See Figures \ref{fig:embedding2} and \ref{fig:embedding3} for symbolic representations of the embedding for $p=2$ and $p=3$.
	
	\begin{figure}
		\includegraphics[height=0.9\textheight]{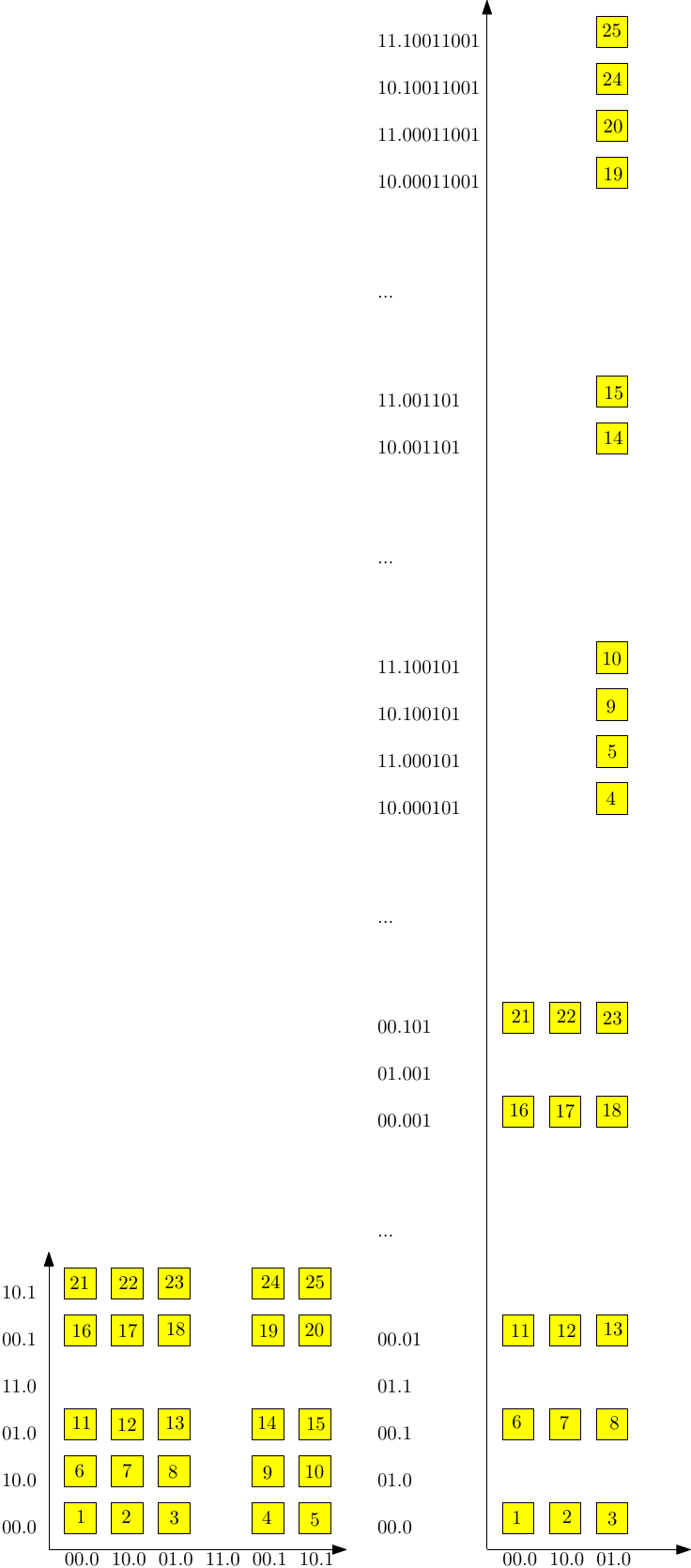}
		\caption{A symbolic representation of the embedding $f$ in the proof of Theorem \ref{thm:total-actions-nofixed} for $p=2$. The squares are balls of radius $1/4$ which correspond to the values that $z_1$ and $z_2$ can take in a ball of radius $2$. We can see that $z_1$ always ends up being integer, while $z_2$ gets dispersed.}
		\label{fig:embedding2}
	\end{figure}

	\begin{figure}
		\begin{tabular}{cc}
			& \includegraphics[width=0.4\linewidth]{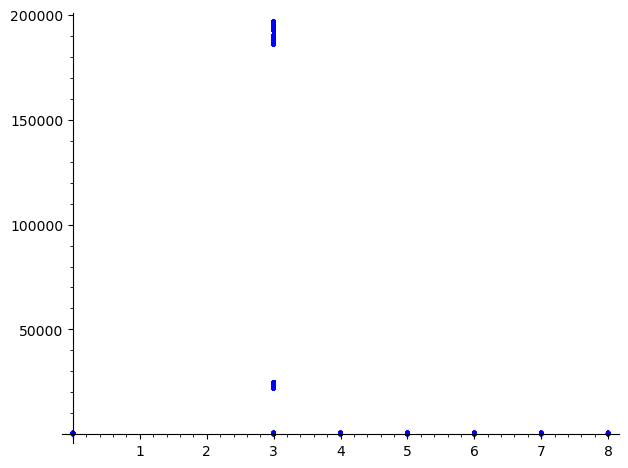} \\
			\includegraphics[width=0.4\linewidth]{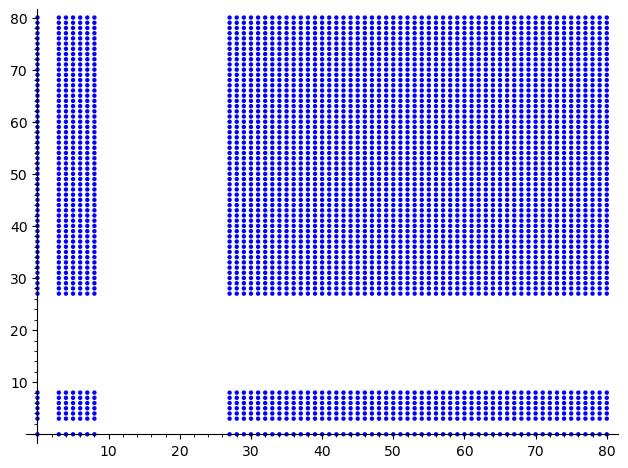} & \includegraphics[width=0.4\linewidth]{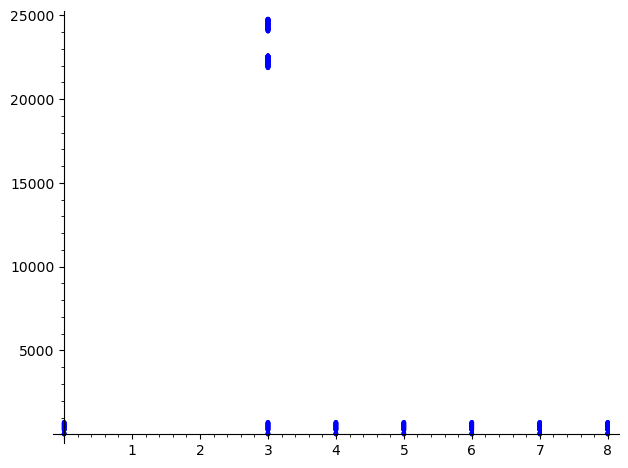} \\
			& \includegraphics[width=0.4\linewidth]{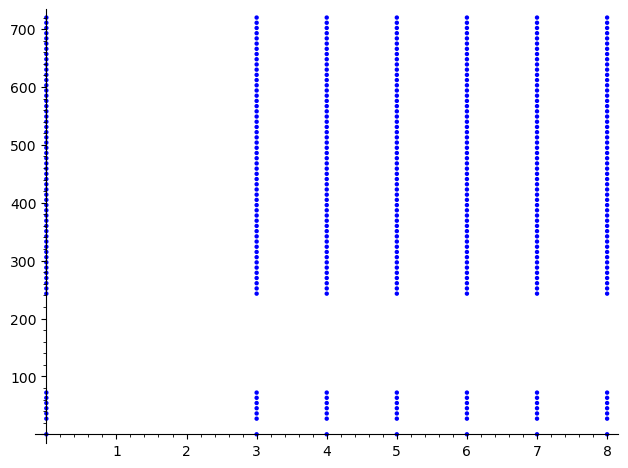}
		\end{tabular}
	\caption{A symbolic representation of the embedding $f$ in the proof of Theorem \ref{thm:total-actions-nofixed} for $p=3$. The dots are balls of radius $1/9$ which correspond to the values that $z_1$ and $z_2$ can take in a ball of radius $9$ (all the possible values in a ball of radius $3$ are inside the cylinder). Left: the original $z_1$ and $z_2$ for $3721$ possible balls. Right: three views of the balls embedded in the cylinder. We can see that $z_1$ is integer (the three horizontal axes are equal) while $z_2$ gets dispersed, and the balls which are seen in the lowest figure are precisely those which were originally inside the cylinder, though in a different position.}
	\label{fig:embedding3}
	\end{figure}
	
	If $p\equiv 1\mod 4$ (note that this implies $p\ge 5$) and the original point was not in the cylinder, we further modify $z_2'$ in the following way:
	\begin{enumerate}
		\item If $k'_2=\infty$, we change the leading digit from $1$ to $2$.
		\item If $k'_1\ge 0$ is finite, we change the digit $k_1'+1$ places at the right of the leading digit to a $1$.
		\item If $k_1'<0$, we change the digit $-k_1'$ places at the right of the leading digit to a $2$.
		\item If $k_1'=\infty$, we change the digit at the right of the leading digit to a $3$.
		\item We repeat the three previous steps with $k_1''$ instead of $k_1'$.
	\end{enumerate}
	After that, we set $k_1'$ and $k_1''$ to $0$. If $k_2'=\infty$, we correct its value to \[\ord_p(z_2')-k_2'';\] otherwise we correct the value of $k_2''$ to \[\ord_p(z_2')-k_2'.\] By Proposition \ref{prop:polar}, the coordinates are valid and are inside the cylinder; moreover, it is straightforward to recover the original $k_1'$ and $k_1''$ from $z_2'$. This means that, defining $f$ in this way, it is also injective.
	
	\medskip
	
	\textit{Second step: passing to the quotient.} At this point we have constructed an injective map $f$ from the total space $\mathrm{T}_p^{2n}$ to the $p$-adic symplectic cylinder $\cyl_p^{2n}(1)$. We now show that it preserves some relations between the points.
	
	We consider on $\mathrm{T}_p^4$ the relation given by
	\[m_1\sim m_2\Longleftrightarrow |z_i(m_1)-z_i(m_2)|_p\le p^{-2}\text{ for }i=1,2,\]
	where $z_i(m_j)$ denotes the polar coordinate $z_i$ of $m_j$. This relation is clearly reflexive and symmetric. Furthermore it is transitive because of the strong triangle inequality for the $p$-adic absolute value. By the construction of $f$,
	\[m_1\sim m_2\Longleftrightarrow f(m_1)\sim f(m_2).\]
	Hence we can quotient $\mathrm{T}_p^4$ by this equivalence relation and consider $f$ as a map on the quotient \[\mathrm{T}_p^4/\sim.\] Moreover, $f$ preserves the values of the coordinates $a_1$, $a_2$, $b_1$, $b_2$, $t_1$ and $t_2$, and shifts $k'_2$ and $k''_2$ by a quantity depending only on $z_1$ and $z_2$. It changes $z_1$ and $z_2$ in a way that depends only on $k_1'$ and $k_1''$, and also on whether $k_2'$ was $\infty$ or not, all of which the action of $\mathrm{G}_p\times\Circle$ cannot change.
	
	The above implies that $f$ commutes with the action of $\mathrm{G}_p\times\Circle$, and we can consider $f$ as a map on
	\[(\mathrm{T}_p^4/\sim)/(\mathrm{G}_p\times\Circle).\]
	Let $Q$ be this quotient and let $Q'$ be the subset of $Q$ contained in $\cyl_p^4(1)$. We know that $f$ is an injective map from $Q$ to $Q'$. By the Schl\"oder-Bernstein Theorem, there exists a bijection \[F:Q\to Q'\] such that
	\begin{equation}\label{eq:bijection}
		F(X)\in\{X,f(X)\}
	\end{equation}
	for any $X\in Q$.
	
	\medskip
	
	\textit{Third step: $f$ is symplectic.} Now we prove that $f$ is a $p$-adic analytic symplectic embedding when restricted to an element $X\in Q$.
	
	Consider $f$ restricted to an element $X\in Q$. Again by the construction of $f$, it is a bijection between $X$ and $f(X)$ which preserves the digits in the $C$ and $D$ parts of $z_1$ and $z_2$. This implies that $f$ is a $p$-adic analytic embedding and preserves the differential forms $\dd z_1$ and $\dd z_2$, and it trivially preserves $\dd t_1$ and $\dd t_2$, hence it also preserves
	\[\frac{1}{2}(\dd z_1\wedge\dd t_1+\dd z_2\wedge \dd t_2),\]
	which is equal to \[\dd x_1\wedge\dd y_1+\dd x_2\wedge \dd y_2\] by Proposition \ref{prop:area}. Hence, $f$ is a $p$-adic analytic symplectomorphism between $X$ and $f(X)$.
	
	\medskip
	
	\textit{Fourth step: construction of the symplectomorphism.} Now we turn $f$ into the desired symplectomorphism.
	
	Let $m\in \mathrm{T}_p^4$ and let $[m]$ be the class of $m$ in $Q$. By \eqref{eq:bijection}, $F([m])$ is either $[m]$ or $f([m])$. If it is the former, we define $\phi(m)=m$; if it is the latter, we define $\phi(m)=f(m)$. In either case, we have that \[[\phi(m)]=F([m]).\]
	
	Since $F$ is a bijection between $Q$ and $Q'$, and $f$ is bijective between $[m]$ and $f([m])$, $\phi$ is also bijective, and $\phi$ either coincides with the identity or with $f$ at all the points of $[m]$. Since $[m]$ is an open set which contains the orbit of $m$ by the group $\mathrm{G}_p\times\Circle$, and both the identity and $f$, when restricted to $[m]$, are $p$-adic analytic symplectic embeddings which commute with the action of $\mathrm{G}_p\times\Circle$, $\phi$ is a $p$-adic analytic symplectomorphism which also commutes with this action, and we are done.
\end{proof}

Now we prove another version of $p$-adic squeezing.

\begin{theorem}[$p$-adic equivariant symplectic squeezing theorem, Version II]\label{thm:total-actions-semi}
	Let $n$ and $s$ be integers with $n\ge 2$ and $1\le s\le n-1$. \letpprime. Endow the $2n$-dimensional $p$-adic cylinder $\cyl_p^{2n}(1)$ of radius $1$ with the standard symplectic form $\sum_{i=1}^n\dd x_i\wedge\dd y_i$, where $(x_1,y_1,\ldots,x_n,y_n)$ are the standard symplectic coordinates on $(\Qp)^{2n}$. Let $\mathrm{G}_p$ be the group given in Definition \ref{def:G}. Then there exists a $(\mathrm{G}_p\times(\Circle)^{s-1})$-equivariant $p$-adic analytic symplectomorphism \[\varphi:(\Qp)^{2n}\overset{\cong}{\longrightarrow}\cyl_p^{2n}(1).\]
\end{theorem}

\begin{proof}
	This has essentially the same proof than Theorem \ref{thm:total-actions-nofixed}. Now we use the polar coordinates of Section \ref{sec:polar} only for the first pair if it is possible:
	\[(z_1,k_1',k_1'',a_1,b_1,t_1,x_2,y_2),\]
	and use $x_2$ instead of $z_2$ in the proof. In these coordinates, again by Proposition \ref{prop:area}, the $p$-adic symplectic form is
	\[\omega_0=\frac{1}{2}\dd z_1\wedge\dd t_1+\dd x_2\wedge \dd y_2.\]
	The space from which we start is now $(\Qp)^4$ instead of $\mathrm{T}_p^4$, and the first pair of coordinates can only be changed to polar coordinates if $(x_1,y_1)\ne(0,0)$. Otherwise, if $x_1=y_1=0$, we leave these two coordinates unchanged. This does not present a problem, because $f$ changes this first pair only if the point is not in the $p$-adic symplectic cylinder $\cyl_p^{2n}(1)$, which is not the case around a point with $(x_1,y_1)=(0,0)$.
\end{proof}

\begin{corollary}[$p$-adic non-linear equivariant squeezing theorem with fixed points removed]\label{cor:analytic-actions}
	Let $n$ and $s$ be integers with $n\ge 2$ and $1\le s\le n-1$. \letpprime. Let $r,R$ be $p$-adic absolute values. let $\mathrm{T}_p^{2n}$ be the $p$-adic analytic manifold in Definition \ref{def:T}. Endow $\mathrm{T}_p^{2n}$, the $2n$-dimensional $p$-adic ball $\ball_p^{2n}(r)$ of radius $r$ and the $2n$-dimensional $p$-adic cylinder $\cyl_p^{2n}(R)$ of radius $R$ with the standard $p$-adic symplectic form $\sum_{i=1}^n\dd x_i\wedge\dd y_i$, where $(x_1,y_1,\ldots,x_n,y_n)$ are the standard coordinates on $(\Qp)^{2n}$. Let $\mathrm{G}_p$ be the group given in Definition \ref{def:G}. Then the following statements hold.
	\begin{enumerate}
		\item There exists a $(\mathrm{G}_p)^n$-equivariant $p$-adic analytic symplectic embedding \[\phi:\ball_p^{2n}(r)\cap \mathrm{T}_p^{2n}\hookrightarrow \cyl_p^{2n}(R).\]
		\item There exists a $(\mathrm{G}_p)^s$-equivariant $p$-adic analytic symplectic embedding \[\varphi:\ball_p^{2n}(r)\hookrightarrow \cyl_p^{2n}(R).\]
	\end{enumerate}
\end{corollary}

\begin{proof}
	The two parts follow directly from Theorems \ref{thm:total-actions-nofixed} and \ref{thm:total-actions-semi}, by restricting the $p$-adic analytic symplectomorphisms to the ball and combining them with a scaling.
\end{proof}

\section{$p$-adic analytic symplectic capacities and $p$-adic analytic symplectic $G$-capacities}\label{sec:capacities}

\letnpos. In the real case, a symplectic capacity $c$ on $\R^{2n}$ (see \cite[page 460]{McDSal} for example) assigns to each open set $A\subset \R^{2n}$ a number $c(A)\in[0,\infty]$ such that the following properties hold (these properties generalize the properties (i), (ii), (iii) of $p$-adic linear symplectic width in Proposition \ref{prop:width}):

\begin{enumerate}
	\item \emph{Monotonicity}: if $X, Y$ are open subsets of $\R^{2n}$ and there exists a symplectic embedding $f:X\hookrightarrow Y$, then $c(X)\le c(Y)$.
	\item \emph{Conformality}: if $X$ is an open subset of $\R^{2n}$ and $\lambda\in\R$, then $c(\lambda X)=\lambda^2c(X)$.
	\item \emph{Non-triviality}: $c(\ball^{2n}(1))>0$ and $c(\cyl^{2n}(1))<\infty$.
\end{enumerate}

As it is explained for example in \cite[page 459]{McDSal}, Gromov's nonsqueezing theorem is equivalent to the existence of a symplectic capacity $c$ on $\R^{2n}$ such that
\[c(\ball^{2n}(1))=c(\cyl^{2n}(1))=\pi.\]

If there exists such a capacity, then by the conformality condition we have \[c(\ball^{2n}(r))=c(\cyl^{2n}(r))\] for any radius $r$. If there is an embedding from $\ball^{2n}(r)$ to $\cyl^{2n}(R)$, then by the monotonicity condition $r\le R$, and Gromov's nonsqueezing holds. Conversely, if Gromov's nonsqueezing holds, the capacity defined as
\[\sup\Big\{\pi r^2:\exists\text{ a symplectic embedding }f:\ball^{2n}(r)\hookrightarrow X\Big\}\]
satisfies by definition the first two properties, and Gromov's nonsqueezing implies the third.

If we make the analogous definition to this one in the $p$-adic case, by replacing everywhere $\R$ by $\Qp$ and
\[\lambda^2c(X)\text{ by }|\lambda|_p^2c(X),\]
because in the $p$-adic case it does not hold that $|\lambda|_p^2=\lambda^2$ for any $\lambda$, we would have the concept of \emph{$p$-adic analytic symplectic capacity}. The definition makes sense formally, but Theorem \ref{thm:total-embedding} implies that there does not exist any such capacity in higher dimensions because no such $c$ satisfies the non-triviality assumption.

\begin{theorem}[Non-existence of $p$-adic analytic symplectic capacity]\label{thm:no-capacity}
	Let $n$ be a positive integer. Let $p$ be a prime number. There exists a $p$-adic analytic symplectic capacity on $(\Qp)^{2n}$ if and only if $n=1$.
\end{theorem}

\begin{proof}
	For $n=1$ the $p$-adic area is a $p$-adic analytic symplectic capacity.
	
	Suppose it exists for some $n>1$. Then, by Theorem \ref{thm:total-embedding}, the capacity of $(\Qp)^{2n}$ must equal that of $\cyl_p^{2n}(1)$. But $c(\ball_p^{2n}(r))$ is different from $0$ by non-triviality, and by conformality it tends to infinity when $r$ tends to infinity, hence \[c((\Qp)^{2n})=\infty\] while \[c(\cyl_p^{2n}(1))<\infty\] again by non-triviality. This is a contradiction.
\end{proof}

On the other hand, it is possible to define a $p$-adic analytic symplectic $(\mathrm{G}_p)^n$-capacity, that is, a similar kind of measure but with the monotonicity condition restricted to those embeddings which preserve the action of $(\mathrm{G}_p)^n$:

\begin{definition}\label{def:capacity}
	Let $n$ be a positive integer. \letpprime. Let $\mathrm{G}_p$ be the group in Definition \ref{def:G}. A \emph{$p$-adic analytic symplectic $(\mathrm{G}_p)^n$-capacity $c$ on $(\Qp)^{2n}$} is an assignment of a $p$-adic absolute value $c(X)$ to each $(\mathrm{G}_p)^n$-invariant open subset $X$ of $(\Qp)^{2n}$ such that the following three conditions hold:
	\begin{enumerate}
		\item \emph{Monotonicity}: if $X, Y$ are open subsets of $(\Qp)^{2n}$ and there exists a generalized $(\mathrm{G}_p)^n$-equivariant $p$-adic analytic symplectic embedding $f:X\hookrightarrow Y$, then $c(X)\le c(Y)$.
		\item \emph{Conformality}: if $X$ is an open subset of $\Qp^{2n}$ and $\lambda\in\Qp$, then $c(\lambda X)=|\lambda|_p^2c(X)$.
		\item \emph{Non-triviality}: $c(\ball^{2n}(1))>0$ and $c(\cyl^{2n}(1))<\infty$.
	\end{enumerate}
\end{definition}

The equivariant capacities in Definition \ref{def:capacity} do exist, in contrast with Theorem \ref{thm:no-capacity}.

\begin{theorem}[Existence of $p$-adic analytic symplectic $(\mathrm{G}_p)^n$-capacities]\label{thm:capacity}
	Let $n$ be a positive integer. Let $p$ be a prime number. Let $\mathrm{G}_p$ be the group in Definition \ref{def:G}. Let $X$ be an open subset of $(\Qp)^{2n}$. Then there exist $p$-adic analytic symplectic $(\mathrm{G}_p)^n$-capacities on $(\Qp)^{2n}$. More explicitly, the $(\mathrm{G}_p)^n$-Gromov width $\mathrm{w}^{(\mathrm{G}_p)^n}(X)$ of $X$ is defined by
	\begin{equation}\label{eq:capacity}
		\sup\Big\{r^2:\exists\text{ a $(\mathrm{G}_p)^n$-equivariant $p$-adic analytic symplectic embedding } f:\ball_p^{2n}(r)\hookrightarrow X\Big\}
	\end{equation}
	is a $p$-adic analytic symplectic $(\mathrm{G}_p)^n$-capacity on $\Qp^{2n}$.
\end{theorem}

\begin{proof}
	The first two conditions are consequences of the definition; for the second we are using that
	\[\Big\{\lambda x:x\in\ball_p^{2n}(r)\Big\}=\ball_p^{2n}(|\lambda|_pr).\]
	The width of $\ball^{2n}(1)$ is $1$ by definition, and that of $\cyl^{2n}(1)$ is $1$ by Theorem \ref{thm:analytic-actions-complete}.
\end{proof}

See Figure \ref{fig:padic-capacity} for an illustration of the $(\mathrm{G}_p)^n$-Gromov width.

\begin{figure}
	\includegraphics{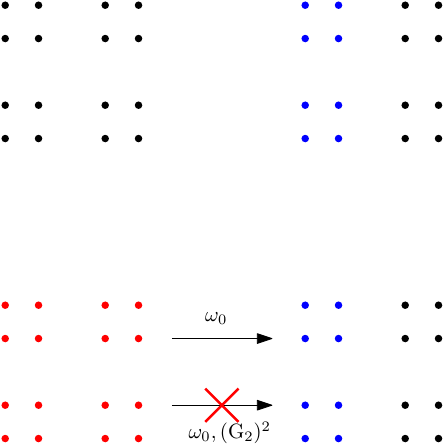}
	\caption{Projections from $(\Q_2)^4$ to $(\Q_2)^2$ of a ball of radius $4$ and a cylinder of radius $2$. It is possible to $p$-adically embed the ball in the cylinder preserving the symplectic form (Corollary \ref{cor:embedding}), but not preserving at the same time the symplectic form and the rotational action (Theorem \ref{thm:analytic-actions-complete}): the ball has width $16$ according to the definition in Theorem \ref{thm:capacity} while the cylinder has width $4$.}
	\label{fig:padic-capacity}
\end{figure}

As far as symplectic capacities are concerned, Theorem \ref{thm:total-actions-nofixed} implies a rather counterintuitive fact (note that almost all points of $(\Qp)^{2n}$ are in $\mathrm{T}_p^{2n}$, and those that are not have only $2n-2$ degrees of freedom):

\begin{corollary}\label{cor:capacity-T}
	Let $n\ge 2$ be an integer. \letpprime. Then
	\[\mathrm{w}^{(\mathrm{G}_p)^n}(\mathrm{T}_p^{2n})=0.\]
\end{corollary}

\begin{proof}
	By Theorem \ref{thm:total-actions-nofixed}, for any $p$-adic absolute value $R$, we can $p$-adically and symplectically embed $\mathrm{T}_p^{2n}$ into the cylinder of radius $R$. Hence its Gromov width is less than $R^2$ for any such $R$, which implies that it is $0$.
\end{proof}

It would be theoretically possible to formulate a definition similar to Definition \ref{def:capacity} but considering the action of $(\mathrm{G}_p)^s$ for $s<n$ instead of the action of $(\mathrm{G}_p)^n$. However, such a capacity does not exist by Theorem \ref{thm:total-actions-semi}:

\begin{corollary}\label{cor:no-capacity}
	Let $n\ge 2$ and $1\le s\le n-1$ be two integers. Let $p$ be a prime number. There does not exist any $p$-adic analytic symplectic $(\mathrm{G}_p)^s$-capacity on $(\Qp)^{2n}$.
\end{corollary}

\begin{proof}
	This has the same proof as Theorem \ref{thm:no-capacity}, but using Theorem \ref{thm:total-actions-semi}, which implies now that $(\Qp)^{2n}$ must have the same capacity than $\cyl_p^{2n}(1)$; the former must be $\infty$ and the latter must be finite.
\end{proof}

\section{Examples}\label{sec:examples}

\begin{example}
	If we construct a map with the same expression as the one in the proof of Theorem \ref{thm:total-embedding} but using the real numbers instead of the $p$-adics, it is discontinuous because all the points in its image have the first two coordinates integer. If we want instead to send $\R^4$ to $[0,1]^2\times\R^2$, we should invert the formulas, and take as $a_0$, $b_0$, $c_0$ and $d_0$ the fractional parts (instead of integer parts) of the coordinates. But this map is discontinuous at every point where those integer parts change. Such a map would look like the one in Figure \ref{fig:embedding}, but with the squares touching, which makes it discontinuous.
\end{example}

\begin{example}
	Preserving the symplectic form when going from the total space to the cylinder is stronger than preserving the volume form. In the real case, the former is impossible by Gromov's non-squeezing theorem, while the latter is possible. For example, if we take polar coordinates at a point of $\R^4$ to be $(r_1,\theta_1,r_2,\theta_2)$, we can define it as
	\[f(r_1,\theta_1,r_2,\theta_2)=\left(\frac{r_1}{r_1+1},\theta_1,r_2(r_1+1)^{\frac{3}{2}},\theta_2\right),\]
	and the volume form \[\dd V=r_1r_2\dd r_1\wedge\dd\theta_1\wedge\dd r_2\wedge\dd\theta_2\] becomes
	\[r_1r_2\sqrt{r_1+1}\frac{\dd r_1}{(r_1+1)^2}\wedge\dd\theta_1\wedge(r_1+1)^{\frac{3}{2}}\dd r_2\wedge\dd\theta_2=\dd V.\]
	In the $p$-adic case, there are many ways to construct a $p$-adic analytic diffeomorphism of $(\Qp)^{2n}$ which preserves the volume form but not the symplectic form, for example scaling the coordinates by constants which multiply to $1$, hence this is also possible from $(\Qp)^{2n}$ to $\cyl_p^{2n}(1)$.
\end{example}

\begin{example}
	By Theorem \ref{thm:rigidity}, the following matrix is squeezing for $p=3$:
	\[A=\begin{pmatrix}
		5 & 4 & 4 & 6 \\
		6 & 5 & 2 & 5 \\
		6 & 2 & 1 & 4 \\
		3 & 2 & 0 & 3
	\end{pmatrix}\]
	Indeed, $\omega_0(\e_1,\e_2)=\omega_0(\e_3,\e_4)=1$, but
	\[|\omega_0(A\tr \e_1,A\tr\e_2)|_3=|\omega_0(A\tr\e_3,A\tr\e_4)|_3=|9|_3=\frac{1}{9}.\]
\end{example}

\begin{example}
	If we apply Corollary \ref{cor:rigidity} to the same matrix, we get
	\[A\Omega_0A\tr=\begin{pmatrix}
		0 & 9 & -4 & 10 \\
		-9 & 0 & -20 & 3 \\
		4 & 20 & 0 & 9 \\
		-10 & -3 & -9 & 0
	\end{pmatrix},\]
	which is not a multiple of $\Omega_0$. On the other hand, if we start with
	\[B=\begin{pmatrix}
		4 & 3 & 2 & 1 \\
		1 & 2 & 3 & 4 \\
		10 & 5 & 0 & 5 \\
		1 & 1 & -1 & -1
	\end{pmatrix}\]
	we have that \[B\Omega_0B\tr=10\Omega_0.\] Hence, the matrix is neither symplectic nor antisymplectic, but Corollary \ref{cor:rigidity} still implies that both $B$ and $B^{-1}$ are non-squeezing. Moreover, they will be non-squeezing for every prime $p$, because $10$ is not divisible by the square of any prime, and its order will always be $0$ or $1$.
\end{example}

\begin{example}
	Two easy examples of $p$-adic linear symplectic width of ellipsoids are shown in Figure \ref{fig:ellipsoids}. They can be given by the matrices
	\[A_1=\begin{pmatrix}
		3 & 0 \\
		1 & -1
	\end{pmatrix},
	A_2=\begin{pmatrix}
		3 & 0 \\
		1 & -3
	\end{pmatrix}.\]
	Applying Theorem \ref{thm:width}, we calculate
	\[A_1\Omega_0A_1\tr=\begin{pmatrix}
		0 & -3 \\
		3 & 0
	\end{pmatrix},
	A_2\Omega_0A_2\tr=\begin{pmatrix}
		0 & -9 \\
		9 & 0
	\end{pmatrix}.\]
	The highest powers of $p^2=9$ which divide all entries of $A_1\Omega_0A_1\tr$ and $A_2\Omega_0A_2\tr$ are $1$ and $9$, respectively.
	 
	For two more complicated examples, consider the $4$-dimensional $2$-adic ellipsoid given by the matrix
	\[C=\begin{pmatrix}
		4 & 3 & 2 & 1 \\
		2 & 4 & 6 & 8 \\
		0 & 2 & 0 & 6 \\
		1 & 1 & 3 & 5
	\end{pmatrix}\]
	and the $8$-dimensional one given by
	\[D=\begin{pmatrix}
		4 & 0 & 0 & 0 & 0 & 0 & 0 & 0 \\
		1 & 4 & 0 & 0 & 0 & 0 & 0 & 0 \\
		1 & 4 & 4 & 0 & 0 & 0 & 0 & 0 \\
		1 & 4 & 1 & 4 & 0 & 0 & 0 & 0 \\
		1 & 4 & 1 & 4 & 4 & 0 & 0 & 0 \\
		1 & 4 & 1 & 4 & 1 & 4 & 0 & 0 \\
		1 & 4 & 1 & 4 & 1 & 4 & 4 & 0 \\
		1 & 4 & 1 & 4 & 1 & 4 & 1 & 4
	\end{pmatrix}\]
	We have that
	\[C\Omega_0C\tr=\begin{pmatrix}
		0 & 20 & 20 & 8 \\
		-20 & 0 & 40 & 4 \\
		-20 & -40 & 0 & -20 \\
		-8 & -4 & 20 & 0
	\end{pmatrix}\]
	and
	\[D\Omega_0D\tr=\begin{pmatrix}
		0 & 16 & 16 & 16 & 16 & 16 & 16 & 16 \\
		-16 & 0 & 0 & 0 & 0 & 0 & 0 & 0 \\
		-16 & 0 & 0 & 16 & 16 & 16 & 16 & 16 \\
		-16 & 0 & -16 & 0 & 0 & 0 & 0 & 0 \\
		-16 & 0 & -16 & 0 & 0 & 16 & 16 & 16 \\
		-16 & 0 & -16 & 0 & -16 & 0 & 0 & 0 \\
		-16 & 0 & -16 & 0 & -16 & 0 & 0 & 16 \\
		-16 & 0 & -16 & 0 & -16 & 0 & -16 & 0
	\end{pmatrix}\]
	The highest power of $p^2=4$ that divides all the entries of $C\Omega_0C\tr$ is $4$ and for $D\Omega_0D\tr$ it is $16$, hence these are the symplectic widths of the ellipsoids.
\end{example}

\begin{example}\label{ex:embedding}
	Consider $p=5$ and $n=2$. A point $m\in(\Qp)^4$ may have as polar coordinates
	\[\left(\frac{194}{25},7,-9,0,1,440,\frac{586}{25},-3,1,4,0,35\right)\]
	\[=(12.34_5,7,-9,0,1,3430_5,43.21_5,-3,1,4,0,120_5),\]
	where the subindex $5$ indicates that the numbers are written in base $5$. The image of this point by the $p$-adic analytic symplectic embedding $f$ described in the proof of Theorem \ref{thm:total-actions-nofixed} is calculated as follows.
	\renewcommand{\theenumi}{\alph{enumi}}
	\renewcommand{\theenumii}{\arabic{enumii}}
	\begin{enumerate}
		\item We see that the point is not in the cylinder because $k_1''$ is negative.
		\item We compute
		\[C=0,c_{-1}=1,c_0=2,c_1=3,c_2=4,D=0,d_{-1}=4,d_0=3,d_1=2,d_2=1,\]
		\[z_1'=0+1=1,\]
		\[u=0+5+1.4233241_5=11.4233241_5.\]
		\item We construct $z_2'$ from $u$ changing the digit at the first even place at the right of $-8$ to $1$:
		\[z_2'=11.4233241001_5.\]
		(Strictly speaking, for $p=5$ we could omit this step, but we go this way for consistency with other primes which require it.)
		\item Since $p\equiv 1\mod 4$, we need to correct $z_2'$:
		\begin{enumerate}
			\item Nothing to do here, $k_2'$ is finite. As $k_1'>0$, we go to step (2).
			\item We change the digit $7+1$ places at the right of the leading digit to $1$:
			\[z_2'=11.42334100100000001_5.\]
			\item[(5)] We repeat with $k_1''$. Since $k_1''<0$, we go to step (3) and change a digit $9$ places at the right of the leading digit:
			\[z_2'=11.42334100100000001000000002_5.\]
		\end{enumerate}
		\item We now set $k_1'=k_1''=0$ and correct $k_2''$ to $\ord(z_2')-k_2'=-26+3=-23$, arriving at our final coordinates:
		\[(1,0,0,0,1,3430_5,11.42334100100000001000000002_5,-3,-23,4,0,120_5).\]
	\end{enumerate}
	If we want to recover the original point, we proceed as follows:
	\begin{enumerate}
		\item We check the digit of order $1$ of $z_2$. It is $1$, which means that the point was not in the cylinder.
		\item We strip the leading digit $2$ from $z_2$. Being $2$ means that $k_1''$ was negative, and the distance from it to the next nonzero digit ($9$) is the opposite of $k_1''$. Hence, $k_1''=-9$.
		\item We strip the next nonzero digit $1$ from $z_2$. Being $1$ means that $k_1'$ was nonnegative, and the distance from it to the next nonzero digit ($8$) is one more than $k_1'$. Hence, $k_1'=7$.
		\item We strip the next nonzero digit $1$ from $z_2$. Being $1$ means that $k_2'$ was finite.
		\item At this point we have recovered $u=11.4233241_5$. If there were any digits at the left of order $1$, those would pass unchanged to $z_2$, and those in $z_1'$ would pass to $z_1$.
		\item Skipping the leftmost $1$, we now assign the rest of digits $14233241$ alternatively to $z_1$ and $z_2$ starting at order $1$ and going right. This yields $z_1=12.34_5$ and $z_2=43.21_5$.
		\item Since in step (d) we deduced that $k_2'$ was finite (which must be anyway because $z_2\ne 0$), we recover $k_2''$ subtracting $k_2'$ from the order of $z_2$, that is, $k_2''=-2+3=1$. $k_2'=-3$ is unchanged and we have recovered the original polar coordinates.
	\end{enumerate}
	The same process would be done for Corollary \ref{cor:analytic-actions}(1), which is just a particular case of Theorem \ref{thm:total-actions-nofixed}. See Figure \ref{fig:example-embedding} for representations of the digit movements we are doing in this case and another three cases for $p=13$.
\end{example}

\begin{figure}
	\[\begin{array}{ccc}
		& \text{Before} & \text{After} \\
		z_1 & \textcolor{red}{12.34}_5 & 1 \\
		k_1' & \colorbox{orange}{7} & 0 \\
		k_1'' & \colorbox{yellow}{$-9$} & 0 \\
		z_2 & \textcolor{blue}{43.21}_5 & 1 \textcolor{red}{1}. \textcolor{blue}{4} \textcolor{red}{2} \textcolor{blue}{3} \textcolor{red}{3} \textcolor{blue}{2} \textcolor{red}{4} \textcolor{blue}{1} 001 \colorbox{orange}{0000000}1 \colorbox{yellow}{000000002}_5 \\
		k_2' & \colorbox{green}{$-3$} & \colorbox{green}{$-3$} \\
		k_2'' & 1 & -23
	\end{array}\]
	\[\begin{array}{ccc}
		& \text{Before} & \text{After} \\
		z_1 & \textcolor{red}{\text{56789A.BC0123}}_{13} & \textcolor{red}{5678}01_{13} \\
		k_1' & \colorbox{orange}{$-4$} & 0 \\
		k_1'' & \colorbox{yellow}{$-2$} & 0 \\
		z_2 & \textcolor{blue}{\text{543210.CBA987}}_{13} & \textcolor{blue}{5432}1 \textcolor{red}{9}. \textcolor{blue}{1} \textcolor{red}{\text{A}} \textcolor{blue}{0} \textcolor{red}{\text{B}} \textcolor{blue}{\text{C}} \textcolor{red}{\text{C}} \textcolor{blue}{\text{B}} \textcolor{red}{0} \textcolor{blue}{\text{A}} \textcolor{red}{1} \textcolor{blue}{9} \textcolor{red}{2} \textcolor{blue}{8} \textcolor{red}{3} \textcolor{blue}{7} 001 \colorbox{orange}{0002} \colorbox{yellow}{02}_{13} \\
		k_2' & \colorbox{green}{$-10$} & \colorbox{green}{$-10$} \\
		k_2'' & 4 & -14
	\end{array}\]
	\[\begin{array}{ccc}
		& \text{Before} & \text{After} \\
		z_1 & \textcolor{red}{\text{00.000000}}_{13} & 1 \\
		k_1' & \colorbox{orange}{$\infty$} & 0 \\
		k_1'' & \colorbox{yellow}{$-2$} & 0 \\
		z_2 & \textcolor{blue}{\text{543210.CBA987}}_{13} & \textcolor{blue}{5432}1 \textcolor{red}{0}. \textcolor{blue}{1} \textcolor{red}{0} \textcolor{blue}{0} \textcolor{red}{0} \textcolor{blue}{\text{C}} \textcolor{red}{0} \textcolor{blue}{\text{B}} \textcolor{red}{0} \textcolor{blue}{\text{A}} \textcolor{red}{0} \textcolor{blue}{9} \textcolor{red}{0} \textcolor{blue}{8} \textcolor{red}{0} \textcolor{blue}{7} 001 \colorbox{orange}{3} \colorbox{yellow}{02}_{13} \\
		k_2' & \colorbox{green}{$-10$} & \colorbox{green}{$-10$} \\
		k_2'' & 4 & -11
	\end{array}\]
	\[\begin{array}{ccc}
		& \text{Before} & \text{After} \\
		z_1 & \textcolor{red}{\text{56789A.BC0123}}_{13} & \textcolor{red}{5678}01_{13} \\
		k_1' & \colorbox{orange}{$-4$} & 0 \\
		k_1'' & \colorbox{yellow}{$-2$} & 0 \\
		z_2 & \textcolor{blue}{\text{00.000000}}_{13} & 1 \textcolor{red}{9}. \textcolor{blue}{0} \textcolor{red}{\text{A}} \textcolor{blue}{0} \textcolor{red}{\text{B}} \textcolor{blue}{0} \textcolor{red}{\text{C}} \textcolor{blue}{0} \textcolor{red}{0} \textcolor{blue}{0} \textcolor{red}{1} \textcolor{blue}{0} \textcolor{red}{2} \textcolor{blue}{0} \textcolor{red}{3} \textcolor{blue}{0} \colorbox{green}{2} \colorbox{orange}{0002} \colorbox{yellow}{02}_{13} \\
		k_2' & \colorbox{green}{$\infty$} & -26 \\
		k_2'' & \colorbox{cyan}{4} & \colorbox{cyan}{4}
	\end{array}\]
	\caption{The construction in Example \ref{ex:embedding} and three similar constructions with $p=13$ (the letters A, B and C correspond to the digits $10$, $11$ and $12$).}
	\label{fig:example-embedding}
\end{figure}

\begin{example}
	The computation in Theorem \ref{thm:total-actions-semi} is the same as in Theorem \ref{thm:total-actions-nofixed}, but only the first two coordinates are polar. For example, we start with $p=7$ and the point
	\[\left(\frac{1266}{49},-1,-1,2,2,14,9,16\right)=(34.56_7,-1,-1,2,2,20_7,12_7,22_7).\]
	We apply the same changes of the previous example, but with $x_2$ instead of $z_2$, and only (a), (b) and (c), because $p$ is not $1\mod 4$. The result is
	\[(1,0,0,2,2,20_7,13.14250601_7,22_7).\]
	Note that $k_1'$ and $k_1''$ have been readjusted to $0$ to match the order of $z_1$. (In this particular case, again, step (c) would not be needed.)
	
	This same process will work for Corollary \ref{cor:analytic-actions}(2), which is a particular case of Theorem \ref{thm:total-actions-semi}.
\end{example}

\begin{example}
	Let $a_i,\ldots,a_n\in\Qp$ and
	\[A=\begin{pmatrix}
		a_1 & & & & & & \\
		& a_1 & & & & & \\
		& & a_2 & & & & \\
		& & & a_2 & & & \\
		& & & & \ddots & & \\
		& & & & & a_n & \\
		& & & & & & a_n
	\end{pmatrix}.\]
	Let $E$ be the ellipsoid given by
	\[\|Av\|_p\le 1.\]
	This ellipsoid is invariant by the action of $(\mathrm{G}_p)^n$ because each $\mathrm{G}_p$ (given in Definition \ref{def:G}) changes two coordinates with the same coefficient, and $p$-adic balls are invariant by $\mathrm{G}_p$. The $(\mathrm{G}_p)^n$-capacity of $E$ is
	\[\frac{1}{\max\Big\{|a_i|_p^2:1\le i\le n\Big\}}.\]
	In order to prove this, let $j$ be the index which attains the maximum and consider the $p$-adic ball
	\[B=\Big\{(x_1,y_1,\ldots,x_n,y_n)\in(\Qp)^{2n}:|a_jx_i|_p\le 1,|a_jy_i|_p\le 1\text{ for every }i\in\{1,\ldots,n\}\Big\}\]
	and the $p$-adic symplectic cylinder
	\[Z=\Big\{(x_1,y_1,\ldots,x_n,y_n)\in(\Qp)^{2n}:|a_jx_j|_p\le 1,|a_jy_j|_p\le 1\Big\}.\]
	Both $B$ and $Z$ have radius $1/|a_j|_p$ and $B\subset E\subset Z$, hence by Theorem \ref{thm:analytic-actions-complete}
	\[\frac{1}{|a_j|_p^2}=c(B)\le c(E)\le c(Z)=\frac{1}{|a_j|_p^2},\]
	as we wanted to prove.
\end{example}

\section{Final remarks}\label{sec:remarks}

\subsection{Non-squeezing in physics and de Gosson's work}

As explained in \cite[page 458]{McDSal}, Weinstein pointed out that Gromov's non-squeezing theorem can be considered a geometric expression of the uncertainty principle. Gromov's non-squeezing theorem has important implications in physics, see de Gosson's articles \cite{Gosson-quantum,Gosson-egg,GosLuef}.

In the $p$-adic case the fact that there is linear non-squeezing (Theorem \ref{thm:linear}) should also have a physical interpretation along similar lines. Similarly the non-squeezing result Theorem \ref{thm:analytic-actions-complete} should have such an interpretation. We do not yet know how to interpret the fact that in general there is no symplectic non-squeezing in the $p$-adic category (Theorems \ref{thm:total-embedding}, \ref{thm:total-actions-nofixed}, \ref{thm:total-actions-semi}).

\subsection{Real versus $p$-adic symplectic topology}

Since non-squeezing is one of the foundational results of modern symplectic topology, the $p$-adic Theorems \ref{thm:total-embedding}, \ref{thm:total-actions-nofixed} and \ref{thm:total-actions-semi} indicate that $p$-adic symplectic topology, in the non-linear case, is going to have a different flavor than real symplectic topology. However, in the affine case the situation is analogous, as seen in Theorem \ref{thm:linear}. This is because the linear case is a purely algebraic problem and the nonlinear case is more of a topological problem. We could say that, as far as symplectic squeezing is concerned, the linear approximations of both real and $p$-adic symplectic topology are analogous.

\subsection{A $p$-adic symplectic camel theorem?}

In the real case, a consequence of Gromov's non-squeezing theorem known as the symplectic camel theorem says that it is not possible to move a $2n$-dimensional ball through a hole in a hyperplane with less radius than the ball.

There is no direct way to generalize this to the $p$-adic case. To begin with, the statement as such makes no sense: though the concept of a hyperplane can be defined in a $p$-adic vector space, the word ``through'' implies the existence of two sides of the hyperplane, which is possible in the topology of $\R$ but not in that of $\Qp$. Also, even if we define arbitrarily the two sides, for example saying that the value of a linear functional which is zero in the hyperplane is square at one side and non-square at the other side, it is possible to move continuously from one to the other side without touching the hyperplane, in the $p$-adic sense of continuity.

\subsection{$p$-adic analytic symplectic embeddings of products}

In the real case, there exist symplectic embeddings \[\ball^2(1)\times\R^{2n-2}\hookrightarrow\ball^{2n-2}(R)\times \R^2\] if $R\ge \sqrt{2^{n-1}+2^{n-2}-2}$, by the work of Guth \cite{Guth} and Pelayo-V\~u Ng\d oc \cite{PelVuN}, or more generally, there exist symplectic embeddings \[\ball^2(r)\times\R^{2n-2}\hookrightarrow\ball^{2n-2}(R)\times \R^2\] if $R\ge r\sqrt{2^{n-1}+2^{n-2}-2}$. In the $p$-adic case, as expected, the analogous symplectic embeddings exist for any radius.

\begin{proposition}\label{prop:embeddings}
	\letnpos\ with $n\ge 2$. \letpprime. Then, for any $p$-adic absolute value $R$, there exists a $p$-adic analytic symplectomorphism
	\[f:(\Qp)^{2n}\overset{\cong}{\longrightarrow}\ball_p^{2n-2}(R)\times (\Qp)^2,\]
	where both the domain and codomain of $f$ are endowed with the standard $p$-adic symplectic form $\sum_{i=1}^n\dd x_i\wedge\dd y_i$, with $(x_1,y_1,\ldots,x_n,y_n)$ being the standard coordinates on $(\Qp)^{2n}$.
\end{proposition}

\begin{proof}
	It is enough to prove it for $R=1$, and the rest follows by scaling. We apply Theorem \ref{thm:total-embedding} $n-1$ times and in this way obtain a $p$-adic analytic symplectomorphism
	\[(\Qp)^{2n}\overset{\cong}{\to}\ball_p^2(1)\times(\Qp)^{2n-2}\overset{\cong}{\to}\ldots\overset{\cong}{\to}\ball_p^{2n-2}(1)\times(\Qp)^2.\qedhere\]
\end{proof}

\section{Appendix: the $p$-adic numbers}\label{sec:appendix}

We recall here some definitions and results about the $p$-adic numbers and $p$-adic symplectic geometry which we need in the paper.

The field $\R$ of the real numbers is defined as a completion of the field $\Q$ with the standard absolute value given by
\[|x|=\max\{x,-x\}.\]
Analogously, $\Qp$ is defined as the completion of $\Q$ with the $p$-adic absolute value.

\begin{definition}
	\letpprime. The \emph{$p$-adic order} (or valuation) is defined in $\Q$ as
	\[\ord_p(n)=\max\{k\in\N:p^k\mid n\}\]
	for $n\in\Z$, and
	\[\ord_p\left(\frac{m}{n}\right)=\ord_p(m)-\ord_p(n)\]
	for any $m,n\in\Z$.
	
	The \emph{$p$-adic absolute value} is defined in $\Q$ as
	\[|x|_p=p^{-\ord_p(x)}\]
	for $x\in\Q$.
	
	The field of \emph{$p$-adic numbers} $\Qp$ is defined as the completion of $\Q$ with respect to the $p$-adic absolute value.
\end{definition}

See Figure \ref{fig:7adics} for a drawing of the $7$-adic numbers.

\begin{figure}
	\includegraphics{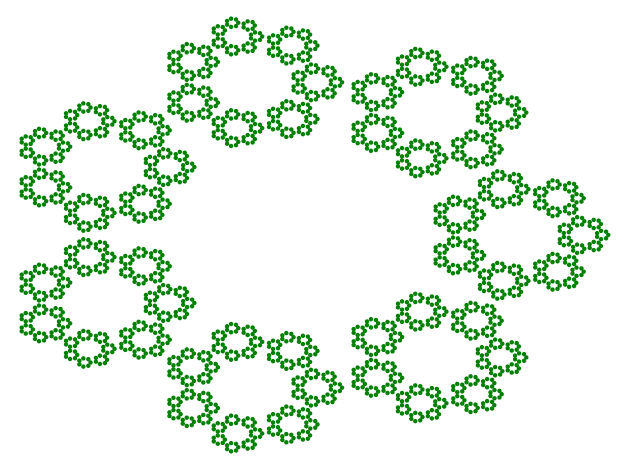}
	\caption{The $7$-adic numbers. If each dot represents a ball of radius $1$, all the dots together represent a ball of radius $7^4$.}
	\label{fig:7adics}
\end{figure}

\begin{proposition}\label{prop:expansion}
	\letpprime. Given $x\in\Qp$, there exist $k\in\Z$ and $a_i\in\{0,\ldots,p-1\}$ for each $i\in\Z$ with $i\ge k$ such that
	\[\sum_{i=k}^{\infty}a_ip^i=x,\]
	meaning that the infinite sum in the left converges to $x$ respect to the $p$-adic distance. These $a_i$ will be called \emph{digits} of $x$, and $k$ is the \emph{order} of $x$, denoted by $\ord_p(x)$; this order satisfies that $|x|_p=p^{-\ord_p(x)}$, and as such it generalizes the order in the previous definition.
\end{proposition}

\begin{definition}
	\letpprime. The \emph{$p$-adic integers} $\Zp$ are defined as the $p$-adic numbers having nonnegative order.
\end{definition}

The $p$-adic absolute value gives rise to a norm in $(\Qp)^n$:

\begin{definition}
	\letnpos. \letpprime. Given $v=(v_1,\ldots,v_n)\in(\Qp)^n$, we define the \emph{$p$-adic norm} of $v$ as
	\[\|v\|_p=\max\{|v_i|_p:i\in\{1,\ldots,n\}\}.\]
\end{definition}

\begin{definition}
	\letnpos. \letpprime. Given an open set $U\subset(\Qp)^n$, an \emph{analytic function} $f:U\to\Qp$ is given by a collection of power series $\{f_i:i\in I\}$ for some set $I$, where $f_i:U_i\to\Qp$ for each $i\in I$, each $U_i$ is an open subset of $(\Qp)^n$, the union of the sets $U_i$ is $U$, and $f_i$ converges in $U_i$ for all $i\in I$.
\end{definition}

The concepts of differential form, the differential operator and the wedge operator are direct extensions of the real case. See \cite[Appendix B]{CrePel-JC} for precise definitions.

We recommend the books \cite{Gouvea,Koblitz,Schneider} for introductions to the $p$-adic numbers and their use in geometry or analysis, and the works \cite{DjoDra,Dragovich-quantum,Dragovich-harmonic,DKKV} for treatments of the $p$-adic numbers in the context of geometry and physics.

\bigskip

\textbf{Acknowledgments.} The first author is funded by grant PID2022-137283NB-C21 of MCIN/AEI/10.13039/501100011033 / FEDER, UE and by project CLaPPo (21.SI03.64658) of Universidad de Cantabria and Banco Santander.

The second author is funded by a FBBVA (Bank Bilbao Vizcaya Argentaria Foundation) Grant for Scientific Research Projects with title \textit{From Integrability to Randomness in Symplectic and Quantum Geometry}.

The second author thanks the Dean of the School of Mathematical Sciences Antonio Br\'u and the Chair of the Department of Algebra, Geometry and Topology at the Complutense University of Madrid, Rutwig Campoamor, for their support and excellent resources he is being provided with to carry out the FBBVA project.

\end{document}